\documentclass{article}
\usepackage[english]{babel}
\usepackage{amsmath,amssymb,stmaryrd,mathrsfs,bbm,xcolor,latexsym,theorem}
\usepackage[left=2.5cm, right=2.5cm]{geometry}
\usepackage{times}
\usepackage{hyperref}
\catcode`\<=\active \def<{
\fontencoding{T1}\selectfont\symbol{60}\fontencoding{\encodingdefault}}
\newcommand{\Iota}{\mathrm{I}}
\newcommand{\assign}{:=}
\newcommand{\backassign}{=:}
\newcommand{\equallim}{\mathop{=}\limits}
\newcommand{\mathD}{\mathrm{D}}
\newcommand{\mathd}{\mathrm{d}}
\newcommand{\nobracket}{}
\newcommand{\tmaffiliation}[1]{\\ #1}
\newcommand{\tmcolor}[2]{{\color{#1}{#2}}}
\newcommand{\tmdummy}{$\mbox{}$}
\newcommand{\tmem}[1]{{\em #1\/}}
\newcommand{\tmemail}[1]{\\ \textit{Email:} \texttt{#1}}
\newcommand{\tmop}[1]{\ensuremath{\operatorname{#1}}}
\newcommand{\tmscript}[1]{\text{\scriptsize{$#1$}}}
\newcommand{\tmstrong}[1]{\textbf{#1}}
\newenvironment{proof}{\noindent\textbf{Proof\ }}{\hspace*{\fill}$\Box$\medskip}
\newenvironment{proof*}[1]{\noindent\textbf{#1\ }}{\hspace*{\fill}$\Box$\medskip}
\newcounter{nnacknowledgments}

{\theorembodyfont{\rmfamily}\newtheorem{acknowledgments*}[nnacknowledgments]{Acknowledgments}}
\newtheorem{corollary}{Corollary}
\newtheorem{definition}{Definition}
\newtheorem{lemma}{Lemma}
\newtheorem{notation}{Notation}
\newtheorem{proposition}{Proposition}
{\theorembodyfont{\rmfamily}\newtheorem{remark}{Remark}}
\newtheorem{theorem}{Theorem}

\newcommand{\VV}{\mathscr{C}}
\newcommand{\CF}{\mathscr{F}}

\newcommand{\CS}{\mathscr{S}}

\begin{document}

\title{A stochastic control approach to Sine-Gordon EQFT}

\author{
  Nikolay Barashkov
  \tmaffiliation{Department of Mathematics and Statistics\\
  University of Helsinki}
  \tmemail{nikolay.barashkov@helsinki.fi}
}

\maketitle

  \begin{center}
    Dedicated to the memory of Mikhail Barashkov
  \end{center}

\begin{abstract}
  We study the Sine-Gordon model for $\beta^{2}< 4 \pi$ in infinite volume. We give a variatonal characterization of it's laplace transform, and deduce from this large deviations. Along the way we obtain estimates which are strong enough to obtain a proof of the Osterwalder-Schrader axioms including exponential decay of correlations as a byproduct. Our method is based on the Boue-Dupuis formula with an emphasis on the stochastic control structure of the problem.
  \end{abstract}
\section{Introduction}

In this article we investigate the Sine Gordon measure on the plane for
$\beta^2 < 4 \pi$, that is
\begin{equation}
  \nu_{\tmop{SG}} = \frac{1}{Z} \exp \left( - \lambda \int_{\mathbb{R}^2} \cos
  (\beta \varphi) \mathd x \right) \mathd \mu \label{eq:formal}  \qquad Z =
  \int \exp \left( - \lambda \int_{\mathbb{R}^2} \cos (\beta \varphi) \mathd x
  \right) \mathd \mu
\end{equation}
where $\mu$ is the Gaussian Free Field on $\mathbb{R}^2$, that is the Gaussian
measure {\cite{janson_gaussian_1997}} on $\CS' (\mathbb{R}^2)$ (tempered
distributions) with covariance
\[ \int_{\CS' (\mathbb{R}^2)} \langle f, \varphi \rangle_{L^2 (\mathbb{R}^2)}
   \langle g, \varphi \rangle_{L^2 (\mathbb{R}^2)} \mathd \mu = \langle f,
   (m^2 - \Delta)^{- 1} g \rangle_{L^2 (\mathbb{R}^2)} ._{} \]
Measures of this form are of interest in mathematical physics since they allow
the construction of relativistic Quantum Field Theories (QFTs)
{\cite{Kazhdan-1999,Strocchi-2013}} via the {\tmem{Osterwalder-Schrader}}
Reconstruction Theorem, provided one is able to prove that they satisfy the
Osterwalder-Schrader Axioms {\cite{Osterwalder-1973}}.

Expression {\eqref{eq:formal}} has no rigorous meaning since $\mu$ is known
to be supported on genuine distributions so $\cos (\beta \varphi)$ is ill
defined. Nevertheless if $\varphi$ is sampled from $\mu$ one can give
rigourous meaning to the {\tmem{Wick ordered}} cosine $\llbracket \cos (\beta
\varphi) \rrbracket$ as a random distribution {\cite{Junnila-2020}}. It is
important to note that the process of Wick ordering ``multiplies the cosine by
an infinite constant'' and thus the resulting potential becomes nonconvex even
if $\lambda \ll m$. In finite volume this is suffient to define a density with
respect to free field, for the sine gordon measure provided $\beta^2 < 4 \pi$.
For $4 \pi \leqslant \beta^2 < 8 \pi$ the construction becomes more
complicated and the partion function $Z$ aquires further divergencies. A
further complication arises since samples of $\mu$ are not expected to have
any decay at infinity so the expression $\int_{\mathbb{R}^2} \llbracket \cos
(\beta \varphi) \rrbracket \mathd x$ loses it's meaning even though
$\llbracket \cos (\beta \varphi) \rrbracket$ is a well defined random
distribution.

Despite the difficulties the Sine-Gordon measures has been studied extensivly
in the mathematical literature. In {\cite{Frohlich1976}} the measure was
contructed in the full space for $\beta < 4 \pi$ and $\lambda \ll m$ and the
authors were able verify the {\tmem{Osterwalder-Schrader axioms}}, exponential
decay of the correlation funcitons as well as nontriviality of the scattering
matrix of the corresponding Quantum Field Theory. In the full range $0 <
\beta^2 < 8 \pi$ the measure was constructed in
{\cite{dimock1993,lacoin_sine-gordon_2020,Dimock2000}} . A markov property for
the Sine-Gordon model was shown in {\cite{albeverio-hk-sine-gordon-1979}}. The
markov property allows one to employ an alternative route to constuct the
corresponding QFT via Nelson's reconstruction
{\cite{Nelson-1966,Nelson-1973}}. In {\cite{brydges_mayer_1987}} Brydges and
Kennedy gave an elegant construction of the Sine-Gordon measure using the
Polchinski equation {\cite{polchinski_renormalization_1984}}. Their method was
later explored in {\cite{bauerschmidt_2019}} to prove a Logarathmic Sobolev
Inequality for $\beta^2 < 6 \pi$ and in
{\cite{bauerschmidt_2020,hofstetter-2021}} to construct a coupling between the
Sine-Gordon measure and the Free Field, and subsequently study the maximum of
the Sine-Gordon field for $\beta^2 < 6 \pi$. In {\cite{oh-Sine-Gordon-2020}}
the authors showed that the Sine gordon measure is invariant under a the flow
a wave quation with Sinus nonlinearity (in finite volume). For $m = 0$ the
Sine-Gordon model also exhibits the Coleman Correspondecne, a relation with
the fermionic Thirring model
{\cite{Frohlich1976,bauerschmidt_fermion_2020,Benfatto-2008}}. \

To the Sine-Gordon measure one can associate Langevin dynamics, that is the
stochastic PDE \
\[ (\partial_t + (m^2 - \Delta)) u - \lambda \beta \llbracket \sin (\beta
   \varphi) \rrbracket = \xi (t, x) \qquad \xi \text{Space time white noise} .
\]
These have been studied in {\cite{Albeverio2001,Hairer-Sine-Gordon-2016}}.
These Langevin dynamics have been employed to study the measure in the context
of the $\Phi_3^4$ model using techiques from stochastic PDE's, see
{\cite{albeverio_invariant_2017,gubinelli_pde_2018}}. This approach is known
as Stochastic Quantizaiton (SQ). An elliptic version of the SQ equations have
also recently been developed {\cite{Albeverio-2019,Albeverio-2020}}.

In {\cite{barashkov_gubinelli_variational}} a variational approach to study
EQFTs was proposed, and the $\Phi_3^4$ model in finite volume was studied.
Here we aim to pursue this point of view in infinite volume. The variational
approach is an intepretation of the Polchinski equation in terms of stochastic
control, and thus most closely resembles {\cite{Brydges_1987}}. More precisely
it studies the stochastic control problem whose associated HJB equation is the
Polchinski equation, see also Appendix \ref{app:stoch-cont}. On the other
hand, as was seen in {\cite{barashkov_gubinelli_variational}}, the variational
method also allows to deploy techniques developed in the study of stochastic
quantization. Furthermore it allows to derive a variational representation of
the measure, without making refrence to a limiting process, even if the
measure is not absolutely continuous with respect to the GFF. Let us mention
that the variational approach has also been used to study phase transitons for
the $\Phi_3^4$ model {\cite{Chandra-2020}}. \

In {\cite{barashkov-2021}} M.Gubinelli and the author studied the $\Phi_2^4$
and Hoergh-Krohn (with exponential interaction) models in infinite volume
using the variational method. Through studying the Euler-Lagrange equation of
the corresponding variational problem it is possible to derive a variational
formula for the Hoergh-Krohn model in infinite volume. This required convexity
of the renormalized interaction and thus we were not able this result for the
$\Phi_2^4$ model although we obtained partial results in this direction. In
the present paper we will study this problem for the Sine-Gordon model in the
case $\beta^2 < 4 \pi$. Although in this case the renormalized interaction is
not convex we will be able to circumvent this using the boundedness of the
sinus.

The starting point of our analysis will be the Boue-Dupuis representation
{\cite{boue_variational_1998}} of the Laplace transform of the approximate
Sine-Gordon measure. See also the papers of {\"U}st{\"u}nel
{\cite{ustunel_variational_2014}} and Zhang~{\cite{zhang_variational_2009}}
for extensions to the infinite dimensional setting. where extensions and
further results on the variational formula are obtained The Boue-Dupuis
formula has been used to derive Large Deviation Principles (LDPs) in different
contexts {\cite{Budhiraja-2019}}. A byproduct of our approach will be to
derive a large deviations principle for the Sine-Gordon model in the
semiclassical limit. A similar result has been obtained for the Lioville
measure in {\cite{lacoin_semiclassical_2020,Lacoin_2017}} and in
{\cite{barashkov-2021}} for $\Phi_2^4$. \

\subsection{Results }

We will now give an overview of our results and a rough outline of the proofs.
We want to study the laplace transform, for some sufficiently nice functional
$f$ (see below for details)
\[ - \log \int \exp \left( - f (\varphi) - \lambda \int_{\mathbb{R}^2} \rho
   \llbracket \cos (\beta \varphi) \rrbracket \mathd x \right) \mathd \mu =
   W^{\rho} (f) . \]
Here we have introduced an ''infrared cutoff'' to deal with the diverence due
to infinite volume. It is chosen according to the following definition:

\begin{definition}
  We denote by
  \[ C = \{ \rho \in C_c^{\infty} (\mathbb{R}^2) : 0 \leqslant \rho \leqslant
     1, | \nabla \rho | \leqslant 1 \} \]
  the space of spacial cutoffs. We will say that family of spacial cutoffs
  $\rho^N \in C$ converges to $1$ or $\rho^N \rightarrow 1$ if for any $K > 0$
  there exists $N_0 \in \mathbb{N}$ such that $\rho^N (x) = 1$ for any $x \in
  B (0, K)$ and $N \geqslant N_0$. Usually we will drop the index $N$ and
  simply write $\rho \rightarrow 1$ in this case. 
\end{definition}

We will first study the measures
\[ \mathd \nu^{\rho}_{\tmop{SG}} = \frac{1}{Z^{\rho}} \exp \left( - \lambda
   \int_{\mathbb{R}^2} \rho \llbracket \cos (\beta \varphi) \rrbracket \mathd
   x \right) \mathd \mu \qquad Z^{\rho} = \exp (W^{\rho} (0)) . \]
The BD formula gives
\[ \begin{array}{lll}
     &  & W^{\rho} (f)\\
     & = & \inf F^{f, \rho} (u)\\
     & = & \inf_{u \in \mathbb{H}_a} \mathbb{E} \left[ f (W_{\infty} +
     I_{\infty} (u)) + \lambda \int \rho \llbracket \cos (\beta (W_{\infty} +
     I_{\infty} (u)) \rrbracket \nobracket + \frac{1}{2} \int^{\infty}_0 \|
     u_s \|^2_{L^2} \mathd s \right]
   \end{array} \label{eq:BD-intro} \]
Here
\begin{itemize}
  \item $W_t$ is a Gaussian process, such that $\tmop{Law} (W_{\infty}) = \mu$
  and $W_t$ is smooth almost surely for $t < \infty$, see Section
  \ref{sec:decomposition} for details.
  
  \item $I : L^2 (\mathbb{R}_+ \times \mathbb{R}^2) \rightarrow L^{\infty}
  (\mathbb{R}_+, H^1 (\mathbb{R}^2))$ is a linear map given by
  \[ I_t (u) = \int^t_0 J_s u_s \mathd s \qquad J_t = \left( \frac{1}{t^2}
     e^{- (m^2 - \Delta) / t} \right)^{1 / 2} \]
  also see Section \ref{sec:decomposition} for details.
  
  \item $\mathbb{H}_a$ is the space of processes adapted to $W_t$ such that
  \[ \mathbb{E} \left[ \int^{\infty}_0 \| u_s \|^2_{L^2} \mathd s \right] <
     \infty \]
\end{itemize}
The minimization problem on r.h.s of {\eqref{eq:BD-intro}} is a stochastic
control problem. There is two ways to study the minimier: One can ignore the
control structure and study the Euler Lagrange equations for the minizer as
was done in {\cite{barashkov-2021}}, which bears strong resemblence with the
study of Stochastic Quantization. Alternativly one can analyse the value
function of the control problem and use the verification theorem from
stochastic control theory (Proposition \ref{prop:verification} below) to infer
properties of the minimizer. We will use both points of view here. First we
will use the control point of view to derive an $L^{\infty}$ bound on the
optimizer:

\begin{theorem}
  \label{thm:L-infty}The minimizer $\bar{u}^{\rho}$ in $\mathbb{H}_a$ of the
  functional
  \[ F^{\rho} (u) =\mathbb{E} \left[ \lambda \int \rho \llbracket \cos (\beta
     (W_{\infty} + I_{\infty} (u))) \rrbracket \mathd x + \frac{1}{2}
     \int^{\infty}_0 \| u_t \|^2_{L^2} \mathd t \right] \]
  satisfies
  \[ \| \bar{u}^{\rho}_t \|_{L^{\infty}} \leqslant C \langle t
     \rangle^{\beta^2 / 8 \pi - 1} . \]
\end{theorem}

Once this is established we will want to study the dependence of the optimizer
on a local perturbation $f.$ We will show that, provided $f$ is sufficiently
nice, with $w (x) = \exp (\gamma | x |)$ and $u^{f, \rho}$ being the
miminimizer for $u^{f, \rho}$ is \
\begin{equation}
  \mathbb{E} \left[ \int^{\infty}_0 \int w | u_t^{f, \rho} - \bar{u}_t^{\rho}
  |^2 \mathd x \mathd t \right] < \infty \label{eq:decay-intro}
\end{equation}
Where we will establish {\eqref{eq:decay-intro}} studying the Euler-Lagrange
equations for $F^{f, \rho}$. For this convexity of $F^{\rho}$ is crucial.
Observe that the second term $\frac{1}{2} \int^{\infty}_0 \| u_t \|^2_{L^2}
\mathd t$ is strictly convex, but the first term seemingly breaks convexity
since the cosine is multiplied with an infinite constant. However applying
Ito's formula one can calculate
\begin{eqnarray*}
  &  & \lambda \int \rho \llbracket \cos (\beta (W_{\infty} + I_{\infty}
  (u))) \rrbracket \mathd x\\
  & = & \lambda \beta \int \rho \llbracket \sin (\beta (W_t + I_t (u)))
  \rrbracket J_t u \mathd x + \text{martingale}\\
  & = & \lambda \beta \int J_t (\rho \alpha (t) (\sin (\beta (W_t + I_t (u)))
  \nobracket) u_t \mathd x + \text{martingale}
\end{eqnarray*}
Below we will show that $\alpha (t) \lesssim \langle t \rangle^{\beta^2 / 8
\pi}$ and $\| J_t u \|_{L^p} \lesssim t^{- 1} \| u \|_{L^p}$ so one can hope
for the functioal to be convex provided that $\lambda$ is small enough and $u$
stays in a bounded region, which is guaranteed by Theorem \ref{thm:L-infty}.
Expaning this heuristic we will prove {\eqref{eq:decay-intro}}, which in turn
will allow us to remove the cutoff $\rho$ in the variational formula and prove
the following theorem:

\begin{theorem}
  \label{thm:characterization-1}Define $\mathbb{D}^f$ to be the space
  \[ \mathbb{D}^f = \left\{ v \in \mathbb{H}_a : \mathbb{E} \left[
     \int^{\infty}_0 \int \| w v \|^2_{L^2 (\mathbb{R}^2)} \mathd x \mathd t
     \right] \leqslant C \right\} . \]
  Then there exists a $\lambda_0$ such that for any $- \lambda_0 < \lambda <
  \lambda_0$\quad
  \[ \lim_{\rho \rightarrow 1} (W^{\rho} (f) - W^{\rho} (0)) = \inf_{v \in
     \mathbb{D}^{f}} \bar{F}^f (v) \]
  where
  \begin{align*} \bar{F}^f (v) = &\mathbb{E} \left[ f (W_{\infty} + I_{\infty} (\bar{u}) +
      I_{\infty} (v)) \right.
      \\ &+ \lambda \int \left( \llbracket \cos (\beta (W_{\infty}
     + I_{\infty} (\bar{u}) + I_{\infty} (v))) \rrbracket - \llbracket \cos
     (\beta (W_{\infty} + I_{\infty} (\bar{u}))) \rrbracket \right) \mathd x \\ & \left. +
    \int^{\infty}_0 \int \bar{u}_t v_t \mathd x \mathd t + \frac{1}{2}
    \int^{\infty}_0 \| v_t \|^2_{L^2} \mathd t \right]
    \end{align*}
  and \={u} is the limit of $\bar{u}^{\rho}$ from Theorem \ref{thm:L-infty}
  which will be shown to exist in Proposition
  \ref{prop:drift-infinite-volume}.
\end{theorem}

It turs out that the minimizer of eq {\eqref{eq:BD-intro}} for $f = 0$
provides a coupling of the Sine-Gordon measure with the Gaussian Free Field,
so from Theorem \ref{thm:L-infty} we will be able to deduce the following:

\begin{theorem}
  \label{thm:coupling}{\tmdummy}
  
  \begin{enumerate}
    \item There exists random variables $I^{\rho} \in L^{\infty} (\mathbb{P}
    \times \mathbb{R}^2)$ and
    \[ \sup_{\rho \in C} \| I^{\rho} \|_{L^{\infty} (\mathbb{P} \times
       \mathbb{R}^2)} < \infty \]
    such that the measures $\nu^{\rho}_{\tmop{SG}}$ satisfy
    \[ \nu^{\rho}_{\tmop{SG}} = \tmop{Law} (W_{\infty} + I^{\rho}) . \]
    \item If $\lambda$ is small enough we have that there exists an $I \in
    L^{\infty} (\mathbb{P} \times \mathbb{R}^2)$ such that for any $\gamma >
    0$
    \[ \| I^{\rho} - I \|_{L^{\infty} (\mathbb{P}, L^{2, - \gamma})}
       \rightarrow 0, \]
    and the Law of $(W, I)$ is euclidean invariant. \ \ 
  \end{enumerate}
\end{theorem}

\begin{proof}
  Theorem \ref{thm:coupling} follows from Lemma \ref{lemma:coupling}, Theorem
  \ref{thm:L-infty} and Proposition \ref{prop:drift-infinite-volume}, note
  that Euclidean invariance follows easily from the uniquenss of $I$. 
\end{proof}

An analogue of Theorem \ref{thm:coupling} has been established in finite volume in {\cite{bauerschmidt_2020}} even in the larger range $\beta^{2}<6 \pi$. We expect their techniques to extend to the infinite volume case at least for weak coupling (when $\lambda$ is small enough).

Both Theorem \ref{thm:characterization} and Theorem \ref{thm:coupling} imply
that if $\lambda$ is small enough $\nu^{\rho}_{\tmop{SG}}$ converges to a
unique measure $\nu_{\tmop{SG}}$ as $\rho \rightarrow 1$. A byproduct of
method is a proof for the Osterwalder-Schrader axioms for $\nu_{\tmop{SG}}$,
see {\cite{gubinelli_pde_2018}} for a nice exposition of the OS axioms. Note
also that from Theorem \ref{thm:coupling} we can deduce tighness of
$\nu^{\rho}_{\tmop{SG}}$ even if $\lambda$ is large.

\begin{theorem}
  \label{thm:OS}$\nu_{\tmop{SG}}$ satisfies the Osterwalder-Schrader axioms.
  Furthermore the correlations decay exponentially and $\nu_{\tmop{SG}}$ is
  non-Gaussian. 
\end{theorem}

\begin{proof}
  Euclidean invariance and follows directly from Theorem \ref{thm:coupling},
  observe that analyticity also follows since Theorem \ref{thm:coupling}
  implies that $\nu_{\tmop{SG}}$ has gaussian tails. The remaining axioms are
  shown in Section \ref{sec:OS} .
\end{proof}

Finally will discuss large deviations for $\nu_{\tmop{SG}}$ in a
semi-classical limit similarly to \
{\cite{lacoin_semiclassical_2020,Lacoin_2017,barashkov-2021}}. For this we
have to introduce Planck's constant into the measure. We want to look at the
measure formally given by
\[ \mathd \nu_{\tmop{SG}, \hbar} = \frac{1}{Z_{\hbar}} e^{- \frac{1}{\hbar}
   \int \lambda \cos (\beta \varphi (x)) + \frac{1}{2} m^2 \varphi (x)^2 +
   \frac{1}{2} | \nabla \varphi (x) |^2 \mathd x} \mathd \varphi . \]
This can be rewritten as
\[ \nu_{\tmop{SG}, \hbar} (\mathd \varphi) = \frac{1}{Z_{\hbar}} e^{-
   \frac{\lambda}{\hbar} \int \cos (\hbar^{1 / 2} \beta \varphi)} \mu (\mathd
   \varphi) . \]
where $Z_{\hbar}$ is normalization constant and we are interested in the limit
$\hbar \rightarrow 0$. These measure can be made sense of in the same way as
$\nu_{\tmop{SG}}$. We will prove

\begin{theorem}
  \label{thm:LD}$\nu_{\tmop{SG}, \hbar}$ satisfies a large deviations
  principle with rate function
  \[ I (\varphi) = \lambda \int (\cos (\beta \varphi (x)) - 1) \mathd x +
     \frac{1}{2} m^2 \int \varphi^2 (x) \mathd x + \frac{1}{2} \int | \nabla
     \varphi (x) |^2 \mathd x. \]
  or equivalently
  \[ \lim_{\hbar \rightarrow 0} - \hbar \log \int e^{- \frac{1}{\hbar} f
     (\varphi)} \mathd \nu_{\tmop{SG}, \hbar} = \inf_{\varphi \in H^{- 1}
     (\langle x \rangle^{- n})} \{ f (\varphi) + I (\varphi) \} . \]
\end{theorem}

Let us remark that for singular SPDE's large deviations have been studied in
{\cite{hairer-2015-largedev}} and the more precise study of Laplace
asymptotics has been carried out in {\cite{friz-2021,klose-2022}}. It would be
interesting to investigate Laplace asymptotics in our context.

{\tmstrong{Structure of the article.}} In Section \ref{sec:prelim} we collect
some useful estimates, introduce some useful processes and collect some facts
on stochastic control. Details and proofs will be given in the corresponding
appendices. In Section \ref{sec:proof} we will prove Theorem
\ref{thm:L-infty}. In Section \ref{sec:char} we will prove Theorem
\ref{thm:characterization-1}. In Section \ref{sec:LD} we will show Theorem
\ref{thm:LD}. Finally in Section \ref{sec:OS} we will establish Theorem
\ref{thm:OS}. The appendices contain supplementary material on stochastic
control and technical estimates.

\begin{acknowledgments*}
  I thank M.Gubinelli for his support during my Ph.D and his helpful comments
  on this article which included a significant simplifiaction of the argument.
  I thank S.Albeverio F. De Vecchi, T.Gunaratnam and I. Zachhuber for their
  comments on a preliminary version of this article. I thank Antti Kupianen
  for some intersting discussions on the Sine-Gordon model. This work was
  supported by the German Research Foundation (DFG) through CRC 1060 - project
  number 211504053 and by the Eurpoean Reseach Council through ERC Advanced
  Grant 741487 ``Quantum Fields and Probability''.
  This paper was written with Texmacs (www.texmacs.org).
\end{acknowledgments*}

\section{Preliminaries}\label{sec:prelim}

In this section we collect some useful notions, estimates and facts which will
be used in the rest of the paper.

\subsection{Weighted spaces }

First recall the definition of Littlewood--Paley blocks. Let $\chi, \varrho$
be smooth radial functions $\mathbb{R}^2 \rightarrow \mathbb{R}$ such that
\begin{itemize}
  \item $\tmop{supp} \chi \subseteq B (0, R)$, $\tmop{supp} \varrho \subseteq
  B (0, 2 R) \setminus B (0, R)$;
  
  \item $0 \leqslant \chi, \varrho \leqslant 1$, $\chi (\xi) + \sum_{j
  \geqslant 0} \varrho (2^{- j} \xi) = 1$ for any $\xi \in \mathbb{R}^d$;
  
  \item $\tmop{supp} \varrho (2^{- j} \cdot) \cap \tmop{supp} \varrho (2^{- i}
  \cdot) = \varnothing$ if $| i - j | > 1$.
\end{itemize}
Introduce the notations $\varrho_{- 1} = \chi$, $\varrho_j = \varrho (2^{- j}
\cdot)$ for $j \geqslant 0$. For any $f \in \CS' (\mathbb{R}^2)$ we define the
operators $\Delta_j f \assign \varrho_j (\mathD) f$, $j \geqslant - 1$.

\begin{definition}
  Let $s \in \mathbb{R}, p, q \in [1, \infty]$. For a Schwarz distribution $f
  \in \CS' (\mathbb{R}^2)$ define the norm
  \[ \| f \|_{B_{p, q}^s} = \| (2^{j s} \| \Delta_j f \|_{L^p})_{j \geqslant
     - 1} \|_{\ell^q} \]
  where $\| \|_{L^p}$ denotes the normalized $L^p (\Lambda)$ norm. The space
  $B^s_{p, q}$ is the set of functions $f \in \CS' (\Lambda)$ such that $\| f
  \|_{B_{p, q}^s} < \infty$ moreover $H^s = B^s_{2, 2}$ are the usual Sobolev
  spaces, and we denote by $\VV^s$ the closure of smooth functions in the
  $B_{\infty, \infty}^s$ norm.
\end{definition}

\begin{definition}
  A polynomial weight $\rho$ is a function $\rho : \mathbb{R}^2 \rightarrow
  \mathbb{R}_+$ of the form $\rho (x) = c \langle x \rangle^{- \sigma}$ for
  $\sigma, c \geqslant 0$. For a polynomial weight $\rho$ let
  \[ \| f \|_{L^p (\rho)} = \left( \int | f (x) |^p \rho (x) \mathd x
     \right)^{1 / p} \]
  and by $L^p (\rho)$ the space of functions for which this norm is finite. 
\end{definition}

\begin{definition}
  For a polynomial weight $\rho$ let
  \[ \| f \|_{L^p (\rho)} = \left( \int | f (x) |^p \rho (x) \mathd x
     \right)^{1 / p} \]
  and by $L^p (\rho)$ the space of functions for which this norm is
  finiteSimilarly we define the weighted Besov spaces $B_{p, q}^s (\rho)$ as
  the set of elements of $\CS' (\mathbb{R}^d)$ for which the norm
  \[ \| f \|_{B_{p, q}^s (\rho)} = \| (2^{j s} \| \Delta_j f \|_{L^p
     (\rho)})_{j \geqslant - 1} \|_{\ell^q} . \]
\end{definition}

We also introduce some spaces with exponential weights:

\begin{definition}
  For a set $z \in \mathbb{R}^2$, $r \in \mathbb{R}$ we define the weighted
  $L^p$ spaces
  \[ \| f \|_{L_z^{p, r}} = \left( \int \exp (r p | x |) f^p (x) \mathd x
     \right)^{1 / p} \]
  And
  \[ \| f \|_{W_z^{1, p, r}} = \| f \|_{L_z^{p, r}} + \left( \int (\exp (r p
     | x - z |)) (\nabla f (x))^p \mathd x \right)^{1 / p} \]
  We will also set $H_z^{1, r} \equallim W_z^{1, 2, r}$. Furthermore we will
  set
  \[ \| f \|_{L^{p, r}} = \| f \|_{L_0^{p, r}}, \qquad \| f \|_{W^{1, p, r}}
     = \| f \|_{W_0^{1, p, r}} . \]
\end{definition}

\begin{notation}
  \label{conv:grad}Throughout the chapter we will frequently compute Gradients
  and Hessian of functionals on $f : L^2 (\mathbb{R}^2) \rightarrow
  \mathbb{R}$. We will always interpret $\nabla f (\varphi)$, to be an element
  $L^2 (\mathbb{R}^2)$ by the Riesz representation theorem. Similarly we will
  always interpret $\tmop{Hess} f (\varphi)$ to be an operator on $L^2
  (\mathbb{R})$.
\end{notation}

\begin{definition}
  \label{def:norm-functional}For a Frechet differentiable functional $G : L^2
  (\langle x \rangle^{- n}) \rightarrow \mathbb{R}$ and $x \in \mathbb{R}^2$
  we define the quantities
  \[ | G |_{1, \infty} = \sup_{\varphi \in L^2 (\langle x \rangle^{- n})} \|
     \nabla G (\varphi) \|_{L^{\infty}} \]
  \[ | G |^x_{1, 2, r} = \sup_{\varphi \in L^2 (\langle x \rangle^{- n})} \|
     \nabla G (\varphi) \|_{L_x^{2, r}} . \]
\end{definition}

\begin{proposition}
  We have the following estimates
  \begin{enumerate}
    \item $\| f \|_{L^2 (\langle x \rangle^{- n})} \leqslant C \langle d (0,
    y) \rangle^{- n / 2} \| f \|_{L_y^{2, \gamma}}$
    
    \item Let $s \in \{ 0, 1 \}$ $r > 0$ and $f \in W_p^{s, r}$ is supported
    on $B (0, N)^c$, $N \geqslant 1$. Then
    \[ \| f \|_{W_p^{s, r - \kappa}} \leqslant \exp (- \kappa N) \| f
       \|_{W_p^{s, r}} . \]
  \end{enumerate}
\end{proposition}

For a proof see Appendix \ref{app:weighted}.

\subsection{Heat Kernel decomposition and martingale Cutoff
}\label{sec:decomposition}

We consider the decomposition $\left(with L = (m^2 - \Delta) \right)$
\[ L^{- 1} = \int^{\infty}_0 J^2_t \mathd t \]
where
\[ J_t = \left( \frac{1}{t^2} e^{- L / t} \right)^{1 / 2} . \]
We denote by
\begin{equation}
  \mathcal{C}_t = \int^t_0 J^2_s \mathd s = L^{- 1} e^{- L / t},
  \label{eq:heatkernel}
\end{equation}
and by $K_t (x, y)$ the kernel of $\mathcal{C}_t$. From the definitions one
can see that
\[ K_t (x, y) = \int^t_0 e^{- m^2 / s} \left( \frac{1}{s^2} \frac{s}{4 \pi}
   e^{- 4 s | x - y |^2} \right) \mathd s = \int^t_0 e^{- m^2 / s^2} \left(
   \frac{1}{4 \pi s} e^{- 4 s | x - y |^2} \right) \mathd s \]
so
\[ K_t (x, x) = \int^t_0 e^{- m^2 / s^2} \left( \frac{1}{4 \pi s} \right)
   \mathd s = \mathbbm{1}_{t \geqslant 1} \frac{1}{4 \pi} \log t + C (t) \]
where $\sup_{t \in \mathbb{R}_+} C (t) < \infty$. Let $0 \leqslant s < t$ and
$u \in L^2 ([s, t], L^2 (\mathbb{R}^2)) .$ For later use we introduce the
notation
\[ I_{s, t} (u) = \int^t_s J_l u_l \mathd l. \]
and we set $I_T (u) = I_{0, T} (u)$.

\begin{proposition}
  Let $- m < \gamma < m$. We have the following estimates
  \begin{enumerate}
    \item $\| J_t f \|_{L^{\infty}} \leqslant t^{- 1} \| f \|_{L^{\infty}}$
    
    \item $\| I_{s, t} (u) \|_{L^{2, \gamma} (\mathbb{R}^2)} \leqslant C
    \langle s \rangle^{- 1 / 2} \| u \|_{L^2 (\mathbb{R}_+, L^{2, \gamma})}$
    
    \item $\| I_{s, t} (u) \|_{W^{1, \infty}} \leqslant C \| \langle l
    \rangle^{1 / 2 + \delta} u_l \|_{L_l^{\infty} ([s, t] \times
    \mathbb{R}^2)}$
  \end{enumerate}
\end{proposition}

For a proof see Appendix \ref{app:weighted}.

We now define regularized GFF as
\[ W_{s, t} = \int^t_s Q_l \mathd X_l \]
where $X_s$ is a cylindrical Brownian motion on $L^2$. We set $W_t = W_{0,
t}$. We can calculate:
\[ \mathbb{E} [W_t (x) W_t (y)] = K_t (x, y) \]
and it is clear that $W_t$ is a martingale.

\

\subsection{Stochastic Control }

We are interested in studying the quantities
\[ V_{t, T} (\varphi) = - \log \mathbb{E} [\exp (- V_T (\varphi + W_{t, T}))]
\]
where $Z_{t, T} = \exp (- V_{t, T})$, for $\varphi \in L^2 (\mathbb{R}^2) .$

For the rest of this article we will denote by $C^n (L^2 (\rho))$ functions
$L^2 (\rho) \rightarrow \mathbb{R}$ which are $n$ times continuously
Fr{\'e}chet differntiable with bounded derivatives. The following proposition
is known as the Boue-Dupuis {\cite{boue_variational_1998}} or Borell formula
{\cite{borell_diffusion_2000}} and is a stochastic control representation of
the Polchinski equation. \

\begin{proposition}
  \label{prop:BD-formula}Asssume that $V_T \in C^2 (L^2 (\langle x \rangle^{-
  n})) \rightarrow \mathbb{R}^2$. Then
  \begin{equation}
    V_{t, T} (\varphi) = - \log \mathbb{E} [e^{- V_T (\varphi + W_{t, T})}] =
    \inf_{u \in \mathbb{H}_a} \mathbb{E} \left[ V_T (W_{t, T} + I_{t, T} (u))
    + \frac{1}{2} \int^T_t \| u_s \|^2_{L^2} \mathd s \right] \label{eq:BD}
  \end{equation}
  
\end{proposition}

\begin{proposition}[Verification]
  \label{prop:verification}Assume that $V_{t, T}$ is defined as in
  {\eqref{eq:BD}}. If $V_T$ is in $C^2 (L^2 (\langle x \rangle^{- n}))$ then
  so is $V_{t, T}$ and the equation
  \begin{equation}
    \mathd Y_{s, t} = - J^{2}_t \nabla V_{t, T} (Y_{s, t}) \mathd t + J_t \mathd
    X_t, \label{eq:minimizer}
  \end{equation}
  possses a unique solution in $L^{2}(\mathbb{P} C([0,T],L^{2}(\mathbb{R}^{2})))$. The infimum on the r.h.s is attained by
  \[ u_s = - J_s \nabla V_{s, T} (Y_{t, s}) = - J_s \nabla V_{s, T} (W_{t, s}
     + I_{t, s} (u)) . \]
\end{proposition}

\subsection{Wick ordered cosine }\label{sec:cosine}

Since the Gaussian Free Field in 2 dimensions is supported on distibutions
$\cos (\beta W_{\infty})$ is ill defined. However we can correct this by Wick
ordering, that is considering instead $\alpha (T) \cos (\beta W_T)$, with a
$\alpha (T) \rightarrow \infty$ and hoping to obtain a nontrivial limit. That
this is indeed possible is the content of the following propositon from
{\cite{Junnila-2020}}:.

\begin{proposition}
  Assume that $\beta^2 < 4 \pi$. Then there exists a (differentiable) function
  $\alpha (T)$ such that
  \begin{itemize}
    \item $\alpha (T) \leqslant C \langle T \rangle^{\beta^2 / 8 \pi}$
    
    \item $\llbracket \cos (\beta W_T) \rrbracket : = \alpha_T \cos (\beta
    W_T)$ is a martingale in $T$.
    
    \item For an $p \in [1, \infty)$, as $T \rightarrow \infty$ $\llbracket
    \cos (\beta W_T) \rrbracket$ converges in \ in $L^p (\mathbb{P}, B_{p,
    p}^{- \beta^2 / 4 \pi - 2 \delta} (\langle x \rangle^{- n}))$ to a limit
    $\llbracket \cos (\beta W_{\infty}) \rrbracket$.
  \end{itemize}
\end{proposition}

For a proof see Appendix \ref{app:cosine}. We can introduce the approximate
measures given by
\[ \int f (\varphi) \nu^{T, \rho}_{\tmop{SG}} (\mathd \varphi) =
   \frac{1}{Z^T} \mathbb{E} \left[ f (W_T) \exp \left( - \lambda \int \rho
   \llbracket \cos (\beta W_T) \rrbracket \mathd x \right) \right] \]
where $Z_T$ is a normalization constant. We shall see that $\nu^{T,
\rho}_{\tmop{SG}} \rightarrow \nu^{\rho}_{\tmop{SG}}$ weakly on $H^{- 1}
(\langle x \rangle^{- n})$.

\subsection{An Envelope theorem}

\begin{lemma}
  Let $T \in [0, \infty]$. Let $f, g : H^{- 1} (\langle x \rangle^{- n})
  \rightarrow \mathbb{R}$ be continuous bounded fuctions.
  \label{lemma:coupling}Let $u^g$ be a minimizer for
  \[ F^{T, \rho} (u) =\mathbb{E} \left[ g (W_T + I_T (u)) + \int \rho
     \llbracket \cos (\beta (W_T + I_T (u))) \rrbracket \mathd x + \frac{1}{2}
     \int^{\infty}_0 \| u_t \|^2_{L^2} \mathd t \right] . \]
  Then
  \[ \int f (\varphi) \nu^{T, \rho}_{\tmop{SG}} (\mathd \phi) =\mathbb{E} [f
     (W_T + I_T (u^g))], \]
  where
  \[ \mathd \nu^{T, \rho, g}_{\tmop{SG}} = \frac{1}{Z^g} \exp (- g (\varphi))
     \mathd \nu^{T, \rho}_{\tmop{SG}} \qquad Z^g = \int \exp (- g (\varphi))
     \mathd \nu^{T, \rho}_{\tmop{SG}} (\varphi) . \]
\end{lemma}

\begin{proof}
  This is a version of the envelope thoerem
  {\cite{Milgrom-2002,albeverio2020mean}}. We have
  \[ \int f (\varphi) \mathd \nu^{T, \rho, g}_{\tmop{SG}} \left. =
     \frac{\mathd}{\mathd \upsilon} \right|_{\upsilon = 0} - \log \int \exp (-
     \upsilon f (\varphi)) \mathd \nu^{T, \rho, g}_{\tmop{SG}} . \]
  and the r.h.s is well known to be differntiable in $\upsilon$ since it's the
  cumulant generating function. Now
  \begin{eqnarray*}
    &  & - \log \int \exp (f (\varphi)) \mathd \nu^{T, \rho, g}_{\tmop{SG}}\\
    & = & \inf_{u \in \mathbb{H}_a} \mathbb{E} \left[ f (W_T + I_T (u)) + g
    (W_T + I_T (u)) + \int \rho \left\llbracket \cos (\beta (W_T + I_T (u)))
    \right\rrbracket \mathd x + \frac{1}{2} \int^T_0 \| u_t \|^2_{L^2} \mathd
    t \right]\\
    & \backassign & \inf_{u \in \mathbb{H}_a} F^f (u)
  \end{eqnarray*}
  For $T < \infty$ this equality follows from Corollary \ref{cor:BD-formula}
  in Appendix \ref{app:stoch-cont}. For $T = \infty$ we have from Theorem 1.3
  in {\cite{Junnila-2020}} that
  \[ \mathbb{E} \left[ \exp \left( p \left| \int \llbracket \cos (\beta
     W_{\infty}) \rrbracket \mathd x \right| \right) \right] < \infty \]
  so $\int \llbracket \cos (\beta W_{\infty}) \rrbracket \mathd x$ is a tame
  functional in the language of {\"U}st{\"u}nel
  {\cite{ustunel_variational_2014}}, which implies the variational formula
  (see {\cite{ustunel_variational_2014}}). Now note that
  \begin{eqnarray*}
    \lim_{\upsilon \rightarrow 0} \frac{\inf_{u \in \mathbb{H}_a} F^{\upsilon
    f} (u) - \inf_{u \in \mathbb{H}_a} F^0 (u)}{\upsilon} & \geqslant &
    \lim_{\upsilon \rightarrow 0} \frac{F^{\upsilon f} (u^g) - F^0
    (u^g)}{\upsilon}\\
    & = & f (W_{\infty} + I_{\infty} (u^{T, \rho, g})) .
  \end{eqnarray*}
  On the other hand by differentiability
  \begin{eqnarray*}
    \lim_{\upsilon \rightarrow 0} \frac{\inf_{u \in \mathbb{H}_a} F^{\upsilon
    f} (u) - \inf_{u \in \mathbb{H}_a} F^0 (u)}{\upsilon} & = & \lim_{\upsilon
    \rightarrow 0} \frac{\inf_{u \in \mathbb{H}_a} F^0 (u) - \inf_{u \in
    \mathbb{H}_a} F^{- \upsilon f} (u)}{\upsilon}\\
    & \leqslant & \lim_{\upsilon \rightarrow 0} \frac{F^0 (u^g) - F^{-
    \upsilon f} (u^g)}{\upsilon}\\
    & = & f (W_T + I_T (u^g))
  \end{eqnarray*}
  which implies the statement.
\end{proof}

Note that Lemma \ref{lemma:coupling} and Proposition \ref{prop:conv-min} below
imply that $\nu^{T, \rho}_{\tmop{SG}} \rightarrow \nu^{\rho}_{\tmop{SG}}$
weakly (e.g on $H^{- 1} (\langle x \rangle^{- n})$).

\section{Proof of Theorem \ref{thm:L-infty}}\label{sec:proof}

\begin{lemma}
  \label{lemma:envelope-grad}Asssume that $| V_T |_{1, \infty} + | V |_{2, 2}
  < \infty$ then
  \begin{equation}
    V_{t, T} (\varphi) = \inf_{u \in \mathbb{H}_a} \mathbb{E} \left[ V_T
    (W_{t, T} + I_{t, T} (u) + \varphi) + \frac{1}{2} \int^T_t \| u_s
    \|^2_{L^2} \mathd t \right] \label{eq:abstract-bd}
  \end{equation}
  satisfies
  \[ \nabla V_T (\varphi) =\mathbb{E} [\nabla V_T (W_{t, T} + I_{t, T}
     (u^{\varphi}) + \varphi)] \]
  where $u^{\varphi}$ is the optimizer on the r.h.s of
  {\eqref{eq:abstract-bd}}, in particular
  \[ | V_{t, T} |_{1, \infty} \leqslant | V_T |_{1, \infty} . \]
\end{lemma}

\begin{proof}
  We have
  \[ \langle \nabla V_{t, T} (\varphi), \psi \rangle = \lim_{\varepsilon
     \rightarrow 0} \frac{1}{\varepsilon} (V_{t, T} (\varphi + \varepsilon
     \psi) - V_{t, T} (\varphi)) \]
  which implies
  \begin{eqnarray*}
    & \langle \nabla V_{t, T} (\varphi), \psi \rangle_{L^2 (\mathbb{R}^2)}\\
    \leqslant & \lim_{\varepsilon \rightarrow 0} \frac{1}{\varepsilon} \left(
    \mathbb{E} \left[ V_T (W_{t, T} + I_{t, T} (u^{\varphi}) + \varphi +
    \varepsilon \psi) + \frac{1}{2} \int^T_t \| u^{\varphi}_s \|^2_{L^2}
    \mathd t \right] \right.\\
    & - \left. \mathbb{E} \left[ V_T (W_{t, T} + I_{t, T} (u^{\varphi}) +
    \varphi) + \frac{1}{2} \int^T_t \| u^{\varphi}_s \|^2_{L^2} \mathd t
    \right] \right)\\
    = & \mathbb{E} \int \nabla V_T (W_{t, T} + I_{t, T} (u^{\varphi}) +
    \varphi) \psi \mathd x
  \end{eqnarray*}
  on the other hand we have by differentiability,
  \[ \langle \nabla V_{t, T} (\varphi), \psi \rangle = \lim_{\varepsilon
     \rightarrow 0} \frac{1}{\varepsilon} (V_{t, T} (\varphi) - V_{t, T}
     (\varphi - \varepsilon \psi)) \]
  \[ \begin{array}{lll}
       \langle \nabla V_{t, T} (\varphi), \psi \rangle & \geqslant &
       \lim_{\varepsilon \rightarrow 0} \frac{1}{\varepsilon} \left(
       \mathbb{E} \left[ V_T (W_{t, T} + I_{t, T} (u^{\varphi}) + \varphi) +
       \frac{1}{2} \int^T_t \| u^{\varphi}_s \|^2_{L^2} \mathd t \right]
       \right.\\
       &  & - \left. \mathbb{E} \left[ V_T (W_{t, T} + I_{t, T} (u^{\varphi})
       + \varphi - \varepsilon \psi) + \frac{1}{2} \int^T_t \| u^{\varphi}_s
       \|^2_{L^2} \mathd t \right] \right)\\
       & = & \mathbb{E} \int \nabla V_T (W_{t, T} + I_{t, T} (u^{\varphi}) +
       \varphi) \psi \mathd x
     \end{array} \]
  and we get the statement. 
\end{proof}

\begin{corollary}
  \label{linfty-1}Assume that $t \geqslant T / 2$ and $| V_T |_{1, \infty}
  \leqslant C \lambda T^{1 / 2 - \delta}$. Then with $u^{\varphi}$ minimizing
  \[ F (u) =\mathbb{E} \left[ V_T (W_{t, T} + I_{t, T} (u) + \varphi) +
     \frac{1}{2} \int^T_t \| u_t \|^2_{L^2} \mathd t \right] . \]
  Then $u^{\varphi}$ satisfies
  \[ \| u_s^{\varphi} \|_{L^{\infty}} \leqslant C \lambda \langle s \rangle^{-
     1 / 2} . \]
\end{corollary}

\begin{proof}
  By Proposition \ref{prop:verification} $u_s^{\varphi}$ satisfies
  \begin{eqnarray*}
    \| u_s^{\varphi} \|_{L^{\infty}} & = & \| J_s \nabla \lambda V_{s, T} (W_s
    + I_s (u^{\varphi})) \|_{L^{\infty}}\\
    & \leqslant & s^{- 1} \lambda T^{1 / 2 - \delta}\\
    & \leqslant & \lambda \langle s \rangle^{- 1 / 2 - \delta},
  \end{eqnarray*}
  since $s \geqslant T / 2$.
\end{proof}

We now show that $u^{\varphi}$ also satisfies an Euler-Lagrange equaiton

\begin{lemma}
  Denote by $u^{\varphi}$ the minizer of the r.h.s of
  {\eqref{eq:abstract-bd}}. It exists by Proposition \ref{prop:verification}
  since $V_T \in C^2 (L^2 (\mathbb{R}^2), \mathbb{R})$. We show that
  $u^{\varphi}$ satisfies, for any $h \in \mathbb{H}_a$
  \begin{equation}
    \mathbb{E} \left[ \int \nabla V_T (W_{t, T} + I_{t, T} (u^{\varphi}) +
    \varphi) I_{t, T} (h) \mathd x + \int^T_t \int u^{\varphi} h \mathd x
    \mathd t \right] . \label{eq:EL-1}
  \end{equation}
\end{lemma}

\begin{proof}
  Indeed since $u^{\varphi}$ is a minimizer we have for any $h \in
  \mathbb{H}_a$
  \begin{eqnarray*}
    &  & \mathbb{E} \left[ V_T (W_{t, T} + I_{t, T} (u^{\varphi} \pm
    \varepsilon h) + \varphi) + \frac{1}{2} \int^T_t \| u^{\varphi}_s \pm
    \varepsilon h_s \|^2_{L^2} \mathd s \right]\\
    &  & -\mathbb{E} \left[ V_T (W_{t, T} + I_{t, T} (u^{\varphi}) + \varphi)
    + \frac{1}{2} \int^T_t \| u^{\varphi}_s \|^2_{L^2} \mathd s \right]
    \geqslant 0
  \end{eqnarray*}

  and taking $\varepsilon \rightarrow 0$ we get the claim since
  \begin{eqnarray*}
    &  & \lim_{\varepsilon \rightarrow 0} \frac{1}{\varepsilon} (V_T (W_{t,
    T} + I_{t, T} (u^{\varphi} \pm \varepsilon h) + \varphi) - V_T (\varphi +
    I_{t, T} (u) + \varphi))\\
    & = & \pm \int \nabla V_T (W_{t, T} + I_{t, T} (u^{\varphi}) + \varphi)
    I_{t, T} (h) \mathd x
  \end{eqnarray*}
  and
  \[ \lim_{\varepsilon \rightarrow 0} \frac{1}{\varepsilon} \left( \int^T_t
     \| u^{\varphi}_s \pm \varepsilon h_s \|^2_{L^2} \mathd t - \int^T_t \|
     u^{\varphi}_s \|^2_{L^2} \mathd t \right) = \pm 2 \int^T_t \int
     u_s^{\varphi} h_s \mathd x \mathd s \]
  so
  \[ \pm \mathbb{E} \left[ \int \nabla V_T (W_{t, T} + I_{t, T} (u^{\varphi})
     + \varphi) I_{t, T} (u^{\varphi}) \mathd x + \int^T_t \int u_s^{\varphi}
     h_s \mathd x \mathd s \right] \geqslant 0 \]
  which implies the claim. 
\end{proof}

\begin{lemma}
  \label{lemma:value-func}Assume that
  \[ \nabla V_T (\varphi) = \lambda \alpha_T \beta \rho \sin (\beta \varphi)
     + R_T (\varphi) \]
  with $\sup_{\varphi \in L^2 (\langle x \rangle^n)} \| R_T (\varphi)
  \|_{L^{\infty}} \leqslant C$ and $\beta^2 < 4 \pi$. Then for $t \geqslant T
  / 2$
  \[ \nabla V_{t, T} (\varphi) = \lambda \alpha_t \beta \rho \sin (\beta
     \varphi) + R_{t, T} (\varphi) \]
  where
  \begin{enumerate}
    \item The following inquality holds
    \[ \sup_{\varphi \in L^2 (\langle x \rangle^n)} \| R_{t, T} (\varphi)
       \|_{L^{\infty}} \leqslant C \lambda^2 \langle t \rangle^{- \delta} +
       \sup_{\varphi \in L^2 (\langle x \rangle^n)} \| R_T (\varphi)
       \|_{L^{\infty}} . \]
    \item There exists a constant dependent on $\rho$ such that
    \[ \sup_{\varphi \in L^2 (\langle x \rangle^n)} \| R_{t, T} (\varphi)
       \|_{L^2} \leqslant C_{\rho} \lambda^2 \langle t \rangle^{- \delta} +
       \sup_{\varphi \in L^2 (\langle x \rangle^n)} \| R_T (\varphi) \|_{L^2}
       . \]
    \item There exists a constant $C$ (independent of $T, \lambda, \beta$)
    such that for $t \geqslant T / 2 \wedge C$
    \[ \sup_{\varphi, \psi \in L^2 (\langle x \rangle^n)} \frac{\| R_{t, T}
       (\varphi) - R_{t, T} (\psi) \|_{L^2}}{\| \varphi - \psi \|_{L^2}}
       \leqslant C \left( \sup_{\varphi, \psi \in L^2 (\langle x \rangle^n)}
       \frac{\| R_T (\varphi) - R_T (\psi) \|_{L^2}}{\| \varphi - \psi
       \|_{L^2}} \right) . \]
  \end{enumerate}
  
\end{lemma}

\begin{proof}
  {\tmstrong{Proof of 1. Step 1}} We first establish the estimate for $t
  \geqslant T / 2$. By Lemma \ref{lemma:envelope-grad}
  \begin{eqnarray*}
    &  & \nabla V_{t, T} (\varphi)\\
    & = & \mathbb{E} [\lambda \alpha_T \beta \rho \sin (\beta (W_{t, T} +
    I_{t, T} (u^{\varphi}) + \varphi))] +\mathbb{E} [R_T (W_{t, T} + I_{t, T}
    (u^{\varphi}) + \varphi)]\\
    & = & \mathbb{E} [\lambda \alpha_T \beta \rho \sin (\beta (W_{t, T} +
    \varphi))] +\mathbb{E} \left[ \lambda \int^1_0 \alpha_T \beta \sin (\beta
    (W_{t, T} + s I_{t, T} (u^{\varphi}) + \varphi)) I_{t, T} (u^{\varphi})
    \mathd s \right] \\
    &  &+\mathbb{E} [R_T (W_{t, T} + I_{t, T} (u^{\varphi}) + \varphi)]
  \end{eqnarray*}
  Now note that by the martingale property of the renormalized sine
  \[ \mathbb{E} [\alpha_T \sin (\beta W_{t, T} + \beta \varphi)] = \alpha_t
     \beta \sin (\beta \varphi) \]
  so it remains to estimate the second term. Since
  \begin{eqnarray*}
    &  & \left\| \mathbb{E} \left[ \lambda \int^1_0 \alpha_T \beta \rho \sin
    (\beta (W_{t, T} + s I_{t, T} (u^{\varphi}) + \varphi)) I_{t, T}
    (u^{\varphi}) \mathd s \right] \right\|_{L^{\infty}}\\
    & \leqslant & \mathbb{E} [\sup_{s \in [0, 1]} \lambda \alpha_T \beta \|
    \rho \sin (\beta (W_{t, T} + s I_{t, T} (u^{\varphi}) + \varphi)) I_{t, T}
    (u^{\varphi}) \|_{L^{\infty}}]\\
    & \leqslant & \lambda \mathbb{E} [\alpha_T \beta \| I_{t, T}
    (u^{\varphi}) \|_{L^{\infty}}]\\
    & \leqslant & C \lambda^2  \langle t \rangle^{- 1 / 2} \alpha_T\\
    & \leqslant & C \lambda^2  \langle t \rangle^{- \delta},
  \end{eqnarray*}
  where in the second to last line we used Corollary \ref{linfty-1}. This
  gives the first inequality for $t \geqslant T / 2$.
  
  {\tmstrong{Proof of 1. Step 2}} By Proposition \ref{prop:dynamicprogramming}
  below we have that
  \[ V_{t, T} (\varphi) = \inf_{u \in \mathbb{H}_a} \mathbb{E} \left[ V_{T /
     2, T} (W_{t, T / 2} + I_{t, T / 2} (u) + \varphi) + \frac{1}{2} \int^{T /
     2}_t \| u_t \|^2_{L^2} \mathd t \right] \]
  so applying Step 1 we get for $t \geqslant T / 4$
  \[ \nabla V_{t, T} (\varphi) = \lambda \beta \alpha_t \rho \sin (\beta
     \varphi) + R_{t, T} (\varphi) \]
  with
  \[ \sup_{\varphi \in L^2 (\langle x \rangle^n)} \| R_{t, T} (\varphi)
     \|_{L^{\infty}} \leqslant C \lambda^2 (\langle T / 2 \rangle^{- \delta} +
     \langle T / 4 \rangle^{- \delta}) . \]
  Prooceeding like this inductivly we obtain for $T / 2^{i - 1} \geqslant t
  \geqslant T / 2^i$
  \begin{eqnarray*}
    \sup_{\varphi \in L^2 (\langle x \rangle^n)} \| R_{t, T} (\varphi)
    \|_{L^{\infty}} & \leqslant & C \lambda^2 \sum_{k = 1}^i \langle T / 2^k
    \rangle^{- \delta}\\
    & \leqslant & 2 C \lambda^2 \sum_{k = 1}^i \langle 2^i t / 2^k \rangle^{-
    \delta}\\
    & \leqslant & C \lambda^2 \langle t \rangle^{- \delta} \sum_{k = 1}^i
    2^{(k - i) \delta}\\
    \tilde{k} = (i - k) & = & C \lambda^2 \langle t \rangle^{- \delta}
    \sum_{\tilde{k} = 1}^i 2^{- \tilde{k} \delta}\\
    & \leqslant & C \lambda^2 \langle t \rangle^{- \delta}
  \end{eqnarray*}
  which is the desired statement.
  
  {\tmstrong{Proof of 2.}} The proof of the second assertion is analogous
  since in that case $\rho \sin (\cdot)$ is in $L^2 (\mathbb{R}^2)$.
  
  {\tmstrong{Proof of 3. Step 1}} We first proof the statment for $t
  \geqslant T / 2 \wedge C$. For this we have to first estimate the difference
  \[ I_{t, T} (u^{\varphi, \rho}) - I_{t, T} (u^{\psi, \rho}) . \]
  Observe that
  \begin{eqnarray*}
    &  & F^{\rho, \psi}  (u) - F^{\rho, \varphi}  (u)\\
    & = & \mathbb{E} \left[ V_T (W_{t, T} + I_{t, T}  (u) + \psi) +
    \frac{1}{2} \int^{\infty}_0 \| u_t \|^2_{L^2} \mathd t \right]\\
    & = & \mathbb{E} \left[ V_T (W_{t, T} + I_{t, T}  (u) + \varphi) \right. \\
    &  & \left. + \int^1_0 \int \nabla V_T \left( W_{t, T} + I_{t, T}  (u) + (1 - \theta)
          \varphi + \theta \psi \right) (\varphi - \psi) \mathd x \mathd \theta
    + \frac{1}{2} \int^{\infty}_0 \| u_t \|^2_{L^2} \mathd t \right]\\
    & = & F^{\rho, \varphi}  (u) +\mathbb{E} \left[ \int^1_0 \int \nabla V_T
    (W_{t, T} + I_{t, T}  (u) + (1 - \theta) \varphi + \theta \psi) (\varphi -
    \psi) \mathd x \right]\\
    & \backassign & F^{\rho, \varphi}  (u) + K (u, \varphi, \psi) .
  \end{eqnarray*}
  and the last line is to be taken as a definition for $K$. By our assumption
  we can estimate
  \begin{eqnarray*}
    &  & | K (u, \varphi, \psi) - K (v, \varphi, \psi) |\\
    & \leqslant & \mathbb{E} \left[ \int^1_0 \int (\nabla V_T (W_{t, T} +
    I_{t, T}  (u) + (1 - \theta) \varphi + \theta \psi) \right. \\ &  &\left. - \nabla V_T (W_{t, T}
    + I_{t, T}  (v) + (1 - \theta) \varphi + \theta \psi)) (\varphi - \psi)
    \mathd x \right]\\
    & \leqslant & C \langle T \rangle^{\beta^2 / 4 \pi} \mathbb{E} [\| I_{t,
    T}  (u - v) \|_{L^2 (\mathbb{R}^2)}] \| \varphi - \psi \|_{L^2
    (\mathbb{R}^2)}\\
    & \leqslant & C t^{- 1 / 2} \langle T \rangle^{\beta^2 / 4 \pi}
    \mathbb{E} [\| u - v \|_{L^2 (\mathbb{R}_+ \times \mathbb{R}^2)}] \|
    \varphi - \psi \|_{L^2 (\mathbb{R}^2)}\\
    & \leqslant & C t^{- \delta} \mathbb{E} [\| u - v \|_{L^2 (\mathbb{R}_+
    \times \mathbb{R}^2)}] \| \varphi - \psi \|_{L^2 (\mathbb{R}^2)}
  \end{eqnarray*}
  where in the second to last line we used Lemma
  \ref{lemma:boundIL2} and in the last line the assumption $T
  \leqslant 2 t$.
  
  Furthermore $u^{\psi, \rho}$ is a minimizer implies
  \begin{eqnarray*}
    F^{\rho, \psi}  (u^{\psi, \rho}) - F^{\rho, \psi}  (u^{\varphi, \rho}) &
    \leqslant & 0.
  \end{eqnarray*}
  On the other hand using the semiconvexity of $V_T$ (which follows from the
  assumption) and the Euler Lagrange equation for $u^{\varphi, \rho}$
  (Equation {\eqref{eq:EL-1}}) we get
  \begin{eqnarray*}
    &  & F^{\rho, \varphi}  (u) - F^{\rho, \varphi}  (u^{\varphi, \rho})\\
    & = & \mathbb{E} \left[ V_T (W_{t, T} + I_{t, T}  (u) + \varphi) - V_T
    \left( W_{t, T} + I_{t, T}  (u^{\varphi, \rho}) + \varphi \right) +
    \frac{1}{2} \int^{\infty}_0 \| u_t \|^2_{L^2} \mathd t - \frac{1}{2}
    \int^{\infty}_0 \| u^{\varphi, \rho}_t \|^2_{L^2} \mathd t \right]\\
    & \geqslant & \mathbb{E} \left[ \int^1_0 \int \nabla V_T (W_{t, T} +
    I_{t, T}  (u^{\varphi, \rho}) + \varphi) I_{t, T}  (u - u^{\varphi, \rho})
    \mathd x + \int^{\infty}_0 \int u^{\varphi, \rho}_t (u_t - u_t^{\varphi,
    \rho}) \mathd x \mathd t \right. \\
    &   & \left. - T^{1 / 2} \| I_{t, T} (u - u^{\varphi, \rho}) \|^2_{L^2} + \frac{1}{2}
    \int^{\infty}_0 \| u_t - u_t^{\varphi, \rho} \|^2_{L^2} \mathd t \right] \\
    & = & - C T^{1 / 2} \mathbb{E} [ \| I_{t, T} (u - u^{\varphi, \rho})
    \|^2_{L^2}] + \frac{1}{2} \mathbb{E} \left[ \int^{\infty}_0 \| u_t -
    u_t^{\varphi, \rho} \|^2_{L^2} \mathd t \right] \\
    & \geqslant & \frac{1}{4} \int^{\infty}_0 \| u_t - u_t^{\varphi, \rho}
    \|^2_{L^2} \mathd t.
  \end{eqnarray*}
  In the last line we used Lemma \ref{lemma:boundIL2} and the assupmution that
  $t \geqslant T / 2 \wedge C$.
  
  Combining everything we get
  \begin{eqnarray*}
    0 & \geqslant & F^{\psi, \rho}  (u^{\psi, \rho}) - F^{\psi, \rho} 
    (u^{\varphi, \rho})\\
    & = & F^{\varphi, \rho}  (u^{\psi, \rho}) - F^{\varphi, \rho} 
    (u^{\varphi, \rho})\\
    &  & + K (u^{\psi, \rho}, \varphi, \psi) - K (u^{\varphi, \rho}, \varphi,
    \psi)\\
    & \geqslant & \frac{1}{4} \mathbb{E} \left[ \int^{\infty}_0 \| u^{\psi,
    \rho}_t - u_t^{\varphi, \rho} \|^2_{L^2} \mathd t \right] - | K (u^{\psi,
    \rho}, \varphi, \psi) - K (u^{\varphi, \rho}, \varphi, \psi) |\\
    & \geqslant & \frac{1}{4} \mathbb{E} [\| u^{\psi, \rho} - u^{\varphi,
    \rho} \|^2_{L^2 (\mathbb{R}_+ \times \mathbb{R}^2)}] - C t^{- \delta}
    \mathbb{E} [\| u^{\psi, \rho} - u^{\varphi, \rho} \|^2_{L^2 (\mathbb{R}_+
    \times \mathbb{R}^2)}]^{1 / 2} \| \varphi - \psi \|_{L^2 (\mathbb{R}^2)}
  \end{eqnarray*}
  which implies
  \[ \| u^{\psi, \rho} - u^{\varphi, \rho} \|_{L^2 (\mathbb{R} \times
     \mathbb{R}^2)} \leqslant C t^{- \delta} \| (\varphi - \psi) \|_{L^2
     (\mathbb{R}^2)} . \]
  Note that this implies $\| I_{t, T} (u^{\psi, \rho} - u^{\varphi, \rho})
  \|_{L^2 (\mathbb{R}^2)} \leqslant C t^{- 1 / 2 - \delta} \| (\varphi - \psi)
  \|_{L^2 (\mathbb{R}^2)}$.
  
  {\tmstrong{Proof of 3. Step 2}} Recall that we have to estimate
  \begin{eqnarray*}
    &  & \left\| \lambda \mathbb{E} \left[ \int^1_0 \alpha_T \beta \sin
    (\beta (W_{t, T} + s I_{t, T} (u^{\varphi, \rho}) + \varphi)) I_{t, T}
    (u^{\varphi, \rho}) \mathd s \right] \right. \\
    &  &- \lambda \left[ \int^1_0 \alpha_T \beta \sin (\beta (W_{t, T} + s I_{t,T} (u^{\psi, \rho}) + \psi)) I_{t, T} (u^{\psi, \rho}) \mathd s \right]
    \left. \right\|_{L^2}\\
    & \leqslant & \lambda \left\| \int^1_0 \alpha_T \beta \sin (\beta (W_{t,
    T} + s I_{t, T} (u^{\varphi, \rho}) + \varphi)) I_{t, T} (u^{\varphi,
    \rho} - u^{\psi, \rho}) \mathd s \right\|_{L^2}\\
    &  & + \lambda \left\| \int^1_0 \alpha_T \beta (\sin (\beta (W_{t, T} + s
    I_{t, T} (u^{\varphi, \rho}) + \varphi)) - \sin (\beta (W_{t, T} + s I_{t,
    T} (u^{\psi, \rho}) + \psi))) I_{t, T} (u^{\psi, \rho}) \mathd s
    \right\|_{L^2}\\
    & \leqslant & C \lambda (T^{1 / 2} \| I_{t, T} (u^{\varphi, \rho} -
    u^{\psi, \rho}) \|_{L^2} + T^{1 / 2} \| I_{t, T} (u^{\psi, \rho})
    \|_{L^{\infty}} (\| I_{t, T} (u^{\varphi, \rho} - u^{\psi, \rho}) \|_{L^2}
    + \| \varphi - \psi \|_{L^2}))\\
    & \leqslant & C t^{- \delta} \| \varphi - \psi \|_{L^2}
  \end{eqnarray*}
  Where in the last line we have used Corrollary \ref{linfty-1}.
  
  Finally we can estimate
  \begin{eqnarray*}
    &  & \| \mathbb{E} [R_T (W_{t, T} + I_{t, T} (u^{\varphi}) + \varphi) -
    R_T (W_{t, T} + I_{t, T} (u^{\psi}) + \psi)] \|_{L^2}\\
    & \leqslant & \| R_T \|_{L^2 \rightarrow L^2} (\| I_{t, T} (u^{\varphi})
    - I_{t, T} (u^{\psi}) \|_{L^2} + \| \varphi - \psi \|_{L^2})\\
    & \leqslant & \| R_T \|_{L^2 \rightarrow L^2} ((1 + t^{- 1 / 2 - \delta})
    \| \varphi - \psi \|_{L^2})
  \end{eqnarray*}
  where we have used the notation
  \[ \| R_T \|_{L^2 \rightarrow L^2} = \sup_{\varphi, \psi \in L^2 (\langle x
     \rangle^n)} \frac{\| R_T (\varphi) - R_T (\psi) \|_{L^2}}{\| \varphi -
     \psi \|_{L^2}} . \]
  \tmcolor{black}{{\tmstrong{Proof of 3. Step 3}} Now prooceding as in the
  proof of Assertion 1, Step 2, we obtain for $t \geqslant T / 2^i \wedge C$
  \[ \| R_{t, T} \|_{L^2 \rightarrow L^2} \leqslant \prod_{k = 1}^i (1 +
     \langle T / 2^k \rangle^{- \delta}) \]}
  
  Now obeserve that
  \[ \log \prod_{k = 1}^i (1 + \langle T / 2^k \rangle^{- \delta}) \leqslant
     \sum_{k = 1}^i \langle T / 2^k \rangle^{- \delta} \leqslant C \sum_{k =
     1}^i 2^{\delta (i - k)} \leqslant C \]
  This gives the statement for $t \geqslant C$. Now finally we can conclude by
  Proposition \ref{HJB-partitionfunction}. 
\end{proof}

\begin{proof*}{Proof of Theorem \ref{thm:L-infty}}
  Let $u^{T, \rho}$ be the minimzer for the functional
  \[ F^T (u) =\mathbb{E} \left[ \int \rho \llbracket \cos (\beta (W_T + I_T
     (u))) \rrbracket \mathd x + \frac{1}{2} \int^T_0 \| u_t \|^2_{L^2} \mathd
     t \right] . \]
  By Proposition \ref{prop:verification} we have that
  \[ \| u^{T, \rho} \|_{L^{\infty}} = \| Q_t \nabla V_{t, T} (W_T + I_t (u))
     \|_{L^{\infty}} \]
  where
  \[ V_{t, T} (\varphi) = \inf_{u \in \mathbb{H}_a} \mathbb{E} \left[ \int
     \rho \llbracket \nobracket \cos (\beta (W_{t, T} + I_{t, T} (u)
     \nobracket + \varphi)) \rrbracket \mathd x + \frac{1}{2} \int^T_t \| u_t
     \|^2_{L^2} \mathd t \right] . \]
  Now by Lemma \ref{lemma:value-func}
  \[ \| Q_t \nabla V_{t, T} (W_t + I_t (u)) \|_{L^{\infty}} \leqslant t^{- 1}
     \sup_{\varphi \in L^2 (\langle x \rangle^n)} \| \nabla V_{t, T} (\varphi)
     \|_{L^{\infty}} \leqslant C t^{- 1} (\alpha (t) + t^{- \delta}) \leqslant
     C t^{- 1 / 2 - \delta}, \]
  which implies the statment. 
\end{proof*}

\begin{proposition}
  \label{prop:conv-min}Let $u^{T, \rho}$ be the minimzer for the functional
  \[ F^T (u) =\mathbb{E} \left[ \int \rho \llbracket \cos (\beta (W_T + I_T
     (u))) \rrbracket \mathd x + \frac{1}{2} \int^{\infty}_0 \| u_t \|^2_{L^2}
     \mathd t \right], \]
  then as $T \rightarrow \infty$ $u^{T, \rho}$ converges to $u^{\infty, \rho}$
  in $\mathbb{H}_a$ to a minimizer of $F^{\infty} (u)$. 
\end{proposition}

\begin{proof}
  We first show that $u^{T, \rho}$ is a Couchy-sequence. By Proposition
  \ref{prop:verification} $u^{T, \rho}$ satisfies the equation
  \[ u_t^{T, \rho} = J_t V_{t, T} (W_t + I_t (u^{T, \rho})_{}) \]
  so
  \[ (u_t^{T_1, \rho} - u_t^{T_2, \rho}) = J_t (V_{t, T_1} (W_t + I_t
     (u^{T_1, \rho})_{}) - V_{t, T_2} (W_t + I_t (u^{T_2, \rho})_{})) . \]
  Now note that by Lemma \ref{lemma:value-func}
  \[ \nabla V_{T_1, T_2} (\varphi) = \lambda \alpha_{T_1} \beta \sin (\beta
     \varphi) + R_{T_1, T_2} (\varphi) \]
  with $\| R_{T_1, T_2} \|_{L^2} \leqslant C_{\rho} T_1^{- \delta}$. This
  implies by Proposition \ref{prop:dynamicprogramming} that with $T_1
  \leqslant T_2$
  \[ \sup_{\varphi \in L^2 (\langle x \rangle^n)} \| \nabla V_{t, T_1}
     (\varphi) - \nabla V_{t, T_2} (\varphi) \|_{L^2} \leqslant C \langle t
     \rangle^{- \delta} \langle T_1 \rangle^{- \delta}, \]
  and
  \begin{eqnarray*}
    &  & \| Q_t (V_{t, T_1} (W_t + I_t (u^{T_1, \rho})_{}) - V_{t, T_2} (W_t
    + I_t (u^{T_2, \rho})_{})) \|_{L^2}\\
    & \leqslant & C t^{- 1} \| \nobracket (V_{t, T_1} (W_t + I_t (u^{T_1,
    \rho})_{}) - V_{t, T_2} (W_t + I_t (u^{T_2, \rho})_{}))) \|_{L^2}\\
    & \leqslant & C t^{- 1} \| \nobracket (V_{t, T_1} (W_t + I_t (u^{T_1,
    \rho})_{}) - V_{t, T_1} (W_t + I_t (u^{T_2, \rho})_{}))) \|_{L^2}\\
    &  & + C t^{- 1} \sup_{\varphi \in L^2 (\langle x \rangle^n)} \| V_{t,
    T_1} (\varphi) - V_{t, T_2} (\varphi) \|_{L^2}\\
    & \leqslant & C t^{- 1} t^{\beta^2 / 8 \pi} \| I_t (u^{T_1, \rho} -
    u^{T_2, \rho}) \|_{L^2} + C t^{- 1} \sup_{\varphi \in L^2 (\langle x
    \rangle^n)} \| V_{t, T_1} (\varphi) - V_{t, T_2} (\varphi) \|_{L^2}\\
    & \leqslant & C \int^t_0 s^{- 1 - \delta} \| u_s^{T_1, \rho} - u_s^{T_2,
    \rho} \|_{L^2 (\mathbb{R}^2)} + C t^{- 1} (T_1)^{- \delta} .
  \end{eqnarray*}
  Gronwalls lemma now gives
  \[ \| u_t^{T_1, \rho} - u_t^{T_2, \rho} \|_{L^2} \leqslant C t^{- 1}
     (T_1)^{- \delta} \]
  which implies that $u_s^{T_1, \rho}$ is a Couchy sequence in $L^2
  (\mathbb{R}_+ \times \mathbb{R}^2)$ (since $t^{- 1}$ is in $L^2$). To prove
  that the limit is a minimizer of $F^{\infty, \rho}$ it is sufficient to
  check that
  \[ \qquad \lim_{T \rightarrow \infty} F^{T, \rho} (u^T) = F^{\infty, \rho}
     (u) \]
  for any sequence $u^T \rightarrow u$ in $\mathbb{H}_a$. Indeed this implies
  \[ F^{\infty, \rho} (u) \geqslant \inf_{u \in \mathbb{H}_a} F^{\infty,
     \rho} (u) = \inf_{u \in \mathbb{H}_a} \lim_{T \rightarrow \infty} F^{T,
     \rho} (u) \geqslant \lim_{T \rightarrow \infty} F^{T, \rho} (u^{T, \rho})
     = F^{\infty, \rho} (u^{\infty, \rho}), \]
  which gives the claim. To show this observe that
  \begin{eqnarray*}
    &  & | F^{T, \rho} (u^T) - F^{\infty, \rho} (u) |\\
    & \leqslant & \mathbb{E} \left| \int \rho (\llbracket \cos (\beta
    (W_{\infty} + I_{\infty} (u))) \rrbracket - \llbracket \cos (\beta (W_T +
    I_T (u))) \rrbracket) \mathd x \right| +\mathbb{E} [\| u - u^T \|^2_{L^2
    (\mathbb{R}_+ \times \mathbb{R}^2)}]
  \end{eqnarray*}
  and recall that for $T \in [0, \infty]$
  \[ \llbracket \cos (\beta (W_T + I_T (u))) \rrbracket = \llbracket \cos
     (\beta (W_T)) \rrbracket \sin (\beta I_T (u)) + \llbracket \sin (\beta
     (W_T)) \rrbracket \cos (\beta I_T (u)) . \]
  We have that if $u^T \rightarrow u$ in $\mathbb{H}_a$ then $\| I_T (u^T) -
  I_{\infty} (u) \|_{H^1} \leqslant \| I_T (u - u^T) \|_{H^1} + \| I_{T,
  \infty} (u) \|_{H^1} \rightarrow 0$ $\mathbb{P}- a.s$.
  
  Furthermore $\| I_T (u^T) - I_{\infty} (u) \|_{H^1} \leqslant \| u^T \|_{L^2
  (\mathbb{R}_+ \times \mathbb{R}^2)} + \| u \|_{L^2 (\mathbb{R}_+ \times
  \mathbb{R}^2)}$ and thereby uniformly in $L^2 (\mathbb{P})$. Finally
  recalling that $\llbracket \cos (\beta (W_T)) \rrbracket \rightarrow
  \llbracket \cos (\beta (W_{\infty})) \rrbracket$ in $L^p (\mathbb{P},
  H_{\tmop{loc}}^1 (\mathbb{R}^2))$ we can conlude that
  \[ \int \rho \llbracket \cos (\beta (W_T + I_T (u^T))) \rrbracket
     \rightarrow \int \rho \llbracket \cos (\beta (W_{\infty} + I_{\infty}
     (u))) \rrbracket \]
  $\mathbb{P}- a.s$ and is unifomly integrable, which implies the claim.
\end{proof}

\section{Variational characterization}\label{sec:char}

\begin{theorem}
  \label{thm:weighted}Let $\bar{u}^{\rho, f}$ be the minimzer of
  \begin{equation}
    F^{\rho, f} (u) =\mathbb{E} \left[ f (W_{\infty} + I_{\infty} (u)) + \int
    \rho (x) \llbracket \cos (\beta (W_{\infty} + I_{\infty} (u))) \rrbracket
    \mathd x + \frac{1}{2} \int \| u_t \|^2_{L^2} \mathd t \right]
    \label{functional-def}
  \end{equation}
  Then there exists $\gamma, \delta > 0$ such that for any $\lambda (\beta +
  \beta^2) < \delta$ and $w (x) = \exp (\gamma | x |)$, there exists $C > 0$
  such that
  \[ \mathbb{E} \left[ \int^{\infty}_0 \| w (\bar{u}^{\rho, f} -
     \bar{u}^{\rho}) \|_{L^2}^2 \mathd t \right] \leqslant C | f |_{1, \gamma}
  \]
\end{theorem}

Before proceeding to the proof we establish and EL-equation for the minimers
of {\eqref{functional-def}}, similarly to {\cite{barashkov-2021}}. \

\begin{lemma}
  \label{lemma:el-f}There exists a minimer $\bar{u}^{\rho, f}$ of
  {\eqref{functional-def}} in $\mathbb{H}_a$ and it satisfies the equation
  \begin{eqnarray*}
    &  & \mathbb{E} \left[ \beta^2 \int^{\infty}_0 \int J_t (\rho \llbracket
    \cos (\beta (W_t + I_t (\bar{u}^{\rho, f}))) \rrbracket I_t (h)) \left(
    \bar{u}_t^{\rho, f} \right) \mathd x \mathd t \right. \\
    & & + \left. \beta \int^{\infty}_0 \int J_t (\rho \llbracket \sin (\beta (W_t + I_t
    (\bar{u}^{\rho, f}))) \rrbracket) h_t \mathd x \mathd t \right]\\
    & = & \mathbb{E} \left[ \int^{\infty}_0 \int \bar{u}_t^{\rho, f} h_t
    \mathd x \mathd t \right] -\mathbb{E} \left[ \int \nabla f (W_{\infty} +
    I_{\infty} (\bar{u}^{\rho, f})) I_{\infty} (h) \mathd x \right] .
  \end{eqnarray*}
  for any $h \in \mathbb{H}_a$. 
\end{lemma}

\begin{proof}
  We first claim that for $T \in [0, \infty]$
  \[ \int \rho \llbracket \cos (\beta (W_T + I_T (u \nobracket)) \rrbracket
     \mathd x = \int^T_0 \int \rho \llbracket \cos (\beta (W_t + I_t (u
     \nobracket)) \rrbracket J_t u_t \mathd x \mathd t + \tmop{martingale} .
  \]
  Indeed by Ito's formula
  \begin{eqnarray*}
    &  & \int \rho \llbracket \cos (\beta (W_T + I_T (u \nobracket))
    \rrbracket \mathd x - \int \rho \mathd x\\
    & = & \int^T_0 \int \rho \llbracket \sin (\beta (W_t + I_t (u
    \nobracket)) \rrbracket \mathd W_t + \int^T_0 \int J_t (\rho \llbracket
    \sin (\beta (W_t + I_t (u \nobracket)) \rrbracket) u_t \mathd t.
  \end{eqnarray*}
  A priori the first term on the r.h.s might only be a local martingale but we
  can see that by Ito isometry it's quadratic variation is
  \[ \int (J_t (\rho \llbracket \sin (\beta (W_t + I_t (u \nobracket))
     \rrbracket))^2 \mathd x \leqslant t^{- 2} t^{\beta^2 / 4 \pi} \| \rho
     \|^2_{L^2} \leqslant t^{- 1 - 2 \delta} \| \rho \|^2_{L^2} \]
  wich is integrable in time. Thus we can conclude that the first part is
  indeed an martingale. From this we deduce that
  \begin{eqnarray*}
    &  & \mathbb{E} \left[ f (W_{\infty} + I_{\infty} (u)) + \lambda \int
    \rho \llbracket \cos (\beta (W_{\infty} + I_{\infty} (u \nobracket)))
    \rrbracket \mathd x + \frac{1}{2} \int \| u_t \|^2_{L^2} \mathd t
    \right] \\
    & = & \mathbb{E} \left[ f (W_{\infty} + I_{\infty} (u)) + \lambda \beta
    \int^{\infty}_0 \int J_t \left( \rho \left\llbracket \sin \left( \beta
    \left( W_t + I_t (u) \right) \right) \right\rrbracket \right) (u_t) \mathd
          x \mathd t+  \frac{1}{2} \int \| u_t \|^2_{L^2} \mathd t + \int \rho
    \mathd x \right] .
  \end{eqnarray*}
  Now assume that $u \in \mathbb{H}_a$ minimizes the functional. Then for any
  $h \in \mathbb{H}_a$ we have
  \[ \lim_{\varepsilon \rightarrow 0} \frac{1}{\varepsilon} (F (u +
     \varepsilon h) - F (u)) \geqslant 0 \qquad \lim_{\varepsilon \rightarrow
     0} \frac{1}{\varepsilon} (F (u) - F (u - \varepsilon h)) \leqslant 0. \]
  It is clear that
  \[ \frac{1}{2} \left( \int \| u_t + h_t \|^2_{L^2} \mathd t - \int \| u_t
     \|^2_{L^2} \mathd t \right) = \int \int u_t h_t \mathd x \mathd t. \]
  Now consider the limit
  \begin{eqnarray*}
    &  & \lim_{\varepsilon \rightarrow 0} \frac{1}{\varepsilon} \left(
    \mathbb{E} \left[ \int \int J_t (\rho \llbracket \sin (\beta (W_t + I_t (u
    + \varepsilon h))) \rrbracket) (u_t + \varepsilon h_t) \mathd x \mathd t
    \right. \right.\\
    &  & \left. - \left. \int \int J_t \rho \llbracket \sin (\beta (W_t + I_t
    (u))) \rrbracket (u_t) \mathd x \mathd t \right] \right)
  \end{eqnarray*}
  Recall that since $\llbracket \sin (\beta (W_t + I_t (u + \varepsilon h)))
  \rrbracket = \alpha (t) \sin (\beta (W_t + I_t (u + \varepsilon h)))$ and
  $\alpha (t) \leqslant C t^{\beta^2 / 8 \pi}$
  \begin{eqnarray*}
    &  & \frac{1}{\varepsilon} \left| \int (J_t \rho \llbracket \sin (\beta
    (W_t + I_t (u + \varepsilon h))) \rrbracket) \left( u_t + \varepsilon h_t
    \right) - J_t (\llbracket \sin (\beta (W_t + I_t (u))) \rrbracket) (u_t)
    \mathd x \right|\\
    & \leqslant & \frac{1}{\varepsilon} \left| \int J_t (\rho \llbracket \sin
    (\beta (W_t + I_t (u + \varepsilon h))) \rrbracket) (u_t) - J_t (\rho
    \llbracket \sin (\beta (W_t + I_t (u))) \rrbracket) (u_t) \mathd x
    \right|\\
    &  & + \left| \int J_t (\rho \llbracket \sin (\beta (W_t + I_t (u +
    \varepsilon h))) \rrbracket) (h_t) \mathd x \right|\\
    & \leqslant & C t^{\beta^2 / 8 \pi - 1} \| I_t (h) \|_{L^2
    (\mathbb{R}^2)} \| u_t \|_{L^2 (\mathbb{R}^2)} + C t^{\beta^2 / 8 \pi - 1}
    \left( \int \rho \mathd x \right)^{1 / 2} \| h_t \|_{L^2 (\mathbb{R}^2)} .
  \end{eqnarray*}
  Note that the last term does not depend on $\varepsilon$ and is integrable
  in time and probability, since $u, h \in \mathbb{H}_a$. By dominated
  convergence we thus have
  \begin{eqnarray*}
    &  & \lim_{\varepsilon \rightarrow 0} \frac{1}{\varepsilon} \left(
    \mathbb{E} \left[ \int^{\infty}_0 \int J_t \rho \llbracket \sin (\beta
    (W_t + I_t (u + \varepsilon h))) \rrbracket (u_t + \varepsilon h_t) \mathd
    x \mathd t \right. \right.  \\
    &  &- \left. \left. \int^{\infty}_0 \int J_t \rho \llbracket \sin (\beta (W_t + I_t (u)))
    \rrbracket (u_t) \mathd x \mathd t \right] \right)\\
    &  & =\mathbb{E} \int^{\infty}_0 \lim_{\varepsilon \rightarrow 0} \left(
    \frac{1}{\varepsilon} \int J_t \rho \llbracket \sin (\beta (W_t + I_t (u +
         \varepsilon h))) \rrbracket \left( u_t + \varepsilon h_t \right) \mathd x \right.
    \\
    &  &\left. - \int J_t \rho \llbracket \sin (\beta (W_t + I_t (u))) \rrbracket (u_t)
    \mathd x \right) \mathd t
  \end{eqnarray*}
  and so it remains to show
  \begin{eqnarray*}
    &  & \lim_{\varepsilon \rightarrow 0} \left( \frac{1}{\varepsilon} \int
    J_t \rho \llbracket \sin (\beta (W_t + I_t (u + \varepsilon h)))
    \rrbracket (u_t + \varepsilon h) \mathd x - \int J_t \rho \llbracket \sin
    (\beta (W_t + I_t (u))) \rrbracket (u_t) \mathd x \right)\\
    & = & \left( \int \frac{1}{\varepsilon} \left( J_t \rho \llbracket \sin
    (\beta (W_t + I_t (u + \varepsilon h))) \rrbracket \left( u_t +
    \varepsilon h \right) - J_t \rho \llbracket \sin (\beta (W_t + I_t (u)))
    \rrbracket (u_t) \right) \mathd x \right)
  \end{eqnarray*}
  from which the statement follows by chain rule. Now
  \begin{eqnarray*}
    &  & \frac{1}{\varepsilon} | J_t \rho \llbracket \sin (\beta (W_t + I_t
    (u + \varepsilon h))) \rrbracket (u_t + \varepsilon h_t) - J_t \rho
    \llbracket \sin (\beta (W_t + I_t (u))) \rrbracket (u_t) |\\
    & \leqslant & | J_t \rho \llbracket \sin (\beta (W_t + I_t (u +
    \varepsilon h))) \rrbracket (h_t) |  \\ &  & + \frac{1}{\varepsilon} | J_t \rho
    \llbracket \sin (\beta (W_t + I_t (u + \varepsilon h))) \rrbracket - J_t
    \rho \llbracket \sin (\beta (W_t + I_t (u))) \rrbracket (u_t) |\\
    & \leqslant & t^{\beta^2 / 8 \pi} J_t \rho | h_t | + t^{\beta^2 / 8 \pi}
    J_t \rho | I_t (h) |  | u_t |
  \end{eqnarray*}
  which is integrable in $\mathbb{R}^2$, so we can conclude by dominated
  convergence. 
\end{proof}

\begin{lemma}
  \label{lemma:decay}Let $\bar{u}^{\rho, g}$ satisfy the equation
  \[ \begin{array}{ll}
       & \mathbb{E} \left[ \lambda \beta^2 \int^{\infty}_0 \int J_t (\rho
       \llbracket \cos (\beta (W_t + I_t (\bar{u}^{\rho, g}))) \rrbracket I_t
       (h)) (\bar{u}^{\rho, g}) \mathd x \mathd t \right.\\
       & \left. + \lambda \beta \int^{\infty}_0 \int J_t (\rho \llbracket
       \sin (\beta (W_t + I_t (\bar{u}^{\rho, g}))) \rrbracket) h_t \mathd x
       \mathd t \right]\\
       = & \mathbb{E} \left[ \int^{\infty}_0 \int \bar{u}_t^{\rho, g} h_t
       \mathd x \mathd t \right] -\mathbb{E} [g (W, \bar{u}^{\rho, g}, h)] .
     \end{array} \]
  and assume that for $0 < \gamma < 2 m$ and $z \in \mathbb{R}^2$, $w (x) =
  \exp (\gamma | x - z |)$, we have
  \[ \mathbb{E} [g (W, \bar{u}^{\rho, g}, h)] \leqslant C_g \mathbb{E} [\|
     w^{- 1 / 2} h \|^2_{L^2 (\mathbb{R}_+ \times \mathbb{R}^2)}]^{1 / 2} . \]
  Then there exists a $\kappa > 0$ such that for $\lambda (\beta^2 + \beta)
  \leqslant \kappa$
  \[ \mathbb{E} [\| w^{1 / 2} (\bar{u}^{\rho, g} - \bar{u}^{\rho, 0})
     \|^2_{L^2 (\mathbb{R}_+ \times \mathbb{R}^2)}] \leqslant 2 C_g . \]
\end{lemma}

\begin{proof*}{Proof}
  Takinge differences of the EL equations we get
  \[ \mathbb{E} \left[ \int^{\infty}_0 \int (\bar{u}_t^{\rho, g} -
     \bar{u}_t^{\rho, 0}) h_t \mathd x \mathd t \right] = \Iota + \Iota \Iota
     + \Iota \Iota \Iota \]
  with
  \begin{eqnarray*}
    \Iota & = & \mathbb{E} \left[ \lambda \beta^2 \int^{\infty}_0 \int J_t
    \left( \rho \left( \llbracket \cos (\beta (W_t + I_t (\bar{u}^{\rho, g})))
    \rrbracket \left( \bar{u}_t^{\rho, g} \right) - \llbracket \cos (\beta
    (W_t + I_t (\bar{u}^{\rho}))) \rrbracket I_t (h) \right) \right)
    (\bar{u}_t^{\rho}) \mathd x \mathd t \right]\\
    & = & \mathbb{E} \left[ \lambda \beta^2 \int^{\infty}_0 \int J_t \left(
    (\rho \llbracket \cos (\beta (W_t + I_t (\bar{u}^{\rho, g}))) \rrbracket -
    \rho \llbracket \cos (\beta (W_t + I_t (\bar{u}^{\rho}))) \rrbracket) I_t
    (h) \right) (\bar{u}_t^{\rho}) \right. \\
    &  &+ \left. J_t (\rho \llbracket \cos (\beta (W_t + I_t (\bar{u}^{\rho, g})))
    \rrbracket I_t (h)) (\bar{u}_t^{\rho, g} - u_t^{\rho}) \mathd x \mathd t
    \right]\\
    & = & \Iota_a + \Iota_b\\
    \Iota \Iota & = & \lambda \beta \mathbb{E} \left[ \int^{\infty}_0 \int J_t
    (\llbracket \sin (\beta (W_t + I_t (\bar{u}^{\rho, g}))) \rrbracket -
    \llbracket \sin (\beta (W_t + I_t (\bar{u}^{\rho}))) \rrbracket) h_t
    \mathd x \mathd t \right]\\
    \Iota \Iota \Iota & = & \mathbb{E} [g (W, u^{\rho, g}, h)]
  \end{eqnarray*}
  Now setting $h = w (\bar{u}^{\rho, g} - \bar{u}^{\rho})$ we can estimate
  \begin{footnotesize}
  \begin{eqnarray*}
    &  & \frac{1}{\lambda \beta^2} | \Iota_a |\\
    & = & \left| \mathbb{E} \left[ \int^{\infty}_0 \int J_t \left( \left(
    \rho \llbracket \cos (\beta (W_t + I_t (\bar{u}^{\rho, g}))) \rrbracket -
    \rho \left\llbracket \cos \left( \beta (W_t + I_t (\bar{u}^{\rho}))
    \right) \right\rrbracket \right) I_t (w (\bar{u}^{\rho, g} -
    \bar{u}^{\rho, 0})) \right) (\bar{u}_t^{\rho}) \mathd x \mathd t \right]
    \right|\\
    & = &\left| \mathbb{E} \left[ \int^{\infty}_0 \int J_t \left( w^{1 / 2}
    \rho \llbracket \cos (\beta (W_t + I_t (\bar{u}^{\rho, g}))) \rrbracket -
    w^{1 / 2} \rho \left\llbracket \cos \left( \beta (W_t + I_t
    (\bar{u}^{\rho, 0})) \right) \right\rrbracket \right. \right. \right.
    \left. \left. \left. w^{- 1 / 2} I_t (w (\bar{u}^{\rho, g} - \bar{u}^{\rho, 0})) \right)
    \bar{u}_t^{\rho, 0} \mathd x \mathd t \right] \right|\\
    & \leqslant & \mathbb{E} \left[ \left\| \left( J_t w^{1 / 2} \rho
    \llbracket \cos (\beta (W_t + I_t (\bar{u}^{\rho, g}))) \rrbracket - J_t
    w^{1 / 2} \rho \llbracket \cos (\beta (W_t + I_t (\bar{u}^{\rho})))
    \rrbracket \right) (\bar{u}_t^{\rho}) \right\|_{L^1 (\mathbb{R}_+, L^2
    (\mathbb{R}^2))}^2 \right]^{1 / 2}\\
    &  & \times \mathbb{E} [\| w^{- 1 / 2} I_t (w (\bar{u}^{\rho, g} -
    \bar{u}^{\rho})) \|^2_{L^{\infty} (\mathbb{R}_+ L^2 (\mathbb{R}^2))}]^{1 /
    2}\\
    & \leqslant & \mathbb{E} [\| t^{- 1} (w^{1 / 2} \rho \llbracket \cos
    (\beta (W_t + I_t (\bar{u}^{\rho, g}))) \rrbracket - w^{1 / 2} \rho
    \llbracket \cos (\beta (W_t + I_t (\bar{u}^{\rho}))) \rrbracket) \|_{L^2
    (\mathbb{R}_+, L^2 (\mathbb{R}^2))}^2]^{1 / 2} \| u^{\rho} \|_{L^2
    (\mathbb{R}_+, L^{\infty} (\mathbb{P} \times \mathbb{R}^2))}\\
    &  & \times \mathbb{E} [\| w^{- 1 / 2} I_t (w (\bar{u}^{\rho, g} -
    \bar{u}^{\rho})) \|^2_{L^{\infty} (\mathbb{R}_+ L^2 (\mathbb{R}^2))}]^{1 /
    2}\\
    & \leqslant & \mathbb{E} [\| t^{\beta^2 / 8 \pi - 1} w^{1 / 2} I_t
    (\nobracket \bar{u}^{\rho, g} - \bar{u}^{\rho}) \nobracket  \|_{L^2
    (\mathbb{R}_+, L^2 (\mathbb{R}^2))}^2]^{1 / 2} \| u^{\rho} \|_{L^2
    (\mathbb{R}_+, L^{\infty} (\mathbb{P} \times \mathbb{R}^2))} \\
    &  &\times \mathbb{E} [\| w^{- 1 / 2} I_t (w (\bar{u}^{\rho, g} -
    \bar{u}^{\rho})) \|^2_{L^{\infty} (\mathbb{R}_+, L^2 (\mathbb{R}^2))}]^{1
    / 2}\\
    & \leqslant & C \| u^{\rho} \|_{L^2 (\mathbb{R}_+, L^{\infty} (\mathbb{P}
    \times \mathbb{R}^2))} \mathbb{E} [\| w^{1 / 2} I_t (\nobracket
    \bar{u}^{\rho, g} - \bar{u}^{\rho}) \nobracket  \|_{L^{\infty}
    (\mathbb{R}_+, L^2 (\mathbb{R}^2))}^2]^{1 / 2} \\
    &  &\times \mathbb{E} [\| w^{- 1 / 2} I_t (w (\bar{u}^{\rho, g} -
    \bar{u}^{\rho})) \|^2_{L^{\infty} (\mathbb{R}_+ L^2 (\mathbb{R}^2))}]^{1 /
    2}\\
    & \leqslant & C \| u^{\rho} \|_{L^2 (\mathbb{R}_+, L^{\infty} (\mathbb{P}
    \times \mathbb{R}^2))} \mathbb{E} [\| w^{1 / 2} (u^{\rho, g} - u^{\rho,
    0}) \|^2_{L^2 (\mathbb{R}_+ \times \mathbb{R}^2)}]
  \end{eqnarray*}
  \end{footnotesize}
  since we have
  \begin{eqnarray*}
    &  & \mathbb{E} [\| w^{- 1 / 2} I_t (w (\bar{u}^{\rho, f} -
    \bar{u}^{\rho})) \|^2_{L^{\infty} (\mathbb{R}_+, L^2 (\mathbb{R}^2))}]\\
    & \leqslant & \mathbb{E} [\| \nobracket w^{- 1 / 2} w (\bar{u}^{\rho, f}
    - \bar{u}^{\rho})) \|^2_{L^2 (\mathbb{R}_+ \times \mathbb{R}^2)}]\\
    & \leqslant & \mathbb{E} [\| \nobracket w^{1 / 2} (\bar{u}^{\rho, f} -
    \bar{u}^{\rho})) \|^2_{L^2 (\mathbb{R}_+ \times \mathbb{R}^2)}]
  \end{eqnarray*}
  and
  \[ \begin{array}{ll}
       & \frac{}{} \frac{1}{\lambda \beta^2} | \Iota_b |\\
       = & \mathbb{E} \left[ \int^{\infty}_0 \int J_t (\rho \llbracket \cos
       (\beta (W_t + I_t (\bar{u}^{\rho, g}))) \rrbracket I_t (w
       (\bar{u}^{\rho, g} - \bar{u}^{\rho}))) (\bar{u}_t^{\rho, g} -
       \bar{u}_t^{\rho}) \mathd x \mathd t \right]\\
       = & \mathbb{E} \left[ \int^{\infty}_0 \int w^{- 1 / 2} J_t \rho
       \left\llbracket \cos (\beta (W_t + I_t (\bar{u}^{\rho, g})))
       \right\rrbracket I_t (w (\bar{u}^{\rho, g} - \bar{u}^{\rho})) w^{1 / 2}
       (\bar{u}_t^{\rho, g} - \bar{u}_t^{\rho}) \mathd x \mathd t \right]\\
       \leqslant & \mathbb{E} [\| J_t \rho \llbracket \cos (\beta (W + I
       (\bar{u}^{\rho, g}))) \rrbracket w^{- 1 / 2} I_t (w (\bar{u}^{\rho, g}
       - \bar{u}^{\rho})) \|_{L^2 (\mathbb{R}_+, L^2 (\mathbb{R}^2))}^2]^{1 /
       2}\\
       & \times \mathbb{E} [\| w^{1 / 2} (\bar{u}^{\rho, g} - \bar{u}^{\rho})
       \|^2_{L^2 (\mathbb{R}_+, L^2 (\mathbb{R}^2))}]^{1 / 2}\\
       \leqslant & \mathbb{E} [\| t^{\beta^2 / 8 \pi - 1} w^{- 1 / 2} I_t (w
       (\bar{u}^{\rho, g} - \bar{u}^{\rho})) \|_{L^1 (\mathbb{R}_+, L^2
       (\mathbb{R}^2))}^2]^{1 / 2} \mathbb{E} [\| w^{1 / 2} (\bar{u}^{\rho, g}
       - \bar{u}^{\rho}) \|^2_{L^2 (\mathbb{R}_+, L^2 (\mathbb{R}^2))}]^{1 /
       2}\\
       \leqslant & \mathbb{E} [\| w^{1 / 2} (\bar{u}^{\rho, g} - u^{\rho})
       \|_{L^2 (\mathbb{R}_+, L^2 (\mathbb{R}^2))}^2]^{1 / 2} \mathbb{E} [\|
       w^{- 1 / 2} I_t (w (\bar{u}^{\rho, g} - \bar{u}^{\rho}))
       \|^2_{L^{\infty} (\mathbb{R}_+, L^2 (\mathbb{R}^2))}]^{1 / 2}\\
       \leqslant & \mathbb{E} [\| w^{1 / 2} (\bar{u}^{\rho, g} - u^{\rho})
       \|_{L^2 (\mathbb{R}_+, L^2 (\mathbb{R}^2))}^2] .
     \end{array} \]
  where we have used Lemma \ref{lemma:boundIL2}. So in total we get that
  \[ | \Iota | \leqslant C \lambda \beta^2 \mathbb{E} [\| w^{1 / 2} (u^{\rho,
     f} - u^{\rho}) \|^2_{L^2 (\mathbb{R}_+ \times \mathbb{R}^2)}] \]
  similarly
  \begin{eqnarray*}
    | \Iota \Iota | & = & \lambda \beta \left| \mathbb{E} \left[
    \int^{\infty}_0 \int J_t \left( \llbracket \sin (\beta (W_t + I_t
    (\bar{u}^{\rho, f}))) \rrbracket - \llbracket \sin (\beta (W_t + I_t
    (\bar{u}^{\rho}))) \rrbracket \right) w (\bar{u}^{\rho, f} -
    \bar{u}^{\rho}) \mathd x \mathd t \right] \right|\\
    & \leqslant & \lambda \beta \mathbb{E} \left[ \int \langle t \rangle^{-
    1} \left\| w^{1 / 2} \llbracket \sin (\beta (W_t + I_t (\bar{u}^{\rho,
    f}))) \rrbracket - \left\llbracket \sin \left( \beta (W_t + I_t
    (\bar{u}^{\rho})) \right) \right\rrbracket \right\|_{L^2} \| w^{1 / 2}
    (\bar{u}^{\rho, f} - \bar{u}^{\rho}) \|_{L^2} \mathd t \right]\\
    & \leqslant & \lambda \beta \mathbb{E} \left[ \int \langle t
    \rangle^{\beta^2 / 8 \pi - 1} \| w^{1 / 2} (I_t (\bar{u}^{\rho, f}) - I_t
    (\bar{u}^{\rho})) \|_{L^2 (\mathbb{R}^2)} \| w^{1 / 2} (\bar{u}^{\rho, f}
    - \bar{u}^{\rho}) \|_{L^2 (\mathbb{R}^2)} \mathd t \right]\\
    & \leqslant & \lambda \beta \mathbb{E} [\| w^{1 / 2} (I_t (\bar{u}^{\rho,
    f}) - I_t (\bar{u}^{\rho})) \|_{L_t^{\infty} L_x^2} \| w^{1 / 2}
    (\bar{u}^{\rho, f} - \bar{u}^{\rho}) \|_{L^2 (\mathbb{R}_+ \times
    \mathbb{R}^2)}]\\
    & \leqslant & C \lambda \beta \mathbb{E} [\| w^{1 / 2} (\bar{u}^{\rho, f}
    - \bar{u}^{\rho}) \|^2_{L^2 (\mathbb{R}_+ \times \mathbb{R}^2)}] .
  \end{eqnarray*}
  By assumption
  \[ | \Iota \Iota \Iota | \leqslant C_g \mathbb{E} [\| w^{- 1 / 2} w
     (\bar{u}^{\rho, f} - \bar{u}^{\rho}) \|^2_{L^2 (\mathbb{R}_+ \times
     \mathbb{R}^2)}]^{1 / 2} =\mathbb{E} [\| w^{1 / 2} (\bar{u}^{\rho, f} -
     \bar{u}^{\rho}) \|^2_{L^2 (\mathbb{R}_+ \times \mathbb{R}^2)}]^{1 / 2} \]
  All together we obtain
  \begin{eqnarray*}
    &  & \mathbb{E} [\| w^{1 / 2} (\bar{u}^{\rho, f} - \bar{u}^{\rho})
    \|^2_{L^2 (\mathbb{R}_+ \times \mathbb{R}^2)}]\\
    & = & | \Iota + \Iota \Iota + \Iota \Iota \Iota |\\
    & \leqslant & C \lambda (\beta + \beta^2) \mathbb{E} [\| w^{1 / 2}
    (\bar{u}^{\rho, f} - \bar{u}^{\rho}) \|^2_{L^2 (\mathbb{R}_+ \times
    \mathbb{R}^2)}]\\
    &  & + C_g \mathbb{E} [\| w^{1 / 2} (\bar{u}^{\rho, f} - \bar{u}^{\rho})
    \|^2_{L^2 (\mathbb{R}_+ \times \mathbb{R}^2)}]^{1 / 2}
  \end{eqnarray*}
  Provided that $C \lambda (\beta + \beta^2) < 1 / 2$ this implies
  \[ \mathbb{E} [\| w^{1 / 2} (\bar{u}^{\rho, f} - \bar{u}^{\rho}) \|^2_{L^2
     (\mathbb{R}_+ \times \mathbb{R}^2)}]^{1 / 2} \leqslant 2 C_g \]
  which is the claim.
\end{proof*}

\begin{proof*}{Proof of Theorem \ref{thm:weighted}}
  By Lemma \ref{lemma:decay} and Lemma \ref{lemma:el-f} it is sufficient to
  verify that with
  \[ \mathbb{E} \left[ \int \nabla f (W_{\infty} + I_{\infty} (u^{\rho, f}))
     I_{\infty} (h) \right] \leqslant | f |_{1, \gamma} \mathbb{E} [\| h
     \|^2_{L^2 (\mathbb{R}_+, L^{2, - \gamma})}]^{1 / 2} . \]
  However this holds since with $w (x) = \exp (2 \gamma | x |)$
  \begin{eqnarray*}
    &  & \mathbb{E} \left[ \int \nabla f (W_{\infty} + I_{\infty} (u^{\rho,
    f})) I_{\infty} (h) \right]\\
    & \leqslant & \mathbb{E} \left[ \int w \nabla f (W_{\infty} + I_{\infty}
    (u^{\rho, f})) w^{- 1} I_{\infty} (h) \right]\\
    & \leqslant & | f |_{1, \gamma} \mathbb{E} [\| w^{- 1} I_{\infty} (h)
    \|_{L^2}]\\
    & \leqslant & | f |_{1, \gamma} \mathbb{E} [\| h \|^2_{L^2 (\mathbb{R}_+,
    L^{2, - \gamma})}]^{1 / 2}
  \end{eqnarray*}
  
\end{proof*}

\begin{proposition}
  \label{prop:drift-infinite-volume}There exists an $\bar{u} \in L^2
  (\mathbb{R}_+, L^{\infty} (\mathbb{R}^2))$ such that for any $0 < \gamma <
  m$
  \[ \lim_{\rho \rightarrow 1} \mathbb{E} \left[ \int^{\infty}_0 \|
     \bar{u}_s^{\rho, 0} - \bar{u}_s \|^2_{L^{2, - \gamma}} \mathd s \right] =
     0. \]
\end{proposition}

\begin{proof}
  We first show that if $\tmop{supp} (\rho) \subseteq B (y, 1)$ $x \in
  \mathbb{R}^2$ then for $w (x) = \exp (- 2 \gamma | x - y |)$ we have if $\|
  u \|_{L^{\infty} (\mathbb{R} \times \mathbb{R}^2)} \leqslant C$ then: \
  \begin{eqnarray*}
    &  & \mathbb{E} \left[ \left| \int^{\infty}_0 \int J_t (\rho \llbracket
    \sin (W_t + I_t (u)) \rrbracket) h_t \mathd x \mathd t \right| + \left|
    \int^{\infty}_0 \int J_t (\rho \llbracket \sin (W_t + I_t (u)) \rrbracket
    I_t (h)) u_t \mathd x \mathd t \right| \right]\\
    & \leqslant & C\mathbb{E} [\| w^{- 1 / 2} h \|^2_{L^2 (\mathbb{R}_+
    \times \mathbb{R}^2)}]^{1 / 2} .
  \end{eqnarray*}
  Indeed
  \[ \begin{array}{lll}
       &  & \mathbb{E} \left[ \left| \int^{\infty}_0 \int J_t (\rho
       \llbracket \sin (W_t + I_t (u)) \rrbracket) h_t \mathd x \mathd t
       \right| + \left| \int^{\infty}_0 \int J_t (\rho \llbracket \sin (W_t +
       I_t (u)) \rrbracket I_t (h)) u_t \mathd x \mathd t \right| \right]\\
       & = & \mathbb{E} \left[ \left| \int^{\infty}_0 \int w^{1 / 2} J_t
       (\rho \llbracket \sin (W_t + I_t (u)) \rrbracket) w^{- 1 / 2} h_t
       \mathd x \mathd t \right| \right]\\
       &  & +\mathbb{E} \left[ \left| \int^{\infty}_0 \int J_t (w^{1 / 2}
       \rho \llbracket \sin (W_t + I_t (u)) \rrbracket w^{- 1 / 2} I_t (h))
       u_t \mathd x \mathd t \right| \right]\\
       & \leqslant & \mathbb{E} [\| w^{1 / 2} J_t (\rho \llbracket \sin (W_t
       + I_t (u)) \rrbracket) \|_{L^2 (\mathbb{R}_+ \times \mathbb{R}^2)} \|
       w^{- 1 / 2} h \|_{L^2 (\mathbb{R}_+ \times \mathbb{R}^2)}]\\
       &  & +\mathbb{E} [\| w^{1 / 2} J_t (\rho \llbracket \sin (W_t + I_t
       (u)) \rrbracket w^{- 1 / 2} I_t (h) ) \|_{L^2 (\mathbb{R}_+ \times
       \mathbb{R}^2)} \| w^{- 1 / 2} h \|_{L^2 (\mathbb{R}_+ \times
       \mathbb{R}^2)}]\\
       & \leqslant & \mathbb{E} [\| t^{\beta^2 / 8 \pi - 1}  \| w^{1 / 2}
       \rho \|_{L^2 (\mathbb{R}^2)}  \|_{L^2 (\mathbb{R}_+)} \| w^{- 1 / 2} h
       \|_{L^2 (\mathbb{R}_+ \times \mathbb{R}^2)}]\\
       &  & +\mathbb{E} [\| t^{\beta^2 / 8 \pi - 1}  \| w^{1 / 2} \rho
       \|_{L^2 (\mathbb{R}^2)}  \|_{L^2 (\mathbb{R}_+)} \| w^{- 1 / 2} I_t (h)
       \|_{L^{\infty} (\mathbb{R}_+, L^2 (\mathbb{R}^2))} \| u \|_{L^{\infty}
       (\mathbb{R}_+ \times \mathbb{R}^2)}]\\
       & \leqslant & C\mathbb{E} [\| w^{- 1 / 2} h \|^2_{L^2 (\mathbb{R}_+
       \times \mathbb{R}^2)}]^{1 / 2} .
     \end{array} \]
  Now suppose that $| x | \geqslant N$, $(\rho_1 - \rho_2) \subset B (x, 1)$
  and applying Lemma \ref{lemma:decay} we get with $w (y) = \exp (\gamma | x -
  y |)$
  \begin{eqnarray*}
    \mathbb{E} [\| w^{1 / 2} (\bar{u}^{\rho_1, 0} - \bar{u}^{\rho_2, 0})
    \|^2_{L^2 (\mathbb{R}_+ \times \mathbb{R}^2)}]^{1 / 2} & \leqslant & C
  \end{eqnarray*}
  which implies by Lemma \ref{lemma:decay1} from Appendix \ref{app:weighted}
  that
  \begin{equation}
    \mathbb{E} [\| (\bar{u}^{\rho_1, 0} - \bar{u}^{\rho_2, 0}) \|^2_{L^2
    (\mathbb{R}_+, L^{2, - \gamma})}]^{1 / 2} \leqslant C \exp (- \gamma | x
    |) . \label{eq:decay-cutoff}
  \end{equation}
  Now suppose that that $\rho_1, \rho_2 = 1$ on $B (0, N)$. We can depompose
  $\rho_2 - \rho_1 = \sum_{i \in \mathbb{Z}^2} \chi_i (\rho_2 - \rho_1)
  \backassign \sum_{i \in \mathbb{Z}^2} \rho_i$ where $\chi_i$ is a partition
  of unity with $\tmop{supp} \chi_i \subset B (i, 2)$. Applying estimate
  {\eqref{eq:decay-cutoff}} iterativly we get
  \begin{eqnarray*}
    \mathbb{E} [\| (\bar{u}^{\rho_2, 0} - \bar{u}^{\rho_1, 0}) \|^2_{L^2
    (\mathbb{R}_+, L^{2, - \gamma})}]^{1 / 2} & \leqslant & \sum_{i \in
    \mathbb{Z}^2, | i | \geqslant N} \mathbb{E} [\|  (\bar{u}^{\rho_1 +
    \rho_i, 0} - \bar{u}^{\rho_1, 0}) \|^2_{L^2 (\mathbb{R}_+, L^{2, -
    \gamma})}]^{1 / 2}\\
    & \leqslant & C \sum_{i \in \mathbb{Z}^2, | i | \geqslant N} \exp (-
    \gamma | i |)\\
    & \leqslant & C \exp (- \gamma N) .
  \end{eqnarray*}
  This shows that $u^{\rho}$ is a Couchy-sequence in $L^2 (\mathbb{P}, L^2
  (\mathbb{R}_+, L^{2, - \gamma}))$ which implies that it converges in this
  space. Since $u^{\rho} \in \mathbb{H}_a$ the limit is also adapted to $W_t$.
  \ 
\end{proof}

\begin{theorem}
  \label{thm:characterization}Define $\mathbb{D}^f$ to be the space
  \[ \mathbb{D}^f = \left\{ v \in \mathbb{H}_a : \mathbb{E} \left[
     \int^{\infty}_0 \int \| v \|^2_{L^{2, \gamma} (\mathbb{R}^2)} \mathd x
     \mathd t \right] \leqslant C \right\} . \]
  Then for $C$ large enough
  \[ \lim_{\rho \rightarrow 1} (W^{\rho} (f) - W^{\rho} (0)) = \inf_{v \in
     \mathbb{D}} \bar{F}^f (v) \]
  where
  \begin{align*} \bar{F}^f (v) =& \mathbb{E} \left[ f (W_{\infty} + I_{\infty} (\bar{u}) +
     I_{\infty} (v)) \right. \\ & + \lambda \int \left( \llbracket \cos (\beta (W_{\infty} +
     I_{\infty} (\bar{u}) + I_{\infty} (v))) \rrbracket - \llbracket \cos
     (\beta (W_{\infty} + I_{\infty} (\bar{u}))) \rrbracket \right) \mathd x \\ & \left.  +
     \int^{\infty}_0 \int \bar{u}_t v_t \mathd x \mathd t + \frac{1}{2}
    \int^{\infty}_0 \| v_t \|^2_{L^2} \mathd t \right]
  \end{align*}
  and \={u} has beeen introduced in Proposition
  \ref{prop:drift-infinite-volume}.
\end{theorem}

\begin{proof}
  We have
  \[ (W^{\rho} (f) - W^{\rho} (0)) = \inf_{u \in \mathbb{H}_a} (F^{f, \rho}
     (u) - F^{0, \rho} (u^{\rho})) \backassign \inf_{v \in \mathbb{H}_a}
     \bar{F}^{f, \rho} (v) \]
  where
  \[ \bar{F}^{f, \rho} (v) = F^{f, \rho} (u^{\rho} + v) - F^{0, \rho}
     (u^{\rho}) . \]
  We can restrict the functional on the the space $\mathbb{D}^f$ without
  changing the infimum by Theorem \ref{thm:weighted}. We now claim that
  $\bar{F}^f (v) - \bar{F}^{\rho, f} (v)$ goes to $0$ uniformly on
  $\mathbb{D}^f$. Indeed we can estimate
  \begin{eqnarray*}
    &  & \bar{F}^f (v) - \bar{F}^{\rho, f} (v)\\
    & = & \mathbb{E} [f (W_{\infty} + I_{\infty} (v) + I_{\infty} (\bar{u}))
    - f (W_{\infty} + I_{\infty} (v) + I_{\infty} (\bar{u}^{\rho}))
    \nobracket\\
    &  & + \lambda \int \rho \llbracket \cos (\beta W_{\infty}) \rrbracket
    \left( (\cos (\beta (I_{\infty} (v) + I_{\infty} (\bar{u}))) - \cos (\beta
    I_{\infty} (\bar{u})) \nobracket - (\nobracket \cos (\beta (I_{\infty} (v)
    + I_{\infty} (\bar{u}^{\rho}))) - \cos (\beta I_{\infty} (\bar{u}^{\rho}))
    \right)\\
    &  & + \lambda \int \rho \llbracket \sin (\beta W_{\infty}) \rrbracket
    \left( (\sin (\beta (I_{\infty} (v) + I_{\infty} (\bar{u}))) - \sin (\beta
    I_{\infty} (\bar{u}))) - (\nobracket \sin (\beta (I_{\infty} (v) +
    I_{\infty} (\bar{u}^{\rho}))) - \sin (\beta I_{\infty} (\bar{u}^{\rho})
    \nobracket \right)\\
    &  & + \lambda \int (1 - \rho) \llbracket \cos (\beta W_{\infty})
    \rrbracket (\cos (\beta (I_{\infty} (v) + I_{\infty} (\bar{u}))) - \cos
    (\beta I_{\infty} (\bar{u})))\\
    &  & + \lambda \int (1 - \rho) \llbracket \sin (\beta W_{\infty})
    \rrbracket (\sin (\beta (I_{\infty} (v) + I_{\infty} (\bar{u}))) - \sin
    (\beta I_{\infty} (\bar{u})))\\
    &  & + \int^{\infty}_0 \int v_t (\bar{u}_t - \bar{u}_t^{\rho}) \mathd x
    \mathd t\\
    &  & + \left. \int^{\infty}_0 \| v_t \|^2_{L^2 (\mathbb{R}^2)} \mathd t
    \right]
  \end{eqnarray*}
  By Interpolation with $L^{\infty}$, for $q$ close enough to $1$:
  \begin{eqnarray*}
    &  & \left\| \left( (\cos (\beta (I_{\infty} (v) + I_{\infty} (\bar{u})))
    - \cos (\beta I_{\infty} (\bar{u}))) - \cos (\beta (I_{\infty} (v) +
    I_{\infty} (\bar{u}^{\rho}))) - \cos (\beta I_{\infty} (\bar{u}^{\rho}))
    \right) \right\|_{B_{q, q}^{1 - 2\delta_1} (\langle x \rangle^k)}\\
    & \leqslant & 4 \beta \int^1_0 \| (\sin (\theta \beta I_{\infty} (v) +
    \beta I_{\infty} (\bar{u})) \nobracket - \sin (\theta \beta I_{\infty} (v)
    + \beta I_{\infty} (\bar{u}^{\rho})) I_{\infty} (v) \|^{1 -
    \delta_1}_{W^{1, 1, \gamma}} \mathd \theta\\
    & \leqslant & C \| \nobracket I_{\infty} (\bar{u}) - I_{\infty}
    (\bar{u}^{\rho}) | |^{1 - \delta_1}_{H^{1, - \gamma}} \| I_{\infty} (v)
    \|^{1 - \delta_1}_{H^{1, 2 \gamma}}\\
    &  & + C \| I_{\infty} (\bar{u}) - I_{\infty} (\bar{u}^{\rho})
    \|^{\delta_2 (1 - \delta_1)}_{L^{2, - \gamma}} \| I_{\infty} (v) \|^{2 (1
    - \delta_1)}_{H^{1, 2 \gamma}}
  \end{eqnarray*}
  Where we have used the embedding $W^{1,1,\gamma}\mapsto B^{1-\delta}_{1,1\/}(\langle x \rangle ^{2k})$ and subsequent interopolation with $L^{\infty}.$
 We have  also applied applied Lemma \ref{lemma:cosine-sobolev},
  using that $\| I (\bar{u}) \|_{W^{1, \infty}} \leqslant C$ from Theorem
  \ref{thm:L-infty} and Lemma \ref{lemma:boundWinfty}. It is clear that
  \[ \| I_{\infty} (\bar{u}^{\rho}) - I_{\infty} (\bar{u}) \|_{L^{2, -
     \gamma}} \leqslant \| I_{\infty} (\bar{u}^{\rho}) - I_{\infty} (\bar{u})
     \|^{1 - \delta}_{L^{\infty}} \| I_{\infty} (\bar{u}^{\rho}) - I_{\infty}
     (\bar{u}) \|^{\delta}_{L^{2, - \gamma}} . \]
  We can then use this estimate to obtain for $p$ large enough such that $1 /
  p + 1 / q = 1$
  \begin{eqnarray*}
    &  & \lambda \mathbb{E} \left[ \int \rho \llbracket \cos (\beta
    W_{\infty}) \rrbracket \left( \left( \cos (\beta (I_{\infty} (v) +
    I_{\infty} (\bar{u}))) - \cos \left( \beta I_{\infty} (\bar{u}) \right)
    \right) - \cos (\beta (I_{\infty} (v) + I_{\infty} (\bar{u}^{\rho}))) -
    \cos (\beta I_{\infty} (u^{\rho})) \right) \right]\\
    & \leqslant & C\mathbb{E} [\| \llbracket \cos (\beta W_{\infty})
    \rrbracket \|^p_{B_{p, p}^{- 1 + \delta} (\langle x \rangle^{- k})}]^{1 /
    p} \\ &  & \times \mathbb{E} \left[ \left( \| \nobracket I_{\infty} (\bar{u}) -
    I_{\infty} (\bar{u}^{\rho}) | |^{q (1 - \delta_1)}_{H^{1, - \gamma}} \|
    I_{\infty} (v) \|^{q (1 - \delta_1)}_{H^{1, 2 \gamma}} + \| I_{\infty}
    (\bar{u}) - I_{\infty} (\bar{u}^{\rho}) \|^{q \delta_2 (1 -
    \delta_1)}_{L^{2, - \gamma}} \| I_{\infty} (v) \|^{q 2 (1 -
    \delta_1)}_{H^{1, 2 \gamma}} \right) \right]^{1 / q}\\
    & \leqslant & \mathbb{E} [\| \llbracket \cos (\beta W_{\infty})
    \rrbracket \|^p_{B_{p, p}^{- 1 + \delta} (\langle x \rangle^{- k})}]^{1 /
    p} \\
    &  &\times \mathbb{E} [(\| \nobracket I_{\infty} (\bar{u}) - I_{\infty}
    (\bar{u}^{\rho}) | |^{2 q (1 - \delta_1)}_{H^{1, - \gamma}})]^{1 / 2 q}
    \mathbb{E} [\| I_{\infty} (v) \|^{2 q (1 - \delta_1)}_{H^{1, 2
    \gamma}}]^{1 / 2 q} \\
    &  &+\mathbb{E} [\| \nobracket I_{\infty} (\bar{u}) - I_{\infty}
    (\bar{u}^{\rho}) | |^2_{H^{1, - \gamma}}]^{1 / q - (1 - \delta_1)}
    \mathbb{E} [\| I_{\infty} (v) \|^2_{H^{1, 2 \gamma}}]^{(1 - \delta_1)},
  \end{eqnarray*}
  provided that we choose $q < 1 / (1 - \delta_1)$ and $\delta_2 = 2 (1 - q (1
  - \delta_1)) / q (1 - \delta_1)$. Now for $v \in \mathbb{D}^f$ the last line
  is bounded by
  \[ C (\mathbb{E} [(\| \nobracket I_{\infty} (\bar{u}) - I_{\infty}
     (\bar{u}^{\rho}) | |^{2 q (1 - \delta_1)}_{H^{1, - \gamma}})]^{1 / 2 q}
     +\mathbb{E} [\| \nobracket I_{\infty} (\bar{u}) - I_{\infty}
     (\bar{u}^{\rho}) | |^2_{H^{1, - \gamma}}]^{1 / q - (1 - \delta_1)}), \]
  which goes to $0$. We can proceed analogously for the sinus term. To
  estimate
  \begin{eqnarray*}
    &  & \beta \int (1 - \rho) \llbracket \cos (\beta W_{\infty}) \rrbracket
    (\cos (\beta (I_{\infty} (v) + I_{\infty} (\bar{u}))) - \cos (\beta
    I_{\infty} (\bar{u})))\\
    & = & \beta \int^1_0 \int (1 - \rho) \llbracket \cos (\beta W_{\infty})
    \rrbracket (\sin (\beta \theta I_{\infty} (v) + \beta I_{\infty}
    (\bar{u})) \nobracket I_{\infty} (v) \mathd \theta
  \end{eqnarray*}
  it is not hard to see that that $\| (1 - \rho) f \|_{W^{1, 1} (\langle x
  \rangle^k)} \leqslant N^{- k / 2} \| f \|_{W^{1, 1} (\langle x \rangle^{k /
  2})}$, so interpolating between $W^{1, 1, \gamma / 2}$ and $L^{\infty}$ we
  have
  \begin{eqnarray*}
    & \leqslant & \mathbb{E} [\| \llbracket \cos (\beta W_{\infty})
    \rrbracket \|^p_{B_{p, p}^{- 1 + \delta} (\langle x \rangle^{- k})}]^{1 /
    p}\\
    &  & \times \mathbb{E} [\| (1 - \rho) ((\sin (\beta (\theta I_{\infty}
    (v) + I_{\infty} (\bar{u}))) \nobracket I_{\infty} (v)) \|^{(1 - \delta)
    q}_{W^{1, 1, \gamma / 2}}]^{1 / q}\\
    & \leqslant & N^{- \gamma / 2} \mathbb{E} [\| \llbracket \cos
    (W_{\infty}) \rrbracket \|^p_{B_{p, p}^{- 1 + \delta} (\langle x
    \rangle^{- k})}]^{1 / p}\\
    &  & \times \mathbb{E} [\| ((\sin (\beta (\theta I_{\infty} (v) +
    I_{\infty} (\bar{u}))) \nobracket I_{\infty} (v)) \|^{(1 - \delta)
    q}_{W^{1, 1, \gamma}}]^{1 / q} .
  \end{eqnarray*}
  Now
  \[ \mathbb{E} [\| ((\sin (\beta (\theta I_{\infty} (v) + I_{\infty}
     (\bar{u}))) \nobracket I_{\infty} (v)) \|^{1 - \delta}_{W^{1, 1,
     \gamma}}] \]
  can be estimated analogously to the above computations. Clearly
  \[ \int^{\infty}_0 \int v (\bar{u} - u^{\rho}) \mathd x \mathd t \leqslant
     \mathbb{E} [\| v \|^2_{L^2 (\mathbb{R}_+, L^{2, \gamma})}]^{1 / 2}
     \mathbb{E} [\| \bar{u} - \bar{u}^{\rho} \|^2_{L^2 (\mathbb{R}^2, L^{2, -
     \gamma})}]^{1 / 2} \]
  Finally by definition of $f$
  \[ \mathbb{E} [| f (W_{\infty} + I_{\infty} (v) + I_{\infty} (\bar{u})) - f
     (W_{\infty} + I_{\infty} (v) + I_{\infty} (\bar{u}^{\rho})) |] \leqslant
     \mathbb{E} [\| \bar{u} - \bar{u}^{\rho} \|^2_{L^2 (\mathbb{R}^2, L^{2, -
     \gamma})}]^{1 / 2} \]
  which allows us to conclude.
\end{proof}

\begin{lemma}
  \label{lemma:cosine-sobolev}Assume that $\| f^1 \|_{W^{1, \infty}} + \| f^2
  \|_{W^{1, \infty}} \leqslant C$.\quad Then
  \begin{eqnarray*}
    &  & \| ((\cos (f^1 + g) \nobracket - \cos (f^2 + g) g) \|_{W^{1, 1,
    \gamma}}\\
    & \leqslant & C (\| \nobracket f^1 - f^2 | |_{H^{1, - \gamma}} \| g
    \|_{H^{1, 2 \gamma}} + \| f^1 - f^2 \|^{\delta}_{L^{2, - \gamma}} \| g
    \|^2_{H^{1, 2 \gamma}})
  \end{eqnarray*}
\end{lemma}

\begin{proof}
  Set $w (x) = \exp (\gamma x)$. Then with $1 / p + 1 / q + 1 / 2 = 1$ and $q$
  close enough to $2$ we have \
  
  $\begin{array}{lll}
    &  & \| \nabla ((\cos (f^1 + g) \nobracket - \cos (f^2 + g) g) \|_{L^{1,
    1, \gamma}}\\
    & \leqslant & \left| \int_{\mathbb{R}^2} w (x) (\cos (f^1 + g) - \cos
    (f^2 + g)) \nabla g \mathd x \right| + \left| \int_{\mathbb{R}^2} w (x)
    (\cos (f^1 + g) - \cos (f^2 + g)) g \nabla g \mathd x \right|\\
    &  & + \left| \int_{\mathbb{R}^2} w (x) (\cos (f^1 + g) - \cos (f^2 + g))
    \nabla f^1 g \mathd x \right| + \left| \int_{\mathbb{R}^2} w (x) (\cos
    (f^1 + g)) (\nabla f^1 - \nabla f^2) g \mathd x \right|\\
    & \leqslant & \int_{\mathbb{R}^2} w (x) | f^1 - f^2 | | \nabla g | \mathd
    x + \int_{\mathbb{R}^2} w (x) | f^1 - f^2 | | g | | \nabla g | \mathd x +
    \int_{\mathbb{R}^2} w (x) | f^1 - f^2 | | \nabla f^1 | | g | \mathd x\\
    &  & + \int_{\mathbb{R}^2} w (x) | (\nabla f^1 - \nabla f^2) | | g |
    \mathd x\\
    & \leqslant & \| \nabla g \|_{L^{2, 2 \gamma}} \| \nobracket \nabla f^1 -
    \nabla f^2 | |_{L^{2, - \gamma}} + \| f^1 - f^2 \|_{L^{p, p, - \gamma}} \|
    g \|_{L^q} \| \nabla g \|_{L^{2, 2 \gamma}} +\\
    &  & \| \nabla f^1 \|_{L^{\infty}} \| f^1 - f^2 \|_{L^{2, - \gamma}} \| g
    \|_{L^{2, 2 \gamma}} + \| \nobracket \nabla f^1 - \nabla f^2 | |_{L^{2, -
    \gamma}} \| g \|_{L^{2, 2 \gamma}}
  \end{array}$
  
  Now using the Sobolev embedding
  \[ \| g \|_{L^q} \leqslant \| g \|_{H^1} \leqslant \| g \|_{H^{1, 2
     \gamma}} \]
  we have
  \begin{eqnarray*}
    &  & \| \nabla g \|_{L^{2, 2 \gamma}} \| \nobracket \nabla f^1 - \nabla
    f^2 | |_{L^{2, - \gamma}} + \| f^1 - f^2 \|_{L^{p, p, - \gamma}} \| g
    \|_{L^q} \| \nabla g \|_{L^{2, 2 \gamma}} +\\
    &  & \| \nabla f^1 \|_{L^{\infty}} \| f^1 - f^2 \|_{L^{2, - \gamma}} \| g
    \|_{L^{2, 2 \gamma}} + \| \nobracket \nabla f^1 - \nabla f^2 | |_{L^{2, -
    \gamma}} \| g \|_{L^{2, 2 \gamma}}\\
    & \leqslant & \| \nobracket \nabla f^1 - \nabla f^2 | |_{L^{2, - \gamma}}
    (\| g \|_{L^{2, 2 \gamma}} + \| \nabla g \|_{L^{2, 2 \gamma}}) + \| \nabla
    f^1 \|_{L^{\infty}} \| f^1 - f^2 \|_{L^{- \gamma}} \| g \|_{L^{2
    \gamma}}\\
    &  & + \| f^1 - f^2 \|^{1 - \delta}_{L^{\infty}} \| f^1 - f^2
    \|^{\delta}_{L^{2, - \gamma}} \| g \|^2_{H^{1, 2 \gamma}}\\
    & \leqslant & C (\| \nobracket f^1 - f^2 | |_{H^{1, - \gamma}} \| g
    \|_{H^{1, 2 \gamma}} + \| f^1 - f^2 \|^{\delta}_{L^{2, - \gamma}} \| g
    \|^2_{H^{1, 2 \gamma}})
  \end{eqnarray*}
  where in the last line we have applied the assumption $\| f^1 \|_{W^{1,
  \infty}} + \| f^2 \|_{W^{1, \infty}} \leqslant C$.
  
  Now using that
  \[ \| ((\cos (f^1 + g) \nobracket - \cos (f^2 + g) g) \|_{L^{1, \gamma}}
     \leqslant \| f^1 - f^2 \|_{L^{2, - \gamma}} \| g \|_{L^{2, 2 \gamma}} \]
  we can conclude. 
\end{proof}

\section{Large deviations }\label{sec:LD}

In this section we want to discuss a Laplace principle for the Sine-Gordon
measure in the ``semiclasssical limit'' as described in the introduction. \ We
introduce the family $\nu_{\tmop{SG}, \hbar}^{T, \rho}$ of measures given by
\begin{equation}
  \int_{\CS' (\mathbb{R}^2)} g (\phi) \nu_{\tmop{SG}, \hbar}^{T, \rho} (\mathd
  \phi) = \frac{\mathbb{E} \left[ g (\hbar^{1 / 2} W_T) e^{-
  \frac{\lambda}{\varepsilon} V_{\hbar}^{T, \rho} (\hbar^{1 / 2} W_T)}
  \right]}{Z^{T, \rho}_{\hbar}},
\end{equation}
\[ \  \]
where similarly as above
\[ V_{\hbar}^{\rho, T} (\varphi) \assign \lambda \alpha^{\hbar} (T)
   \int_{\mathbb{R}^2} \cos (\beta \varphi (x)) \mathd x \qquad Z^{T,
   \rho}_{\hbar} \assign \mathbb{E} [e^{- V_{\hbar}^{\rho, T} (W_{0, T})}] \]
for any bounded measurable $g : H^{- 1} (\langle x \rangle^{- n}) \rightarrow
\mathbb{R}$. Here $\alpha^{\hbar} (T) = e^{\frac{\beta^2}{2} \hbar K_T (0)}$
and $\alpha^{\hbar} (T) \cos (\hbar^{1 / 2} \beta W_T)$ enjoys the same
properties as $\alpha (T) \cos (\beta W_T) .$ It will also be convenient to
introduce the unnormalized measures $\tilde{\nu}_{\tmop{SG}, \hbar}^{T, \rho}
= Z^{T, \rho}_{\hbar} \nu_{\tmop{SG}, \hbar}^{T, \rho}$.

Note that this corresponds (modulo a normalization constant) to the measure
heuristically defined by
\[ e^{- \frac{1}{\hbar} \int_{\mathbb{R}^2} \lambda \alpha^{\hbar} (T) \cos
   (\beta \varphi (x)) + \frac{1}{2} m^2 \varphi (x)^2 + \frac{1}{2} | \nabla
   \varphi (x) |^2 \mathd x} \mathd \varphi . \]
Our goal is now to show that $\nu$ given as the weak limit of $\nu_{\tmop{SG},
\hbar}^{T, \rho}$ as $T \rightarrow \infty, \rho \rightarrow 1$ satisfies a
Laplace principle as $\hbar \rightarrow 0$. We recall the definition of the
Laplace principle.

\begin{definition}
  A sequence of Borel measures $\nu_{\varepsilon}$ on a metric space $S$
  satisfies the Laplace principle with rate function $I$ if for any continuous
  bounded function $f : S \rightarrow \mathbb{R}$
  \[ - \lim_{\varepsilon \rightarrow 0} \varepsilon \log \int e^{-
     \frac{1}{\varepsilon} f (x)} \nu_{\varepsilon} (\mathd x) = \inf_{x \in
     S} \{ f (x) + I (x) \}_{} . \]
\end{definition}

\begin{definition}
  For a metric space $S$ and let $I : S \rightarrow \mathbb{R}$ be a rate
  function. A set $D \subseteq C (S)$ is called rate function determining if
  any exponentially tight sequence $\nu_{\varepsilon}$ of measures on $S$ such
  that
  \[ - \lim_{\varepsilon \rightarrow 0} \varepsilon \log \int e^{-
     \frac{1}{\varepsilon} f} \mathd \nu_{\varepsilon} = \inf_{x \in S} \{ f
     (x) + I (x) \}, \]
  for all $f \in D$ satisfies a large deviations principle with rate function
  $I$. \ 
\end{definition}

\begin{lemma}
  \label{lemma:rate-function-determing}Assume that $D \subseteq C (S)$ is
  bounded below, i.e $f \geqslant - C$ for any $f \in D$ with $C$ independent
  of $f$. Furthermore assume that $D$ isolates points i.e for each compact set
  $K \subseteq S, x \in S$ and $\varepsilon > 0$ there exists $f \in D$ such
  that
  \begin{itemize}
    \item $| f (x) | < \varepsilon$
    
    \item $\inf_{y \in K} f (y) \geqslant 0$
    
    \item $\inf_{y \in K \cap B^c (x, \varepsilon)} f (y) \geqslant
    \varepsilon^{- 1}$
  \end{itemize}
  Then $D$ is rate function determining. 
\end{lemma}

For a proof see {\cite{Feng_2015}} Proposition 3.20.

\begin{lemma}
  \label{lemma:class-rate-function-specific}Let $S = H^{- 1} (\langle x
  \rangle^{- n})$ for any $\gamma > 0$ Then
  \[ D = C^2 (L^2 (\mathbb{R}^2), \mathbb{R}_+) \cap C (H^{- 1} (\langle x
     \rangle^{- n})) \cap \{ | f |_{1, 2, m} < \infty \} \cap \{ f \geqslant 0
     \} \]
  is rate function determining.
\end{lemma}

\begin{proof}
  We want to verify the assumptions of Lemma
  \ref{lemma:rate-function-determing}: By translating it is enough to verify
  the assumptions for $x = 0 \in H^{- 1} (\langle x \rangle^{- n})$.
  Furthermore we can assume that $K \subseteq B (0, N)$ for some $N > 0$. Now
  choose $\chi \in C^{\infty}_c (\mathbb{R}, \mathbb{R}_+)$ such that $\chi
  (0) = 0$ and $\chi (y) \geqslant \varepsilon^{- 1}$ if $N^2 \geqslant | y
  |^2 > \varepsilon$. $f (\varphi) = \chi (\| \varphi \|^2_{H^{- 1, - m}})$
  satisfies the requirement of Lemma \ref{lemma:rate-function-determing}.
  Clearly $f \in C^2 (L^2 (\mathbb{R}^2), \mathbb{R}_+) \cap C (H^{- 1}
  (\langle x \rangle^{- n}))$, furthermore
  \[ \nabla f (\varphi) = 2 \chi' (\| \varphi \|^2_{H^{- 1, - m}}) (w (1 -
     \Delta)^{- 1} w \varphi) \]
  where $w (y) = \exp (- m y) .$ This implies that $| f |_{1, 2, m} < \infty$
  since
  \[ \| w (1 - \Delta)^{- 1} w \varphi \|_{L^{2, m}} \leqslant \| (1 -
     \Delta)^{- 1} w \varphi \|_{L^2} \leqslant \| \varphi \|_{H^{- 1, - m}} .
  \]
\end{proof}

From the Boue-Dupuis formula we obtain

\begin{proposition}
  {\tmdummy}
  
  \begin{equation}
    \begin{array}{ll}
      & \hbar \log \int e^{- \frac{1}{\hbar} f} \mathd \nu_{\tmop{SG},
      \hbar}^{T, \rho} (\mathd \phi)\\
      = & \inf_{u \in \mathbb{H}_a} \mathbb{E} \left[ f (\hbar^{1 / 2}
      W_{\infty} + I_{\infty} (u)) + \lambda \int \rho (x) \llbracket \cos
      (\hbar^{1 / 2} \beta W_{\infty} + \beta I_{\infty} (u)) \rrbracket
      \mathd x + \frac{1}{2} \int^{\infty}_0 \| u_t \|^2_{L^2 (\mathbb{R}^2)}
      \mathd t \right]
    \end{array} \label{eq:LD-finite-vol}
  \end{equation}
  
\end{proposition}

Repeating the proof of Proposition \ref{prop:drift-infinite-volume} word for
word we get:

\begin{proposition}
  \label{prop:uhbar}Let $\bar{u}^{\hbar, \rho}$ be the minimzer of
  \[ F^{\hbar, \rho} (u) =\mathbb{E} \left[ \lambda \int \rho (x) \llbracket
     \cos (\hbar^{1 / 2} \beta W_{\infty} + \beta I_{\infty} (u)) \rrbracket
     \mathd x + \frac{1}{2} \int^{\infty}_0 \| u_t \|^2_{L^2 (\mathbb{R}^2)}
     \mathd t \right] . \]
  Then $\bar{u}^{\hbar, \rho}$ converges in $L^2 (\mathbb{P}, L^2
  (\mathbb{R}_+, L^{2, - \gamma}))$, we denote the limit by $\bar{u}^{\hbar}$,
  more precisely we have
  \[ \lim_{\rho \rightarrow 1} \sup_{\hbar \leqslant 1} \mathbb{E} \left[
     \int^{\infty}_0 \| \bar{u}^{\hbar}_t - \bar{u}_t^{\hbar, \rho}
     \|^2_{L^{2, - \gamma}} \mathd t \right] = 0. \label{eq:cov-uh} \]
\end{proposition}

\begin{notation}
  Denote
  \begin{eqnarray*}
    &  & G_{\hbar}^f (v)\\
    & = & \mathbb{E} \left[ f (\hbar W_{\infty} + I_{\infty}
    (\bar{u}^{\hbar}) + I_{\infty} (v)) + \lambda \int \left\llbracket \cos
    \left( \beta W_{\infty} + \beta I_{\infty} (\bar{u}^{\hbar}) + \beta
    I_{\infty} (v) \right) \right\rrbracket - \llbracket \cos (\beta
    W_{\infty} + \beta I_{\infty} (\bar{u}^{\hbar})) \rrbracket \right.\\
    &  & \int \int \bar{u}_t^{\rho} v_t \mathd t \mathd x \left. +
    \frac{1}{2} \int^{\infty}_0 \| v_t \|^2_{L^2 (\mathbb{R}^2)} \mathd t
    \right]
  \end{eqnarray*}
\end{notation}

In complete analogy with Theorem \ref{thm:characterization} we can deduce the
following proposition.

\begin{proposition}
  {\tmdummy}
  
  \begin{eqnarray*}
    & \hbar \log \int e^{- \frac{1}{\hbar} f} \mathd \nu_{\tmop{SG},
    \hbar}^{T, \rho} (\mathd \phi) & \backassign \inf_{v \in \mathbb{D}^f}
    G_{\hbar}^f (v)
  \end{eqnarray*}
\end{proposition}

\begin{proposition}
  With the notation of Proposition \ref{prop:uhbar} we have:
  \[ \lim_{\hbar \rightarrow 0} \mathbb{E} \left[ \int^{\infty}_0 \|
     \bar{u}^{\hbar}_t \|^2_{L^{2, - \gamma}} \mathd t \right] = 0. \]
\end{proposition}

\begin{proof}
  Keeping in mind {\eqref{eq:cov-uh}} it is suffient to show that
  \[ \lim_{\hbar \rightarrow 0} \mathbb{E} \left[ \int^{\infty}_0 \|
     \bar{u}^{\hbar, \rho}_t \|^2_{L^2} \mathd t \right] = 0. \]
  By a simple modification of Lemma \ref{lemma:el-f} we can assume that
  $\bar{u}^{\hbar, \rho}$ satisfies, for any $h \in \mathbb{H}_a$,
  \[ \mathbb{E} \left[ \lambda \beta \int \rho (x) \llbracket \sin (\hbar^{1
     / 2} \beta W_{\infty} + \beta I_{\infty} (\bar{u}^{\hbar, \rho}))
     \rrbracket I_{\infty} (h) \mathd x + \int^{\infty}_0 \int
     \bar{u}_t^{\hbar, \rho} h_t \mathd x \right] = 0. \]
  By choosing $h = \bar{u}^{\hbar, \rho}$ we get
  \[ \mathbb{E} [\| \bar{u}^{\hbar, \rho} \|^2_{L^2 (\mathbb{R}_+, L^2)}] =
     -\mathbb{E} \left[ \lambda \beta \int \rho (x) \llbracket \sin (\hbar^{1
     / 2} \beta W_{\infty} + \beta I_{\infty} (\bar{u}^{\hbar, \rho}))
     \rrbracket I_{\infty} ( \bar{u}^{\hbar, \rho}) \mathd x \right] . \]
  Expanding we get
  \begin{eqnarray*}
    \llbracket \sin (\hbar^{1 / 2} \beta W_{\infty} + \beta I_{\infty}
    (\bar{u}^{\hbar, \rho})) \rrbracket & = & \llbracket \sin (\hbar^{1 / 2}
    \beta W_{\infty}) \rrbracket \cos (\beta I_{\infty} (\bar{u}^{\hbar,
    \rho})) + \llbracket \cos (\hbar^{1 / 2} \beta W_{\infty}) \rrbracket \sin
    (\beta I_{\infty} (\bar{u}^{\hbar, \rho})) .\\
    & = & \llbracket \sin (\hbar^{1 / 2} \beta W_{\infty}) \rrbracket \cos
    (\beta I_{\infty} (\bar{u}^{\hbar, \rho}))\\
    &  & + (\llbracket \cos (\hbar^{1 / 2} \beta W_{\infty}) \rrbracket - 1)
    \sin (\beta I_{\infty} (\bar{u}^{\hbar, \rho})) + \sin (\beta I_{\infty}
    (\bar{u}^{\hbar, \rho}))
  \end{eqnarray*}
  Now
  \begin{eqnarray*}
    &  & \| \rho \llbracket \sin (\hbar^{1 / 2} \beta W_{\infty}) \rrbracket
    \cos (\beta I_{\infty} (\bar{u}^{\hbar, \rho})) \|_{W^{- \beta^2 / 4 \pi -
    \delta, 2}}\\
    & \leqslant & \| \rho \llbracket \sin (\hbar^{1 / 2} \beta W_{\infty})
    \rrbracket \|_{W^{- \beta^2 / 4 \pi - \delta, 2}} \| \cos (\beta
    I_{\infty} (\bar{u}^{\hbar, \rho})) \|_{W^{1, \infty}}\\
    & \leqslant & C \| \rho \llbracket \sin (\hbar^{1 / 2} \beta W_{\infty})
    \rrbracket \|_{W^{- \beta^2 / 4 \pi - \delta, 2}}
  \end{eqnarray*}
  where we have used that $\| \bar{u}^{\hbar, \rho} \|_{L^{\infty}} \leqslant
  C \langle t \rangle^{- 1 / 2 - \delta}$ and Lemma \ref{lemma:boundWinfty}.
  Furthermore \ \
  \begin{eqnarray*}
    &  & \| \rho (\llbracket \cos (\hbar^{1 / 2} \beta W_{\infty}) \rrbracket
    - 1) \sin (\beta I_{\infty} (\bar{u}^{\hbar, \rho})) \|_{W^{- \beta^2 / 4
    \pi - \delta, 2}}\\
    & \leqslant & \| \rho (\llbracket \cos (\hbar^{1 / 2} \beta W_{\infty})
    \rrbracket - 1) \|_{W^{- \beta^2 / 4 \pi - \delta, 2}} \| I_{\infty}
    (\bar{u}^{\hbar, \rho}) \|_{W^{1, \infty}}\\
    & \leqslant & C \| \rho (\llbracket \cos (\hbar^{1 / 2} \beta W_{\infty})
    \rrbracket - 1) \|_{W^{- \beta^2 / 4 \pi - \delta, 2}} .
  \end{eqnarray*}
  These two estimates imply
  \begin{eqnarray*}
    &  & \left| \int \rho \left( \llbracket \sin (\hbar^{1 / 2} \beta
    W_{\infty}) \rrbracket \cos (\beta I_{\infty} (\bar{u}^{\hbar, \rho})) +
    \left( \left\llbracket \cos (\hbar^{1 / 2} \beta W_{\infty})
    \right\rrbracket - 1 \right) \sin (\beta I_{\infty} (\bar{u}^{\hbar,
    \rho})) \right) I_{\infty} (\bar{u}^{\hbar, \rho}) \mathd x \right|\\
    & \leqslant & C \left( \| \rho \llbracket \sin (\hbar^{1 / 2} \beta
    W_{\infty}) \rrbracket \|_{W^{- \beta^2 / 4 \pi - \delta, 2}} + \| \rho
    (\llbracket \cos (\hbar^{1 / 2} \beta W_{\infty}) \rrbracket - 1) \|_{W^{-
    \beta^2 / 4 \pi - \delta, \infty}} \right) \| \bar{u}^{\hbar, \rho}
    \|_{L^2 (\mathbb{R}_+ \times \mathbb{R}^2)}
  \end{eqnarray*}
  Finally
  \[ \lambda \beta \left| \int \rho \sin (\beta I_{\infty} (\bar{u}^{\hbar,
     \rho})) I_{\infty} (\bar{u}^{\hbar, \rho}) \right| \leqslant \lambda
     \beta \| I_{\infty} (\bar{u}^{\hbar, \rho}) \|^2_{L^2 (\mathbb{R}^2)}
     \leqslant C \lambda \beta \| \bar{u}^{\hbar, \rho} \|^2_{L^2
     (\mathbb{R}_+ \times \mathbb{R}^2)} \]
  Now putting everything together we obtain for $C \lambda \beta \leqslant 1 /
  2$
  \begin{eqnarray*}
    &  & \mathbb{E} [\| \bar{u}^{\hbar, \rho} \|^2_{L^2 (\mathbb{R}_+,
    L^2)}]^{1 / 2}\\
    & \leqslant & C\mathbb{E} \left[ \| \rho \llbracket \sin (\hbar^{1 / 2}
    \beta W_{\infty}) \rrbracket \|^2_{W^{- \beta^2 / 4 \pi - \delta, 2}} + \|
    \rho (\llbracket \cos (\hbar^{1 / 2} \beta W_{\infty}) \rrbracket - 1)
    \|^2_{W^{- \beta^2 / 4 \pi - \delta, 2}} \right]^{1 / 2}
  \end{eqnarray*}
  and the r.h.s goes to $0$ as $\hbar \rightarrow 0$ Reamark
  \ref{rem:conv-cos-semi} below.
\end{proof}

\begin{proposition}
  Assume that $| f |_{1, 2, m} < \infty$ and $f : H^{- 1} (\langle x
  \rangle^{- n}) \rightarrow \mathbb{R}$ be Lipschitz continuous.
  \[ \lim_{\hbar \rightarrow 0} \sup_{v \in \mathbb{D}^f} | G^f_{\hbar} (v) -
     G^f_0 (v) | = 0 \]
  where
  \begin{eqnarray*}
    &  & G^f_0 (u)\\
    & = & \mathbb{E} \left[ f (I_{0, \infty} (v)) + \lambda \int (\cos (\beta
    I_{0, \infty} (v)) - 1) + \frac{1}{2} \int^{\infty}_0 \| v_t \|^2_{L^2}
    \mathd t \right] .
  \end{eqnarray*}
\end{proposition}

\begin{proof}
  By Lipschitz continuity of $f$
  \begin{eqnarray*}
    &  & | f (\hbar^{1 / 2} W_{\infty} + I_{\infty} (v) + I_{\infty}
    (\bar{u}^{\hbar})) - f (I_{\infty} (v)) |\\
    & \leqslant & \hbar^{1 / 2} \mathbb{E} \| W_{\infty} \|_{H^{- 1} (\langle
    x \rangle^{- n})} +\mathbb{E} [\| I_{\infty} (\bar{u}^{\hbar}) \|_{H^{- 1}
    (\langle x \rangle^{- n})}]\\
    & \rightarrow & 0.
  \end{eqnarray*}
  Furthermore
  \begin{eqnarray*}
    &  & \left| \int \llbracket \sin (\hbar^{1 / 2} \beta W_{\infty})
    \rrbracket (\sin (\beta I_{\infty} (v) + \beta I_{\infty}
    (\bar{u}^{\hbar})) - \sin (\beta I_{\infty} (\bar{u}^{\hbar}))) \right|\\
    & \leqslant & \| \llbracket \sin (\hbar^{1 / 2} \beta W_{\infty})
    \rrbracket \|_{B_{p, p}^{- 1 + \delta} (\langle x \rangle^{- n})} \\
    &  &\times \| (\sin (\beta I_{\infty} (u) + \beta I_{\infty}
    (\bar{u}^{\hbar})) - \sin (\beta I_{\infty} (\bar{u}^{\hbar}))) \|_{B^{1 -
    \delta}_{q, q} (\langle x \rangle^{- n})}
  \end{eqnarray*}
  Now for $q$ close enough to $1$ we have for any $\gamma > 0$
  \begin{eqnarray*}
    &  & \| (\sin (\beta I_{0, \infty} (v) + \beta I_{0, \infty}
    (\bar{u}^{\hbar})) - \sin (\beta I_{0, \infty} (\bar{u}^{\hbar})))
    \|_{B^{1 - \delta}_{q, q} (\langle x \rangle^{- n})}\\
    & \leqslant & \| (\sin (\beta I_{\infty} (v) + \beta I_{\infty}
    (\bar{u}^{\hbar})) - \sin (\beta I_{\infty} (\bar{u}^{\hbar}))) \|^{1 -
    \delta}_{W^{1, 1, \gamma}}\\
    & \leqslant & \| (\cos (\beta I_{\infty} (v) + \beta I_{0, \infty}
    (\bar{u}^{\hbar})) - \cos (\beta I_{\infty} (\bar{u}^{\hbar}))) \nabla
    I_{\infty} (\bar{u}^{\hbar}) \|^{1 - \delta}_{L^{1, \gamma}}\\
    &  & + \beta \| (\cos (\beta I_{\infty} (v) + \beta I_{\infty}
    (\bar{u}^{\hbar})) \nabla I_{\infty} (v) \nobracket \|^{1 - \delta}_{L^{1,
    \gamma}}\\
    &  & + \beta \| (\sin (\beta I_{\infty} (v) + \beta I_{\infty}
    (\bar{u}^{\hbar})) - \sin (\beta I_{\infty} (\bar{u}^{\hbar}))) \|^{1 -
    \delta}_{L^{1, \gamma}}\\
    & \leqslant & C (\| I_{\infty} (v) \|_{L^{2, 2 \gamma}} \| \nabla
    I_{\infty} (\bar{u}^{\hbar}) \|_{L^{2, - \gamma}} + \| \nabla I_{\infty}
    (v) \|_{L^{2, 2 \gamma}} + \| I_{\infty} (\bar{u}^{\hbar}) \|_{L^{2, 2
    \gamma}})^{1 - \delta}
  \end{eqnarray*}
  so
  \begin{eqnarray*}
    &  & \mathbb{E} \left| \int \llbracket \sin (\hbar^{1 / 2} \beta
    W_{\infty}) \rrbracket (\sin (\beta I_{\infty} (v) + \beta I_{\infty}
    (\bar{u}^{\hbar})) - \sin (\beta I_{\infty} (\bar{u}^{\hbar}))) \right|\\
    & \leqslant & C\mathbb{E} [\| \llbracket \sin (\hbar^{1 / 2} W_{\infty})
    \rrbracket \|^{1 / \delta}_{B_{p, p}^{- 1 + \delta, - \gamma}}] \\
    &  &\times \mathbb{E} [\| I_{\infty} (v) \|_{L^{2, 2 \gamma}} \| \nabla
    I_{\infty} (\bar{u}^{\hbar}) \|_{L^{2, - \gamma}} + \| \nabla I_{\infty}
    (v) \|_{L^{2, 2 \gamma}} + \| I_{\infty} (v) \|_{L^{2, 2 \gamma}}]\\
    & \leqslant & C\mathbb{E} [\| \llbracket \sin (\hbar^{1 / 2} \beta
    W_{\infty}) \rrbracket \|^{1 / \delta}_{B_{p, p}^{- 1 + \delta, -
    \gamma}}] \mathbb{E} [\| I_{\infty} (v) \|^2_{L^{2, 2 \gamma}}]
    +\mathbb{E} [\| \nabla I_{\infty} (\bar{u}^{\hbar}) \|^2_{L^{2, -
    \gamma}}]\\
    &  & +\mathbb{E} [\| \llbracket \sin (\hbar^{1 / 2} \beta W_{\infty})
    \rrbracket \|^{1 / \delta}_{B_{p, p}^{- 1 + \delta, - \gamma}}] \mathbb{E}
    [\| \nabla I_{\infty} (v) \|_{L^{2, 2 \gamma}} + \| I_{\infty} (v)
    \|_{L^{2, 2 \gamma}}]
  \end{eqnarray*}
  and since by Remark \ref{rem:conv-cos-semi}
  \[ \mathbb{E} [\| \llbracket \sin (\hbar^{1 / 2} \beta W_{\infty})
     \rrbracket \|^{1 / \delta}_{B_{p, p}^{- 1 + \delta} (\langle x \rangle^{-
     n})}] \rightarrow 0, \]
  as $\hbar \rightarrow 0$, we have uniform convergence of this term to $0$.
  We now rewrite
  \begin{eqnarray*}
    &  & \left| \int \llbracket \cos (\hbar^{1 / 2} \beta W_{\infty})
    \rrbracket \left( \cos \left( \beta I_{\infty} (v) + \beta I_{0, \infty}
    (\bar{u}^{\hbar}) \right) - \cos (\beta I_{\infty} (\bar{u}^{\hbar}))
    \right) \right. \\
    &  &- \left. \int (\cos (\beta I_{\infty} (v)) - 1) \right|\\
    & \leqslant & \left| \int (\llbracket \cos (\hbar^{1 / 2} \beta
    W_{\infty}) \rrbracket - 1) (\cos (\beta I_{\infty} (v) + \beta I_{\infty}
    (\bar{u}^{\hbar})) - \cos (\beta I_{\infty} (\bar{u}^{\hbar}))) \right|\\
    &  & + \left| \int (\cos (\beta I_{\infty} (v) + \beta I_{\infty}
    (\bar{u}^{\hbar})) - \cos (\beta I_{\infty} (\bar{u}^{\hbar}))) - \int
    (\cos (\beta I_{\infty} (v)) - 1) \right| .
  \end{eqnarray*}
  The first term can be estimated in the same way as the sinus term, provided
  we replace $\llbracket \sin (\hbar^{1 / 2} \beta W_{0, \infty}) \rrbracket$
  with $\llbracket \cos (\hbar^{1 / 2} \beta W_{0, \infty}) \rrbracket - 1$
  which also satisfies
  \[ \mathbb{E} [\| \llbracket \cos (\hbar^{1 / 2} \beta W_{\infty})
     \rrbracket - 1 \|^{1 / \delta}_{B_{p, p}^{- 1 + \delta} (\langle x
     \rangle^{- n})}] \rightarrow 0. \]
  by Remark \ref{rem:conv-cos-semi}. For the second term by fundamental
  theorem of calculus we can write
  \begin{eqnarray*}
    & (\cos (\beta I_{\infty} (v) + I_{\infty} (\bar{u}^{\hbar})) - \cos
    (I_{\infty} (\bar{u}^{\hbar}))) - (\cos (I_{\infty} (v)) - 1) & \\
    = & - \beta \int^1_0 ((\cos (\theta \beta I_{\infty} (v) + \beta
    I_{\infty} (\bar{u}^{\hbar})) \nobracket - \cos (\theta \beta I_{\infty}
    (\bar{u}^{\hbar})) I_{0, \infty} (v)) \mathd \theta & \\
    = & - \beta \int^1_0 \int^1_0 ((\cos (\theta \beta I_{\infty} (v) + \xi
    \beta I_{\infty} (\bar{u}^{\hbar})) I_{0, \infty} (\bar{u}^{\hbar})
    \nobracket I_{\infty} (v)) \mathd \theta \mathd \xi & 
  \end{eqnarray*}
  and so
  \begin{eqnarray*}
    &  & \mathbb{E} \left| \int^1_0 \int^1_0 ((\cos (\theta \beta I_{\infty}
    (v) + \xi \beta I_{\infty} (\bar{u}^{\hbar})) I_{\infty} (\bar{u}^{\hbar})
    \nobracket I_{\infty} (v)) \mathd \theta \mathd \xi \right|\\
    & \leqslant & \mathbb{E} [\| I_{\infty} (\bar{u}^{\hbar}) \|_{L^{2, -
    \gamma}} \| I_{0, \infty} (v) \|_{L^{2, \gamma}}]\\
    & \leqslant & \mathbb{E} [\| I_{\infty} (\bar{u}^{\hbar}) \|^2_{L^{2, -
    \gamma}}]^{1 / 2} \mathbb{E} [\| I_{\infty} (v) \|^2_{L^{2, \gamma}}]^{1 /
    2}
  \end{eqnarray*}
  which implies also that term converges to $0$. Finally
  \[ \mathbb{E} \left[ \int^{\infty}_0 \int v_t \bar{u}_t^{\hbar} \mathd t
     \right] \leqslant \mathbb{E} [\| v \|_{L^2 (\mathbb{R}_+, L^{2, \gamma})}
     \| \bar{u}^{\hbar} \|_{L^2 (\mathbb{R}_+, L^{2, \gamma})}] \leqslant
     \mathbb{E} [\| v \|_{L^2 (\mathbb{R}_+, L^{2, \gamma})}^2]^{1 / 2}
     \mathbb{E} [\| \bar{u}^{\hbar} \|^2_{L^2 (\mathbb{R}_+, L^{2,
     \gamma})}]^{1 / 2} \]
  and we can conclude. 
\end{proof}

We now relate $G^f$ to the rate function.

\begin{lemma}
  {\tmdummy}
  
  \[ \inf_{u \in \mathbb{D}^f} G^f_0 (u) = \inf_{\psi \in H^1 (\mathbb{R}^2)}
     \{ f (\psi) + I (\psi) \} \]
\end{lemma}

\begin{proof}
  By Lemma \ref{lemma:restriction-classical} below it is enough to show that
  \[ \inf_{u \in \mathbb{D}^f} G^f_0 (u) = \inf_{\| \psi \|_{H^{1, \gamma}}
     \leqslant C | f |_{1, 2, \gamma}} \{ f (\psi) + I (\psi) \} \]
  for some $\gamma > 0$.
  
  {\tmstrong{Step 1}}. First we prove
  \[ \inf_{u \in \mathbb{H}_a} F (u) \leqslant \inf_{\| \psi \|_{H^{1,
     \gamma}} \leqslant C | f |_{1, 2, \gamma}} \{ f (\psi) + I (\psi) \} . \]
  Restricting the infimum to processes of the form
  \[ u_s = J_s (m^2 - \Delta) \psi \]
  with $\psi \in H^2 (\mathbb{R}^2) \cap H^{1, 2 \gamma}$ ,we see that
  \[ I_{0, \infty} (u) = \int^{\infty}_0 J_s u_s \mathd s = \int^{\infty}_0
     J^2_s (m^2 - \Delta) \psi \mathd s = \psi . \]
  We also compute
  \[ \int^{\infty}_0 \int_{\mathbb{R}^2} u^2_s \mathd s = \int^{\infty}_0
     \langle J^2_s (m^2 - \Delta) \psi, (m^2 - \Delta) \psi \rangle_{L^2
     (\mathbb{R}^2)} = \langle \psi, (m^2 - \Delta) \psi \rangle_{L^2
     (\mathbb{R}^2)} \]
  and with $w (x) = \exp (\gamma x)$
  \[ \int^{\infty}_0 \int_{\mathbb{R}^2} w u^2_s \mathd s = = \int^{\infty}_0
     \langle w J^2_s (m^2 - \Delta) \psi, (m^2 - \Delta) \psi \rangle_{L^2
     (\mathbb{R}^2)} = \langle w \psi, (m^2 - \Delta) \psi \rangle_{L^2
     (\mathbb{R}^2)} \]
  from which we can deduce that $\| w u \|^2_{L^2 (\mathbb{R}_+ \times
  \mathbb{R}^2)} \leqslant C \| \psi \|_{H^{1, \gamma}}$ and $u$ is in
  $\mathbb{D}^f$. So
  \begin{eqnarray*}
    &  & \inf_{u \in \mathbb{D}^f} F (u)\\
    & \leqslant & \inf_{\tmscript{\begin{array}{c}
      u_s = Q_s (m^2 - \Delta) \psi\\
      \psi \in H^2\\
      \| \psi \|_{H^{1, 2 \gamma}} \leqslant C | f |_{1, 2, m}
    \end{array}}} F (u)\\
    & \leqslant & \inf_{\tmscript{\begin{array}{c}
      \psi \in H^2\\
      \| \psi \|_{H^{1, 2 \gamma}} \leqslant C | f |_{1, 2, m}
    \end{array}}} \{ f (\psi) + I (\psi) \}\\
    & \leqslant & \inf_{\tmscript{\begin{array}{c}
      \| \psi \|_{H^{1, 2 \gamma}} \leqslant C | f |_{1, 2, m}
    \end{array}}} \{ f (\psi) + I (\psi) \}
  \end{eqnarray*}
  where the last equality follows from the density of the $H^2$ in $H^{1, 2
  \gamma}$ and continuity of the functional in $H^1$.
  
  {\tmstrong{Step 2}}.We now prove the converse inequality
  \[ \inf_{u \in \mathbb{D}^f} F_0^{\rho} (u) \geqslant \inf_{\psi \in H^1
     (\mathbb{R}^2)} \{ f (\psi) + I (\psi) \} . \]
  Recall that from Lemma \ref{bound-I-l2} $\| u \|_{L^2 (\mathbb{R}_+ \times
  \mathbb{R}^2)} \geqslant \| (m^2 - \Delta)^{1 / 2} I_{\infty} (u) \|_{L^2}$
  so
  \begin{eqnarray*}
    \inf_{u \in \mathbb{D}^C} F (u) & \geqslant & \inf_{u \in \mathbb{D}^C}
    \mathbb{E} \left[ f (I_{\infty} (u)) + \lambda \int \rho \cos (\beta I_{0,
    \infty} (u)) + \frac{1}{2} \int_{\mathbb{R}^2} ((m^2 - \Delta) I_{0,
    \infty} (u)) I_{0, \infty} (u) \right]\\
    & \geqslant & \inf_{\psi \in H^1 (\mathbb{R}^2)} \{ f (\psi) + I (\psi)
    \} .
  \end{eqnarray*}
  which proves the statement. 
\end{proof}

\begin{lemma}
  \label{lemma:restriction-classical}Assume that $2 \gamma^2 + \lambda < m^2$.
  Then for $\rho \in C^{\infty} (\mathbb{R}^2)$ and $\rho, | \nabla \rho |
  \leqslant 1$(note that this includes the $\rho = 1$ case.)
  \[ \inf_{\psi \in H^1 (\mathbb{R}^2)} f (\varphi) + I^{\rho} (\varphi) =
     \inf_{\| \psi \|_{H^{1, \gamma}} \leqslant C | f |_{1, 2, 2 \gamma}} f
     (\varphi) + I^{\rho} (\varphi) \]
\end{lemma}

\begin{proof}
  By a standard argument we obtain that any minimizer of $f (\varphi) + I
  (\varphi)$ satisfies the Euler Lagrange equation
  \begin{equation}
    \nabla f (\varphi) + \lambda \rho \sin (\beta \varphi) + m^2 \varphi -
    \Delta \varphi = 0. \label{EL-equation-classical}
  \end{equation}
  Now multiplying {\eqref{EL-equation-classical}} with $w \varphi$ where $w
  (x) = \exp (2 \gamma | x |)$ and integrating we obtain
  \begin{eqnarray*}
    0 & = & \int w \nabla f (\varphi) \varphi + \lambda \int w \rho \sin
    (\beta \varphi) \varphi + m^2 \int w \varphi^2 - \int w \varphi \Delta
    \varphi\\
    & = & \int \rho \nabla f (\varphi) \varphi + \lambda \int w \rho \sin
    (\beta \varphi) \varphi + m^2 \int w \varphi^2 + \int w | \nabla \varphi
    |^2 + \int \varphi \nabla w \cdot \nabla \varphi
  \end{eqnarray*}
  now observe that $\nabla w = 2 \gamma \frac{x}{| x |} \exp (2 \gamma | x |)$
  so $| \nabla w | \leqslant 2 \gamma w$
  \[ \int | \varphi \nabla w \cdot \nabla \varphi | \leqslant 2 \gamma^2 \int
     w \varphi^2 + \frac{1}{2} \int w | \nabla \varphi |^2 \]
  note also that
  \[ \lambda \int | \rho w \sin (\beta \varphi) \varphi | \leqslant \lambda
     \int w \varphi^2 . \]
  Since $m^2 \int w \varphi^2 + \int w | \nabla \varphi |^2 > 0$ we have
  \begin{eqnarray*}
    0 & = & \int w \nabla f (\varphi) \varphi + \lambda \int w \rho \sin
    (\beta \varphi) \varphi + m^2 \int w \varphi^2 + \int w | \nabla \varphi
    |^2 + \int \varphi \nabla w \cdot \nabla \varphi\\
    & \geqslant & (m^2 - 2 \gamma^2 + \lambda - \delta) \int w \varphi^2 +
    \frac{1}{2} \int w | \nabla \varphi |^2 - C | f |^2_{1, 2, m}
  \end{eqnarray*}
  which implies
  \[ \int w \varphi^2 + \frac{1}{2} \int w | \nabla \varphi |^2 \leqslant
     C_{\gamma} | f |^2_{1, 2, m} . \]
  
\end{proof}

\section{Osterwalder Schrader Axioms}\label{sec:OS}

In this section we complete the proof of Theorem \ref{thm:OS}.

\subsection{Reflection Positivity}\label{sec:ref-pos}

To prove Reflection Positivity we prove that the measure $\nu_{\tmop{SG}}$ is
a limit of reflection positive measures, since reflection positivity is
clearly preserved by weak limits.. We denote by $\nu^{\rho}_{\tmop{SG}} : =
\lim_{T \rightarrow \infty} \nu^{\rho, T}_{\tmop{SG}}$. Since
$\nu^{\rho}_{\tmop{SG}} \rightarrow \nu_{\tmop{SG}}$ as $\rho \rightarrow 1$
it is enough to construct a sequence $\nu^{\varepsilon, \rho}_{\tmop{SG}}
\rightarrow \nu^{\rho}_{\tmop{SG}}$ such that $\nu^{\varepsilon,
\rho}_{\tmop{SG}}$ is reflection positive. We can take $\rho$ being invariant
under the time reflection $\Theta f (x_1, x_2) \backassign f (- x_1, x_2)$. To
construct $\nu^{\varepsilon, \rho}_{\tmop{SG}}$ we cannot smooth in the
``physical time'' direction since this would destroy reflection positivity.
Instead define $\theta = \delta_0 \otimes \eta$, $\theta \in \mathcal{} \CS'
(\mathbb{R}^2)$ where $\eta \in C^{\infty}_c (\mathbb{R})$. Also set
$\theta^{\varepsilon} = \varepsilon^{- 2} \theta (\cdot / \varepsilon) =
\delta_0 \otimes \eta_{\varepsilon}$ where $\eta_{\varepsilon} =
\varepsilon^{- 1} \eta (\cdot / \varepsilon)$. Finally we set
$W_T^{\varepsilon} = \theta_{\varepsilon} \ast W_T$, $T \in [0, \infty]$. We
define
\[ \nu^{\varepsilon, \rho}_{\tmop{SG}} = e^{- \lambda \int \rho
   \alpha^{\varepsilon} \cos (\beta W_{\infty}^{\varepsilon})} \mathd
   \mathbb{P}. \]
We will now proceed in three steps: In Step 1 we show that for the correct
choice of $\alpha^{\varepsilon}$
\[ \alpha^{\varepsilon} \cos (\beta W_{\infty}^{\varepsilon}) \rightarrow
   \llbracket \cos (\beta W_{\infty}) \rrbracket . \]
In Step 2 we show that for any $p > 1$
\[ \sup_{\varepsilon} \mathbb{E} \left[ e^{- \lambda p \int \rho
   \alpha^{\varepsilon} \cos (\beta W_{\infty}^{\varepsilon})} \right] <
   \infty . \]
Steps 1 and 2 together imply that $\nu^{\varepsilon, \rho}_{\tmop{SG}}
\rightarrow \nu^{\rho}_{\tmop{SG}} .$ In Step 3 we prove that
$\nu^{\varepsilon, \rho}_{\tmop{SG}}$ is indeed reflection positive.

\textbf{Step 1.}Observe that
\[ \mathbb{E} [\theta_{\varepsilon} \ast W_{T_1} (x) \theta_{\varepsilon}
   \ast W_{T_2} (y)] = (\theta_{\varepsilon} \otimes \theta_{\varepsilon} \ast
   K_{T_1 \wedge T_2}) (x, y) . \]
Now observe that for $T \in [0, \infty]$, $K_T (x, y) = \bar{K}_T (x - y)$
with $\bar{K}_T (x) \leqslant - \frac{1}{4 \pi} \log (T \wedge | x |) + g (x)$
with $g$ a bounded function. Furthermore \ \ \
\[ (\theta_{\varepsilon} \otimes \theta_{\varepsilon} \ast K_T) (x, y) =
   (\theta_{\varepsilon} \ast \theta_{\varepsilon} \ast \bar{K}_T) (x - y) .
\]
Then it not hard to see that \
\[ \bar{K}^{\varepsilon} (x) = \theta_{\varepsilon} \ast \theta_{\varepsilon}
   \ast K_{\infty} = \frac{1}{4 \pi} \log \left( \frac{1}{| x | \vee
   \varepsilon} \right) + g^{\varepsilon} (x) . \]
with $\sup_{\varepsilon} \| g^{\varepsilon} \|_{L^{\infty}} < \infty$. From
this we can deduce that $\theta_{\varepsilon} \ast W_{\infty} (x)$ is in
$L_{\tmop{loc}}^2 (\mathbb{R}^2)$ almost surely since for any bounded $U
\subseteq \mathbb{R}^2$
\[ \mathbb{E} \left[ \int_U ((\theta_{\varepsilon} \ast W_{\infty}) (x))^2
   \mathd x \right] = | U | \bar{K}^{\varepsilon} (0) . \]
We claim that for any $f \in C^{\infty}_c (\mathbb{R}^2)$
\[ \int_{\mathbb{R}^2} f \tau^{\varepsilon} : = \int f e^{\frac{\beta^2}{2}
   \bar{K}^{\varepsilon} (0)} e^{i \beta W_{\infty}^{\varepsilon}} \rightarrow
   \int f \llbracket e^{i \beta W_{\infty}^{}} \rrbracket \]
where the convergence is in $L^2 (\mathbb{P})$. To prove this we calculate
\begin{eqnarray*}
  &  & \mathbb{E} \left[ \left| \int_{\mathbb{R}^2} e^{\frac{1}{2} \beta^2
  \bar{K}^{\varepsilon} (0)} e^{i \beta W_{\infty}^{\varepsilon} (x)} f (x) -
  e^{\frac{1}{2} \beta^2 K_T (0)} e^{i \beta W_T (x)} f (x) \mathd x \right|^2
  \right]\\
  & = & \mathbb{E} \left[ \int_{\mathbb{R}^2} \int_{\mathbb{R}^2} e^{\beta^2
  \bar{K}^{\varepsilon} (0)} e^{i \beta (W_{\infty}^{\varepsilon} (x) -
  W_{\infty}^{\varepsilon} (y))} - e^{\frac{\beta^2}{2} (\bar{K}_T (0) +
  \bar{K}^{\varepsilon} (0))} e^{i \beta (W_T (x) - W_{\infty}^{\varepsilon}
  (y))} \right.\\
  &  & \left. - e^{\frac{\beta^2}{2} (\bar{K}_T (0) + \bar{K}^{\varepsilon}
  (0))} e^{i \beta (W_{\infty}^{\varepsilon} (x) - W_T (y))} + e^{\beta^2
  \bar{K}_T (0)} e^{i \beta (W_T (y) - W_T (y))} f (x) f (y) \mathd x \mathd y
  \right]\\
  & = & \int_{\mathbb{R}^2} \int_{\mathbb{R}^2} e^{\beta^2
  \bar{K}^{\varepsilon} (x - y)} + e^{\beta^2 \bar{K}_T (x - y)} - 2
  e^{\beta^2 \mathbb{E} [W_T (x) W^{\varepsilon}_{\infty} (y)]} f (x) f (y)
  \mathd x \mathd y.
\end{eqnarray*}
W.l.o.g we can take $f \geqslant 0$. Now since $\bar{K}^{\varepsilon} (x - y)
\leqslant - \frac{1}{4 \pi} \log | x - y | + C$, $K_T (x - y) \leqslant -
\frac{1}{4 \pi} \log | x - y | + C$. We have by dominated convergence and
Fatou's lemma
\begin{eqnarray*}
  &  & \lim_{\varepsilon \rightarrow 0} \lim_{T \rightarrow \infty}
  \int_{\mathbb{R}^2} \int_{\mathbb{R}^2} e^{\beta^2 \bar{K}^{\varepsilon} (x
  - y)} + e^{\beta^2 \bar{K}_T (x - y)} - 2 e^{\beta^2 \mathbb{E} [W_T (x)
  W^{\varepsilon}_{\infty} (y)]} f (x) f (y) \mathd x \mathd y.\\
  & = & 0
\end{eqnarray*}
which proves the claim. This clearly implies
\[ \int \rho e^{\frac{1}{2} \beta^2 \bar{K}^{\varepsilon} (0)} \cos (\beta
   W_{\infty}^{\varepsilon}) \rightarrow \int \rho \llbracket \cos (\beta
   W_{\infty}) \rrbracket \]
in $L^2 (\mathbb{P}) .$ In particular we can select a subsequence (not
relabeled) such that this implies that $\mathbb{P}- a.s$
\[ e^{- \lambda \int \rho e^{\frac{1}{2} \beta^2 \bar{K}^{\varepsilon} (0)}
   \cos (\beta W_{\infty}^{\varepsilon})} \rightarrow e^{- \lambda \int \rho
   \llbracket \cos (\beta W_{\infty}) \rrbracket} . \]
\textbf{Step 2.} Step 1 will imply that $\nu^{\varepsilon, \rho}_{\tmop{SG}}
\rightarrow \nu^{\rho}_{\tmop{SG}}$ as soon as we have established that
\[ \sup_{\varepsilon} \mathbb{E} \left[ e^{- \lambda p \int \rho e^{\beta^2
   \bar{K}^{\varepsilon} (0)} \cos (\beta W_{\infty}^{\varepsilon})} \right] <
   \infty . \]
From Corollary \ref{cor:BD-formula} we know
\[ - \log \mathbb{E} \left[ e^{- \lambda p \int \rho e^{\frac{\beta^2}{2}
   K^{\varepsilon} (0)} \cos (W_{\infty}^{\varepsilon})} \right] = \inf_{u \in
   \mathbb{H}_a} \mathbb{E} \left[ \lambda p \int \rho e^{\frac{\beta^2}{2}
   K^{\varepsilon} (0)} \cos (\beta (W_{\infty}^{\varepsilon} +
   I^{\varepsilon} (u))) + \frac{1}{2} \int^{\infty}_0 \| u_t \|^2_{L^2}
   \mathd t \right] \]
with $I^{\varepsilon} (u) = \theta^{\varepsilon} \ast I_{\infty} (u)$.
Expanding the cosine we get
\begin{eqnarray*}
  &  & \left| \int \rho e^{\frac{\beta^2}{2} K^{\varepsilon} (0)} \cos (\beta
  (W_{\infty}^{\varepsilon} + I^{\varepsilon} (u))) \right|^2\\
  & = & \left| \int \rho e^{\frac{\beta^2}{2} K^{\varepsilon} (0)} \cos
  (\beta W_{\infty}^{\varepsilon}) \cos (\beta I^{\varepsilon} (u)) \right|^2
  + \left| \int \rho e^{\frac{\beta^2}{2} K^{\varepsilon} (0)} \sin (\beta
  W_{\infty}^{\varepsilon}) \sin (\beta I^{\varepsilon} (u)) \right|^2\\
  & \leqslant & \mathbb{E} \left[ \left\| \rho e^{\frac{\beta^2}{2}
  K^{\varepsilon} (0)} \cos (\beta W_{\infty}^{\varepsilon}) \right\|^2_{H^{-
  1}} \right] \mathbb{E} [\| \cos (\beta I^{\varepsilon} (u)) \|_{H^1}^2] \\
  &  &+\mathbb{E} \left[ \left\| \rho e^{\frac{\beta^2}{2} K^{\varepsilon} (0)}
  \sin (\beta W_{\infty}^{\varepsilon}) \right\|^2_{H^{- 1}} \right]
  \mathbb{E} [\| \sin (\beta I^{\varepsilon} (u)) \|_{H^1}^2]\\
  & \leqslant & C \left( \mathbb{E} \left[ \left\| \rho e^{\frac{\beta^2}{2}
  K^{\varepsilon} (0)} \cos (\beta W_{\infty}^{\varepsilon}) \right\|^2_{H^{-
  1}} \right] +\mathbb{E} \left[ \left\| \rho e^{\frac{\beta^2}{2}
  K^{\varepsilon} (0)} \sin (\beta W_{\infty}^{\varepsilon}) \right\|^2_{H^{-
  1}} \right] \right) \mathbb{E} [\| I^{\varepsilon} (u) \|_{H^1}^2]\\
  & \leqslant & C \left( \mathbb{E} \left[ \left\| \rho
  e^{\frac{\beta^2}{2}^{} K^{\varepsilon} (0)} \cos (\beta
  W_{\infty}^{\varepsilon}) \right\|^2_{H^{- 1}} \right] +\mathbb{E} \left[
  \left\| \rho e^{\frac{\beta^2}{2} K^{\varepsilon} (0)} \sin (\beta
  W_{\infty}^{\varepsilon}) \right\|^2_{H^{- 1}} \right] \right) \mathbb{E}
  \left[ \int^{\infty}_0 \| u_t \|_{L^2}^2 \mathd t \right]
\end{eqnarray*}
where in the last line we have used Lemma \ref{IH1estimate}. This implies by
Young's inequality
\begin{eqnarray*}
  &  & \inf_{u \in \mathbb{H}_a} \mathbb{E} \left[ \lambda p \int \rho
  e^{\beta^2 K^{\varepsilon} (0)} \cos (\beta (W_{\infty}^{\varepsilon} +
  I^{\varepsilon} (u))) + \frac{1}{2} \int^{\infty}_0 \| u_t \|^2_{L^2} \mathd
  t \right]\\
  & \geqslant & - C \left( \mathbb{E} \left[ \left\| \rho
  e^{\frac{\beta^2}{2} K^{\varepsilon} (0)} \cos (\beta
  W_{\infty}^{\varepsilon}) \right\|^2_{H^{- 1}} \right] +\mathbb{E} \left[
  \left\| \rho e^{\frac{\beta^2}{2} K^{\varepsilon} (0)} \sin (\beta
  W_{\infty}^{\varepsilon}) \right\|^2_{H^{- 1}} \right] \right) + \frac{1}{4}
  \mathbb{E} \left[ \int^{\infty}_0 \| u_t \|_{L^2}^2 \mathd t \right] .
\end{eqnarray*}
Now note that from a simple calculation we get
\[ \mathbb{E} \left[ \left| e^{\frac{\beta^2}{2} K^{\varepsilon} (0)} \cos
   (\beta W_{\infty}^{\varepsilon} (x)) e^{\frac{\beta^2}{2} K^{\varepsilon}
   (0)} \cos (\beta W_{\infty}^{\varepsilon} (y)) \right| \right] \leqslant C
   \frac{1}{| x - y |^{\beta^2 / 2 \pi}}, \]
from which we can conclude by Lemma \ref{regularityN} that $\sup_{\varepsilon}
\mathbb{E} \left[ \left\| \rho e^{\frac{\beta^2}{2} K^{\varepsilon} (0)} \cos
(\beta W_{\infty}^{\varepsilon}) \right\|^2_{H^{- 1}} \right] < \infty$, so we
can deduce that $\sup_{\varepsilon} \mathbb{E} \left[ e^{- \lambda p \int \rho
e^{\beta K^{\varepsilon} (0)} \cos (\beta W_{\infty}^{\varepsilon})} \right] <
\infty$. \

\textbf{Step 3.} We now show that $\nu^{\varepsilon, \rho}_{\tmop{SG}}$ are
reflection positive. We can write
\[ \nu^{\varepsilon, \rho}_{\tmop{SG}} = e^{- \lambda S^{\rho}_{\varepsilon}
   (\phi)} \mu_F^{\varepsilon} (\mathd \phi), \text{\quad with\quad}
   S^{\rho}_{\varepsilon} (\phi) = e^{\frac{1}{2} \beta^2 K^{\varepsilon} (0)}
   \int \rho \cos (\beta \phi) \]
where $\mu_F^{\varepsilon} = \tmop{Law} (W_{\infty}^{\varepsilon})$ is the
gaussian measure with covariance operator
\[ C^{\varepsilon} (f) = \theta^{\varepsilon} \ast (m^2 - \Delta)^{- 1} \ast
   \theta^{\varepsilon} f. \]
We claim that $\mu_F^{\varepsilon}$ is reflection positive. Since it is
Gaussian by Theorem 6.2.2 in {\cite{Glimm-1987}} it is enough to show that
\[ \langle f, \Pi_+ \Theta C^{\varepsilon} \Pi_+ f \rangle_{L^2} \geqslant 0.
\]
Where $\Pi_+$ is the projection on $L^2 (\mathbb{R}_+ \times \mathbb{R})$.
Since the convolution with $\theta^{\varepsilon}$ commutes with $\Pi_+$ we
have
\begin{eqnarray*}
  &  & \langle f, \Pi_+ \Theta C^{\varepsilon} \Pi_+ f \rangle\\
  & = & \langle \Pi_+ (\theta^{\varepsilon} \ast f), \Theta (m^2 - \Delta)^{-
  1} \Pi_+ (\theta^{\varepsilon} \ast f) \rangle\\
  & \geqslant & 0,
\end{eqnarray*}
where in the last line we have used reflection positivity of $(m^2 -
\Delta)^{- 1}$. Now finally we prove that $\nu^{\varepsilon,
\rho}_{\tmop{SG}}$ is indeed reflection positive. Write
\[ S^{\rho, +}_{\varepsilon} (\phi) = e^{\frac{\beta^2}{2} K^{\varepsilon}
   (0)} \int_{\mathbb{R}_+ \times \mathbb{R}} \rho \cos (\beta \phi) . \]
Observe that provided $\rho$ is symmetric
\[ S^{\rho}_{\varepsilon} (\phi) = S^{\rho, +}_{\varepsilon} (\phi) + S^{\rho,
   +}_{\varepsilon} (\Theta \phi) . \]
Then
\[ \int F (\phi) \Theta F (\phi) \mathd \nu^{\varepsilon, \rho}_{\tmop{SG}}
   =_{} \int F (\phi) e^{- \lambda S^{\rho, +}_{\varepsilon} (\phi)} \Theta (F
   (\phi) e^{- \lambda S^{\rho, +}_{\varepsilon} (\phi)}) \mathd
   \mu_F^{\varepsilon} \geqslant 0 \]
by reflection positivity of $\mu_F^{\varepsilon}$.

\subsection{Exponential clustering}\label{sec:MG}

In this section we want to study expectations under the Sine Gordon measure of
the form
\[ \int_{\CS' (\mathbb{R}^2)} \prod_{i = 1}^k \langle \psi_i, \phi
   \rangle_{L^2 (\mathbb{R}^2)} \nu_{\tmop{SG}} (\mathd \phi) . \]
Our goal is to show that there exist constants $C = C (\{ \psi_i \}^k_{i =
1})$ and an $m_p > 0$ independent of $\psi$, such that for any $a \in
\mathbb{R}^2$ and $\tmop{supp} \psi_i \subset B (0, 1)$
\begin{eqnarray*}
  &  & \left| \int_{\CS' (\mathbb{R}^2)} \prod_{i = 1}^l \langle \psi_i, \phi
  \rangle_{L^2 (\mathbb{R}^2)} \prod_{i = l + 1}^k \langle \psi_i (\cdot + a),
  \phi \rangle \nu_{\tmop{SG}} (\mathd \phi) \right.\\
  &  & \left. - \int_{\CS' (\mathbb{R}^2)} \prod_{i = 1}^l \langle \psi_i,
  \phi \rangle_{L^2 (\mathbb{R}^2)} \nu_{\tmop{SG}} (\mathd \phi) \int_{\CS'
  (\mathbb{R}^2)} \prod_{i = l + 1}^k \langle \psi_i, \phi \rangle_{L^2
  (\mathbb{R}^2)} \nu_{\tmop{SG}} (\mathd \phi) \right|\\
  & \leqslant & C \exp (- m_p | a |) .
\end{eqnarray*}
In this subsection all constants will be allowed to depend on $\psi_i$. The
idea of proof is similar to the analogous stament in
{\cite{Barashkov-elliptic}}. First note that a simple computation gives, for
$f, g : H^{- 1} (\langle x \rangle^{- n}) \rightarrow \mathbb{R}$
continuous,bounded \
\[ \frac{\mathd}{\mathd t} \frac{\mathd}{\mathd s} \left( - \log \int_{\CS'
   (\mathbb{R}^2)} e^{- t f - s g} \mathd \nu_{\tmop{SG}} \right) = \int_{\CS'
   (\mathbb{R}^2)} f g \mathd \nu_{\tmop{SG}} - \int_{\CS' (\mathbb{R}^2)} f
   \mathd \nu_{\tmop{SG}} \int_{\CS' (\mathbb{R}^2)} g \mathd \nu_{\tmop{SG}}
   . \]
\tmcolor{black}{\begin{lemma}
  \label{lemma:pre-massgap}Assume that $0 < \gamma < m$ and $f, g : H^{- 1}
  (\langle x \rangle^{- n}) \rightarrow \mathbb{R}^2$ are
  Frechet-differentiable such that $| f |^y_{1, 2, \gamma} + | g |^z_{1, 2,
  \gamma}$
  \[ \frac{\mathd}{\mathd t} \frac{\mathd}{\mathd s} \left( - \log \int_{\CS'
     (\mathbb{R}^2)} e^{- t f - s g} \mathd \nu_{\tmop{SG}} \right) \leqslant
     C | f |^z_{1, 2, \gamma} | g |_{1, 2, \gamma}^y \exp (- \gamma | x - z |)
     . \]
\end{lemma}}

\begin{proof}
  By weak convergence it is enough to prove the statement for
  $\nu^{\rho}_{\tmop{SG}}$ with $C, \gamma$ uniform in $\rho .$ By Lemma
  \ref{lemma:coupling} we have
  \[ \frac{\mathd}{\mathd s} \frac{\mathd}{\mathd t} \left( - \log \int_{\CS'
     (\mathbb{R}^2)} e^{- t f - s g} \mathd \nu^{\rho}_{\tmop{SG}} \right) =
     \lim_{s \rightarrow 0} \frac{1}{s} (\mathbb{E} [f (W_{\infty} +
     I_{\infty} (u^{s g, \rho}))] -\mathbb{E} [f (W_{\infty} + I_{\infty}
     (u^{0, \rho}))]) . \]
  Now from Theoerem \ref{thm:weighted}\tmcolor{red}{ }we get
  \[ \| I_{\infty} (u^{s g, \rho}) - I_{\infty} (u^{0, \rho}) \|_{L_z^{2,
     \gamma} (B)} \leqslant s | g |_{1, 2}^B, \]
  so we have by Lemma \ref{lemma:decay1}
  \begin{eqnarray*}
    &  & | \mathbb{E} [f (W_T + I_T (u^{s g, \rho}))] -\mathbb{E} [f (W_T +
    I_T (u^{0, \rho}))] |\\
    & \leqslant & C | f |^z_{1, 2, y} \| I_T (u^{s g, \rho}) - I_T (u^{0,
    \rho}) \|_{L_y^{2, - \gamma}}\\
    & \leqslant & C | f |^z_{1, 2, x} \| I_T (u^{s g, \rho}) - I_T (u^{0,
    \rho}) \|_{L_z^{2, \gamma}} \exp (- \gamma | y - z |)\\
    & \leqslant & C s | f |^z_{1, 2, \gamma} | g |_{1, 2, \gamma}^y \exp (-
    \gamma | y - z |)
  \end{eqnarray*}
  which implies the statement. 
\end{proof}

Finally we are able to prove the exponential clustering: Take $\chi^N \in
C_c^{\infty} (\mathbb{R}, \mathbb{R})$ with $\chi^N (x) = 1$ if $| x |
\leqslant N$ and $\chi^N (x) = 0$ if $| x | \geqslant N + 1$, $\sup_{N \in
\mathbb{N}} \| (\chi^N)' \|_{L^{\infty}} \leqslant C$. Now define
\[ f^N (\phi) = \prod_{i = 1}^l \langle \psi_i, \phi \rangle_{L^2
   (\mathbb{R}^2)} \chi^N (\| \phi \|_{H^{- 1, - \gamma}}), \qquad g^N (\phi)
   = \prod_{i = l + 1}^k \langle \psi_i, \phi \rangle_{L^2 (\mathbb{R}^2)}
   \chi^N (\| \phi \|_{H^{- 1, - \gamma}}) . \]
Furthermore introduce
\[ g^{N, a} (\phi) = \prod_{i = l + 1}^k \langle \psi_i (\cdot + a), \phi
   \rangle_{L^2 (\mathbb{R}^2)} \chi (\| \phi \|_{H_a^{- 1, - \gamma}}) . \]
Observe that $f^N, g^N \in C^2 (L^2 (\mathbb{R}^2)) .$ Note that with $w (x) =
\exp (- \gamma | x - a |)$ by product rule
\begin{eqnarray*}
  &  & \nabla f^N (\phi)\\
  & = & \chi^N (\| \phi \|_{H^{- 1, - \gamma}}) \sum^l_{j = 1}
  \prod_{\tmscript{\begin{array}{l}
    i = 0\\
    i \neq j
  \end{array}}}^l \langle \psi_i, \phi \rangle_{L^2 (\mathbb{R}^2)} \psi_j\\
  &  & + \frac{(\chi^N)' (\| \phi \|_{H^{- 1, - \gamma}})}{\| \phi \|_{H^{-
  1, - \gamma}}} \prod_{\tmscript{\begin{array}{l}
    i = 0
  \end{array}}}^l \langle \psi_i, \phi \rangle_{L^2 (\mathbb{R}^2)} (w (1 -
  \Delta)^{- 1} w \phi)
\end{eqnarray*}
so since
\[ \| w (1 - \Delta)^{- 1} w \phi \|_{L^{2, \gamma}} \leqslant \| (1 -
   \Delta)^{- 1} w \phi \|_{L^2} \leqslant C \| \phi \|_{H^{- 1, - \gamma}} \]
\[ \  \]
\[ | \nabla f^N (\phi) |_{1, 2, \gamma} \leqslant C N^l \left( \prod_{j = 1}^l
   | \psi_j |_{1, 2, \gamma} \right) \]
and now by exponential integrability and translation invariance of
$\nu_{\tmop{SG}}$
\begin{eqnarray*}
  &  & \int_{\CS' (\mathbb{R}^2)} \left| \prod_{i = 1}^l \langle \psi_i, \phi
  \rangle_{L^2 (\mathbb{R}^2)} \prod_{i = l + 1}^k \langle \psi_i (\cdot + a),
  \phi \rangle - f^N (\phi) g^{N, a} (\phi) \right| \nu_{\tmop{SG}} (\mathd
  \phi)\\
  & \leqslant & C \int_{\left\{ \| \phi \|_{H_a^{- 1, - \gamma}} \geqslant N
  \text{or}  \| \phi \|_{H^{- 1, - \gamma}} \geqslant N \right\} } \| \phi
  \|^l_{H^{- 1, - \gamma}} \| \phi \|_{H_a^{- 1, - \gamma}}^{k - l}
  \nu_{\tmop{SG}} (\mathd \phi)\\
  & \leqslant & 2 \nu_{\tmop{SG}} (\| \phi \|_{H^{- 1, - \gamma}} \geqslant
  N)^{1 / 2} \int_{\CS' (\mathbb{R}^2)} \| \phi \|^{4 l}_{H^{- 1, - \gamma}}
  \nu_{\tmop{SG}} (\mathd \phi) \int \| \phi \|_{H^{- 1, - \gamma}_a}^{4 k - 4
  l} \nu_{\tmop{SG}} (\mathd \phi)\\
  & \leqslant & C 2 \nu_{\tmop{SG}} (\| \phi \|_{H^{- 1, - \gamma}} \geqslant
  N)^{1 / 2} \int_{\CS' (\mathbb{R}^2)} \| \phi \|_{H^{- 1, - \gamma}}^{4 k}
  \nu_{\tmop{SG}} (\mathd \phi)\\
  & \leqslant & C e^{- N} .
\end{eqnarray*}
And analogous statements hold for
\[ \int_{\CS' (\mathbb{R}^2)} \left| \prod_{i = 1}^l \langle \psi_i, \phi
   \rangle_{L^2 (\mathbb{R}^2)} - f^N (\phi) \right| \nu_{\tmop{SG}} (\mathd
   \phi), \int_{\CS' (\mathbb{R}^2)} \left| \prod_{i = l + 1}^k \langle
   \psi_i, \phi \rangle - g^N (\phi) \right| \nu_{\tmop{SG}} (\mathd \phi) .
\]
Now by Lemma \ref{lemma:pre-massgap}
\begin{eqnarray*}
  &  & \left| \int_{\CS' (\mathbb{R}^2)} f^N (\phi) g^{N, a} (\phi)
  \nu_{\tmop{SG}} (\mathd \phi) - \int_{\CS' (\mathbb{R}^2)} f^N (\phi)
  \nu_{\tmop{SG}} (\mathd \phi) \int_{\CS' (\mathbb{R}^2)} g^N (\phi)
  \nu_{\tmop{SG}} (\mathd \phi) \right|\\
  & = & \left| \int_{\CS' (\mathbb{R}^2)} f^N (\phi) g^{N, a} (\phi)
  \nu_{\tmop{SG}} (\mathd \phi) - \int_{\CS' (\mathbb{R}^2)} f^N (\phi)
  \nu_{\tmop{SG}} (\mathd \phi) \int_{\CS' (\mathbb{R}^2)} g^{N, a} (\phi)
  \nu_{\tmop{SG}} (\mathd \phi) \right|\\
  & \leqslant & | \nabla f^N (\phi) |_{1, 2, \gamma} | \nabla g^{N, a} (\phi)
  |^a_{1, 2, \gamma} \exp (- \gamma a)\\
  & = & | \nabla f^N (\phi) |_{1, 2, \gamma} | \nabla g^N (\phi) |_{1, 2,
  \gamma} \exp (- \gamma a)\\
  & \leqslant & C N^k \exp (- \gamma a)
\end{eqnarray*}
Putting things together we have
\begin{eqnarray*}
  &  & \int_{\CS' (\mathbb{R}^2)} \prod_{i = 1}^l \langle \psi_i, \phi
  \rangle_{L^2 (\mathbb{R}^2)} \prod_{i = l + 1}^k \langle \psi_i (\cdot + a),
  \phi \rangle \nu_{\tmop{SG}} (\mathd \phi)\\
  &  & - \int_{\CS' (\mathbb{R}^2)} \prod_{i = 1}^l \langle \psi_i, \phi
  \rangle_{L^2 (\mathbb{R}^2)} \nu_{\tmop{SG}} (\mathd \phi) \int_{\CS'
  (\mathbb{R}^2)} \prod_{i = l + 1}^k \langle \psi_i, \phi \rangle_{L^2
  (\mathbb{R}^2)} \nu_{\tmop{SG}} (\mathd \phi)\\
  & \leqslant & C (N^k \exp (- \gamma a) + \exp (- N))\\
  N = \gamma | a | & = & C ((\gamma a)^k \exp (- \gamma | a |) + \exp (-
  \gamma | a |))\\
  & \leqslant & C \exp (- (1 - \delta) \gamma | a |) .
\end{eqnarray*}
\subsection{Non Gaussianity}\label{sec:NG}

In this section we prove that $\nu_{\tmop{SG}}$ is indeed not a Gaussian
measure. Assume $\nu_{\tmop{SG}}$ would be Gaussian, we can regard it as a
gaussian measure on the Hilbert space $H^{- 1} (\langle x \rangle^{- n})$ with
$n \in \mathbb{N}$ sufficiently large. Then there exists a Banach space
$\mathcal{H} \subseteq \CS' (\mathbb{R}^2)$ and $M \in H^{- 1} (\langle x
\rangle^{- n})$ such that for any $\psi \in \mathcal{H}$
\[ \log \int e^{- \langle \psi, \phi \rangle} \mathd \nu_{\tmop{SG}} (\mathd
   \phi) = \| \psi \|^2_{\mathcal{H}} + (M, \psi)_{H^{- 1} (\langle x
   \rangle^{- n})} \]
(This follows easily from Lemma 5.1 in {\cite{Kukush-2019}}). On the other
hand we know that with $V^{\rho}_T (\phi) = \alpha (T) \int \rho (x) \cos
(\phi (x)) \mathd x$ by the Cameron-Martin theorem for the Gaussian Free Field
\begin{eqnarray*}
  &  & \log \int e^{- \langle \psi, \phi \rangle} \mathd \nu_{\tmop{SG}}
  (\mathd \phi)\\
  & = & \lim_{\rho \rightarrow 1, T \rightarrow \infty} \log
  \frac{1}{Z_{\rho, T}}  \int e^{- \langle \psi, \phi \rangle} \mathd
  \nu^{\rho, T}_{\tmop{SG}} (\mathd \phi)\\
  & = & \lim_{\rho \rightarrow 1, T \rightarrow \infty} \log
  \frac{1}{Z_{\rho, T}}  \int e^{- \langle \psi, \phi \rangle} e^{- \lambda
  V^{\rho}_T (\phi)} \mathd \mu_T\\
  & = & \lim_{\rho \rightarrow 1, T \rightarrow \infty} \log
  \frac{1}{Z_{\rho, T}}  \int e^{- \langle \psi, \mathcal{C}_T \phi \rangle}
  e^{- \lambda V^{\rho}_T (\mathcal{C}_T \phi)} \mathd \mu\\
  & = & \lim_{\rho \rightarrow 1, T \rightarrow \infty} \log \left(
  e^{\langle \mathcal{C}_T \psi, (m^2 - \Delta)^{- 1} \mathcal{C}_T \psi
  \rangle} \frac{1}{Z_{\rho, T}} \int e^{- \lambda V^{\rho}_T (\phi + (m^2 -
  \Delta)^{- 1} \psi)} \mathd \mu_T \right)\\
  & = & \lim_{\rho \rightarrow 1} \lim_{T \rightarrow \infty} (\langle
  \mathcal{C}_T \psi, (m^2 - \Delta)^{- 1} \mathcal{C}_T \psi \rangle +
  V^{\rho}_{0, T} ((m^2 - \Delta)^{- 1} \psi) - V^{\rho}_{0, T} (0)) .
\end{eqnarray*}
Recall that since $\sup_{\varphi \in L^2} \| \nabla V^{\rho}_{0, T}
\|_{L^{\infty}} \leqslant C \lambda$ by Lemma \ref{lemma:value-func} we have
that for $\psi \in C^{\infty}_c$
\begin{eqnarray*}
  &  & \| \psi \|^2_{\mathcal{H}} - \langle \mathcal{C}_T \psi, (m^2 -
  \Delta)^{- 1} \mathcal{C}_T \psi \rangle\\
  & = & \log \int e^{- \langle \psi, \phi \rangle} \mathd \nu_{\tmop{SG}}
  (\mathd \phi) - (M, \psi)_{H^{- 1} (\langle x \rangle^{- n})} - \langle
  \mathcal{C}_T \psi, (m^2 - \Delta)^{- 1} \mathcal{C}_T \psi \rangle\\
  & \leqslant & \liminf_{\rho \rightarrow 1, T \rightarrow \infty} \log \int
  e^{- \langle \psi, \phi \rangle} \mathd \nu^{\rho, T}_{\tmop{SG}} (\mathd
  \phi) - (M, \psi)_{H^{- 1} (\langle x \rangle^{- n})} - \langle
  \mathcal{C}_T \psi, (m^2 - \Delta)^{- 1} \mathcal{C}_T \psi \rangle\\
  & \leqslant & \sup_{T < \infty, \rho \in C^{\infty}_c (\mathbb{R}^2, [0,
  1])} | V^{\rho}_{0, T} |_{1, \infty} \| (m^2 - \Delta)^{- 1} \psi \|_{L^1} -
  \| M \|_{H^1 (\langle x \rangle^{- n})} \| \psi \|_{H^1 (\langle x
  \rangle^n)}\\
  & < & \infty .
\end{eqnarray*}
So in particular $\mathcal{H}$ contains $C^{\infty}_c$ functions. We now show
that $\lim_{\rho \rightarrow 1} \lim_{T \rightarrow \infty} V^{\rho}_{0, T}
(\psi)$ is not a quadratic functional which will imply that
\[ \lim_{\rho \rightarrow 1} \lim_{T \rightarrow \infty} \langle
   \mathcal{C}_T \psi, (m^2 - \Delta)^{- 1} \mathcal{C}_T \psi \rangle +
   V^{\rho}_{0, T} (\psi) - V^{\rho}_{0, T} (0) \neq \| \psi
   \|^2_{\mathcal{H}} - (M, \psi)_{H^{- 1} (\langle x \rangle^{- n})} . \]
giving a contradiction. Observe that
\[ \nabla V^{\rho}_{0, T} (\psi) = \lambda \alpha (0) \sin (\psi) + \nabla
   R_{0, T} (\psi) \]
with $\sup_{\psi \in L^2} \| \nabla R_{0, T} (\psi) \|_{L^{\infty}} \leqslant
C \lambda^2$, by Lemma \ref{lemma:value-func}. Now for a quadratic functional
we would have that $\nabla V (\psi)$ is linear in $\psi$ so
\begin{equation}
  \lim_{T \rightarrow \infty, \rho \rightarrow 1} \nabla V^{\rho}_{0, T} (\psi
  + \varphi) + \nabla V^{\rho}_{0, T} (\psi - \varphi) - 2 \nabla V^{\rho}_{0,
  T} (\psi) = 0. \label{eq:contradiction}
\end{equation}
Let us choose $\psi, \varphi$ such that on $\varphi, \psi \in C^{\infty}_c$
and for $x \in B (0, 1)$ $\psi (x) = \pi / 2$ and $\varphi (x) = \pi / 4$.
Then for any $x \in B (0, 1)$ \
\[ \lambda \alpha (0) \sin (\varphi (x) + \psi (x)) + \lambda \alpha (0) \sin
   (\psi (x) - \varphi (x)) - 2 \lambda \alpha (0) \sin (\psi (x)) = \lambda
   \left( 2 \sqrt{2} / 2 - 2 \right) = \lambda \left( \sqrt{2} - 2 \right) \]
and since $\| \nabla R_{0, T} (\psi) \|_{L^{\infty}} \leqslant C \lambda^2$
this implies that for $\lambda$ sufficiently small and $x \in B (0, 1)$
\[ \lim_{\rho \rightarrow 1} \lim_{T \rightarrow \infty} \nabla V^{\rho}_{0,
   T} (\psi + \varphi) (x) + \nabla V^{\rho}_{0, T} (\psi - \varphi) (x) - 2
   \nabla V^{\rho}_{0, T} (\psi) (x) > \lambda \left( \sqrt{2} - 2 \right) /
   2. \]
This is clearly a contradiction to {\eqref{eq:contradiction}}.

\appendix\section{Wick ordered cosine }\label{app:cosine}\label{app:cosine}

We recall the definition of the regularized GFF as
\[ W_t = W_{0, t} = \int^t_0 Q_s \mathd X_s \]
where $X_s$ is a cylindrical Brownian motion on $L^2$. We can calculate:
\[ \mathbb{E} [W_t (x) W_t (y)] = K_t (x, y) . \]
Now it is not hard to see from Ito's formula that the quantity
\begin{equation}
  e^{\frac{\beta^2}{2} K_t (x, x)} \cos (\beta W_t (x)) \backassign \alpha (t)
  \cos (\beta W_t (x)) \label{eq:alpha}
\end{equation}
is a martingale. We will write
\begin{eqnarray*}
  \llbracket \cos (\beta W_t) \rrbracket (x) & = & \alpha (t) \cos (\beta W_t
  (x))\\
  \llbracket \sin (\beta W_t) \rrbracket (x) & = & \alpha (t) \sin (\beta W_t
  (x))\\
  \llbracket e^{i \beta W_t} \rrbracket (x) & = & \alpha (t) e^{i W_t (x)}
\end{eqnarray*}
We claim that is $\llbracket \cos (\beta W_t) \rrbracket$ bounded in $L^2
(\mathbb{P}, H^{- 1 + \delta} (\langle x \rangle^{- n}))$ uniformly in $t, g$.
Since it is also a martingale it converges almost surely. We will largely
follow {\cite{Junnila-2020}}. To prove this the following lemma will be
helpful:

\begin{lemma}\label{bound-quadratic-variation}
  Consider the martingale
  \[ M_t^{i, x} = K_i \ast \llbracket \cos (\beta W_t) \rrbracket (x) . \]
  Then the quadratic variation of $M^{i, x},$ denoted by $[M^{i, x}]$
  satisfies for any $\delta > 0$,
  \[ | \langle M^{i, x} \rangle_t | \leqslant C_{\delta} 2^{i \beta^2 / 2 \pi
     + \delta} \]
  where the constant $C_{\delta}$ is deterministic and does not depend on $x$
  and $t$.
\end{lemma}

\begin{proof}
  We have
  \[ K_i \ast \llbracket \cos (\beta W_t) \rrbracket (x) = \int K_i (x - z)
     \int^t_0 \llbracket \sin (\beta W_s) \rrbracket (z) \mathd W_s (z) \mathd
     z \]
  So
  \begin{eqnarray*}
    &  & | [K_i \ast \llbracket \cos (\beta W_t) \rrbracket (x)]_t |\\
    & = & \left| \int \int \int^t_0 K_i (x - z_1) K_i (x - z_2) \llbracket
    \sin (\beta W_s) \rrbracket (z_1) \llbracket \sin (\beta W_s) \rrbracket
    (z_2) \langle \mathd W_s (z_1), \mathd W_s (z_2) \rangle \mathd z_1 \mathd
    z_2 \right|\\
    & \leqslant & \int \int \int^t_0 | K_i (x - z_1) K_i (x - z_2) \llbracket
    \sin (\beta W_s) \rrbracket (z_1) \llbracket \sin (\beta W_s) \rrbracket
    (z_2) | Q_s (z_1, z_2) \mathd z_1 \mathd z_2\\
    & \leqslant & \int \int \int^t_0 | K_i (x - z_1) K_i (x - z_2) | \alpha^2
    (s) \frac{1}{2 \pi s} \exp (- 4 s | z_1 - z_2 |^2) \mathd z_1 \mathd z_2\\
    & \leqslant & \int \int \int^t_0 | K_i (x - z_1) K_i (x - z_2) |
    \frac{\langle s \rangle^{\beta^2 / 4 \pi}}{4 \pi s} \exp (- 4 s | z_1 -
    z_2 |^2) \mathd z_1 \mathd z_2\\
    & \leqslant & C \int \int | K_i (x - z_1) K_i (x - z_2) | \frac{1}{| z_1
    - z_2 |^{\beta^2 / 2 \pi}} \mathd z_1 \mathd z_2\\
    & = & C \int \int | K_i (x - z_1) K_i (x - z_2) | \frac{1}{| z_1 - z_2
    |^{\beta^2 / 2 \pi}} (1_{| z_1 - z_2 | \leqslant 1} + 1_{| z_1 - z_2 |
    \geqslant 1}) \mathd z_1 \mathd z_2\\
    & = & \Iota + \Iota \Iota
  \end{eqnarray*}
  To estimate term one we write
  \begin{eqnarray*}
    &  & \int | K_i (x - z_1) | \int | K_i (x - z_2) | \frac{1}{| z_1 - z_2
    |^{\beta^2 / 2 \pi}} 1_{| z_1 - z_2 | \leqslant 1} \mathd z_1 \mathd z_2\\
    & \leqslant & \int | K_i (x - z_1) | \mathd z_1 \left\| 1_{| z |
    \nobracket \leqslant 1 \nobracket} \frac{1}{| z |^{\beta^2 / 2 \pi}}
    \right\|_{L^p} \| K_i \|_{L^{p'}}\\
    & \leqslant & \| K_i \|_{L^{p'}} \| K_i \|_{L^1} \left\| 1_{| z |
    \nobracket \leqslant 1 \nobracket} \frac{1}{| z |^{\beta^2 / 2 \pi}}
    \right\|_{L^p}
  \end{eqnarray*}
  where we choose $p = 4 \pi / \beta^2 - \delta''$ for $\delta''$ sufficently
  small. This implies $\frac{1}{p'} = 1 - \beta^2 / 4 \pi - \delta'$ for some
  $\delta' $ which can be made arbitrarily small. Recall that
  \[ \| K_i \|_{L^1} \leqslant C \qquad \| K_i \|_{L^{\infty}} \leqslant C
     2^{2 i} \]
  So interpolating with the parameter $1 - \beta^2 / 4 \pi - \delta'$ we get
  $\| K_i \|_{L^{p'}} \leqslant \| K_i \|^{1 - \beta^2 / 4 \pi -
  \delta'}_{L^1} \| K_i \|^{\beta^2 / 4 \pi + \delta'}_{L^{\infty}}$which
  implies $\| K_i \|_{L^{p'}} \leqslant 2^{\left( \frac{\beta^2}{2 \pi} + 2
  \delta' \right) i}$. We have chosen $p$ in such a way that $\left\| 1_{| z
  \leqslant 1 |} \frac{1}{| z |^{\beta^2 / 2 \pi}} \right\|_{L^p} < \infty$.
  
  To estimate term $\Iota \Iota$ we simply write
  \begin{eqnarray*}
    &  & \int \int | K_i (x - z_1) K_i (x - z_2) | \frac{1}{| z_1 - z_2
    |^{\beta^2 / 2 \pi}} 1_{| z_1 - z_2 | \geqslant 1} \mathd z_1 \mathd z_2\\
    & \leqslant & \int \int | K_i (x - z_1) K_i (x - z_2) | \mathd z_1 \mathd
    z_2\\
    & = & \| K_i \|^2_{L^1} .
  \end{eqnarray*}
  in total we obtain that
  \[ | \langle M^{i, x} \rangle_t | \leqslant C 2^{i \beta^2 / 2 \pi + 2
     \delta'} . \]
\end{proof}

\begin{lemma}\label{bound-besov-cos}
  For any $p < \infty$ and $\delta > 0$ and $\rho$ such that $\int \rho \mathd
  x < \infty$
  \[ \sup_{t \geqslant 0} \mathbb{E} \left[ \| \llbracket \cos (\beta W_t)
     \rrbracket \|_{B_{p, p}^{- \beta^2 / 4 \pi - \delta} (\rho)}^p \right] <
     \infty . \]
\end{lemma}

\begin{proof}
  Using Burkholder's inequality we obtain
  \begin{eqnarray*}
    &  & \mathbb{E} \left[ \| \llbracket \cos (\beta W_t) \rrbracket
    \|_{B_{p, p}^{- \beta^2 / 4 \pi - \delta} (\rho)}^p \right]\\
    & \leqslant & \mathbb{E} \left[ \sum_{i \in \mathbb{N}} 2^{- p i (\beta^2
    / 4 \pi + \delta)} \| K_i \ast \llbracket \cos (\beta W_t) \rrbracket
    \|^p_{L^p (\rho)} \right]\\
    & = & \sum_{i \in \mathbb{N}} 2^{- p i (\beta^2 / 4 \pi + \delta)} \int
    \rho (x) \mathbb{E} [| K_i \ast \llbracket \cos (\beta W_t) \rrbracket (x)
    |^p] \mathd x\\
    & \leqslant & C \sum_{i \in \mathbb{N}} 2^{- p i (\beta^2 / 4 \pi +
    \delta)} \int \rho (x) \mathbb{E} [| \langle M^{i, x} \rangle_t |^{p / 2}]
    \mathd x\\
    & \leqslant & C \sum_{i \in \mathbb{N}} 2^{- p i (\beta^2 / 4 \pi +
    \delta)} 2^{i \beta^2 p / 4 \pi + \delta'}\\
    & < & \infty
  \end{eqnarray*}
  if $\delta > \delta'$.
\end{proof}

\begin{definition}
  \label{def:wick-order}Since $\llbracket \cos (\beta W_t) \rrbracket$ is a
  martingale and
  \[ \sup_t \mathbb{E} \left[ \| \llbracket \cos (\beta W_t) \rrbracket
     \|^p_{B_{p, p}^{- \beta^2 / 4 \pi - 2 \delta} (\langle x \rangle^{- n})}
     \right] < \infty \]
  it converges in $L^p (\mathbb{P}, B_{p, p}^{- \beta^2 / 4 \pi - 2 \delta}
  (\langle x \rangle^{- n}))$ to a limit. We will denote this limit by
  $\llbracket \cos (\beta W_{\infty}) \rrbracket$(and analogously for $\alpha
  (t) \sin (W_t)$ and $\alpha (t) e^{i W_t}$).
\end{definition}

\begin{remark}
  \label{rem:conv-cos-semi} From Lemma \ref{bound-quadratic-variation} we see that as $\beta
  \rightarrow 0$ $\mathbb{E} [\| \Delta_i (\llbracket \cos (\beta W_t)
  \rrbracket - 1) \|^2_{L^2 (\langle x \rangle^{- n})}] \rightarrow 0$.
  Together with Lemma \ref{bound-besov-cos} we can easily deduce from this that
  \[ \mathbb{E} \left[ \| (\llbracket \cos (\beta W_t) \rrbracket - 1)
     \|^2_{B_{p, p}^{- \beta^2 / 4 \pi - 3 \varepsilon} (\langle x \rangle^{-
     n + 1})} \right] \rightarrow 0, \quad \mathbb{E} \left[ \| (\llbracket
     \sin (\beta W_t) \rrbracket) \|^2_{B_{p, p}^{- \beta^2 / 4 \pi - 3
     \varepsilon} (\langle x \rangle^{- n + 1})} \right] \rightarrow 0. \]
\end{remark}

\section{Weighted estimates }\label{app:weighted}

\begin{definition}
  For a set $z \in \mathbb{R}^2$, $r \in \mathbb{R}$ we define the weighted
  $L^p$ spaces
  \[ \| f \|_{L_z^{p, r}} = \left( \int \exp (r p | x |) f^p (x) \mathd x
     \right)^{1 / p} \]
  And
  \[ \| f \|_{W_z^{1, p, r}} = \| f \|_{L^{p, r} (A)} + \left( \int (\exp (r
     p | x - z |)) (\nabla f (x))^p \mathd x \right)^{1 / p} \]
  We will also set $H^{1, r} \equallim W^{1, 2, r}$. Furthermore we will set
  \[ \| f \|_{L^{p, r}} = \| f \|_{L_0^{p, r}}, \qquad \| f \|_{W^{1, p, r}}
     = \| f \|_{W_0^{1, p, r}} . \]
\end{definition}

\begin{lemma}
  Let $r > 0$. Then for $f \in L_y^{2, r_1}, g \in L_z^{2,
  r_2}$\label{lemma:decay1}
  \[ \int f g \mathd x \leqslant \exp (- (r_1 \wedge r_2) | y - z |) \| f
     \|_{L_y^{2, r_1} (A)} \| g \|_{L_z^{2, r_2} (B)} . \]
\end{lemma}

\begin{proof}
  
  \begin{eqnarray*}
    &  & \int f g \mathd x\\
    & \leqslant & \int \exp (r_1 | x - y |) \exp (r_2  | x - z |) \exp (- r_1
    \wedge r_2  | y - z |) f (x) g (x) \mathd x\\
    & = & \exp (- r_1 \wedge r_2  | y - z |) \int \exp (r_1  | x - y |) f
    \exp (r_2  | x - z |) g \mathd x\\
    & \leqslant & \exp (- (r_1 \wedge r_2)  | y - z |) \| f \|_{L_y^{2, r_1}
    (A)} \| g \|_{L_z^{2, r_2} (B)}
  \end{eqnarray*}
  where we have used that by triangle inequality
  \[ r_1  | x - y | + r_2  | x - z | - r_1 \wedge r_2  | y - z | \geqslant 0.
  \]
  
\end{proof}

\begin{lemma}
  \label{lemma:estimate-decay-set}For any $\gamma > 0, n \leqslant 0$
  \[ \| f \|_{L^2 (\langle x \rangle^{- n})} \leqslant C \langle d (0, y)
     \rangle^{- n / 2} \| f \|_{L_y^{2, \gamma}} \]
\end{lemma}

\begin{proof}
  
  \begin{eqnarray*}
    &  & \int f^2 (x) \langle x \rangle^{- n} \mathd x\\
    & = & \int f^2 (x) e^{2 d (x, A)} e^{- 2 d | x - y |} \langle x
    \rangle^{- n} \mathd x\\
    & \leqslant & \int f^2 (x) e^{2 d (x, A)} \langle x - y \rangle^{- n}
    \langle x \rangle^{- n} \mathd x\\
    & \leqslant & C \langle d (0, A) \rangle^{- n} \int f^2 (x) e^{2 | x - y
    |} \mathd x
  \end{eqnarray*}
  
\end{proof}

\begin{lemma}
  Let $s \in \{ 0, 1 \}$ $r > 0$ and $f \in W_p^{s, r}$ is supported on $B (0,
  N)^c$, $N \geqslant 1$. Then
  \[ \| f \|_{W_p^{s, r - \kappa}} \leqslant N^{- \kappa} \| f \|_{W_p^{s,
     r}} \]
\end{lemma}

\begin{proof}
  
  \begin{eqnarray*}
    &  & \left( \int f^p \exp ((r - \kappa) p | x |) \mathd x \right)^{1 /
    p}\\
    & = & \left( \int_{| x | \geqslant N} f^p \exp ((r - \kappa) p | x |)
    \mathd x \right)^{1 / p}\\
    & \leqslant & N^{- \kappa} \left( \int f^p \exp (r p | x |) \mathd x
    \right)^{1 / p}\\
    & = & N^{- \kappa} \| f \|_{L^{p, r}}
  \end{eqnarray*}
  This proves the claim with $s = 0$. Applying this inequality also to $\nabla
  f$ we obtain the full statment.
\end{proof}

\begin{lemma}
  \label{bound:Q-linfty}{\tmdummy}
  
  \[ \| J_t f \|_{L^{\infty}} \leqslant t^{- 1} \| f \|_{L^{\infty}} \]
\end{lemma}

\begin{proof}
  This follows directly from Young's inequality. 
\end{proof}

\begin{lemma}
  Assume that $t / 2 \leqslant s \leqslant t$, or $0 \leqslant t \leqslant 1$
  then
  \[ \| I_{s, t} (u) \|_{L^{\infty}} \leqslant C \| u_{} \|_{L^{\infty} ([s,
     t] \times \mathbb{R}^2)} . \]
\end{lemma}

\begin{proof}
  
  \begin{eqnarray*}
    &  & \sup_x \left| \int^t_s \int_{\mathbb{R}^2} e^{- \frac{1}{2} m^2 / l}
    \frac{1}{\sqrt{4 \pi} l^{1 / 2}} e^{- 2 l | x - y |^2} u_l (y) \mathd l
    \mathd y \right|\\
    & \leqslant & \sup_x \int^t_s \int_{\mathbb{R}^2} e^{- \frac{1}{2} m^2 /
    l} \frac{1}{\sqrt{4 \pi} l^{1 / 2}} e^{- 2 l | x - y |^2} \mathd l \mathd
    y \| u \|_{L^{\infty} ([s, t] \times \mathbb{R}^2)}\\
    & \leqslant & \int^t_s e^{- \frac{1}{2} m^2 / l} l^{- 1} \mathd l \| u
    \|_{L^{\infty} ([s, t] \times \mathbb{R}^2)}
  \end{eqnarray*}
  Now in the case $t / 2 \leqslant s \leqslant t$
  \[ \int^t_s e^{- \frac{1}{2} m^2 / l} l^{- 1} \mathd l \| u \|_{L^{\infty}
     ([s, t] \times \mathbb{R}^2)} \leqslant \int^t_{t / 2} l^{- 1} \mathd l
     \| u \|_{L^{\infty} ([s, t] \times \mathbb{R}^2)} \leqslant \log 2 \| u
     \|_{L^{\infty} ([s, t] \times \mathbb{R}^2)} \]
  and in the case $0 \leqslant t \leqslant 1$
  \[ \int^t_s e^{- \frac{1}{2} m^2 / l} l^{- 1} \mathd l \| u \|_{L^{\infty}
     ([s, t] \times \mathbb{R}^2)} \leqslant \int^1_0 e^{- \frac{1}{2} m^2 /
     l} \mathd l \| u \|_{L^{\infty} ([s, t] \times \mathbb{R}^2)} \leqslant C
     \| u \|_{L^{\infty} ([s, t] \times \mathbb{R}^2)} . \]
  
\end{proof}

\begin{lemma}
  \label{lemma:boundWinfty}{\tmdummy}
  
  \[ \| I_{s, t} (u) \|_{W^{1, \infty}} \leqslant C \| \langle l \rangle^{1 /
     2 + \delta} u_l \|_{L_l^{\infty} ([s, t] \times \mathbb{R}^2)} \]
\end{lemma}

\begin{proof}
  We firs threat the case $s\geq m^{2}$. Then
  \begin{eqnarray*}
    &  & \sup_x \left| \int^t_s \int_{\mathbb{R}^2} \nabla_x e^{- \frac{1}{2}
    m^2 / l} \frac{1}{\sqrt{4 \pi}} e^{- 2 l | x - y |^2} u_l (y)
    \mathd l \mathd y \right|\\
    & = & \sup_x \left| \int^t_s \int_{\mathbb{R}^2} e^{- \frac{1}{2} m^2 /
    l} \frac{2 (x - y) l}{\sqrt{\pi}} e^{- 2 l | x - y |^2} u_l (y)
    \mathd l \mathd y \right|\\
    & \leqslant & \sup_x \left| \int^t_s \int_{\mathbb{R}^2} e^{- \frac{1}{2}
    m^2 / l} \frac{2 |x - y|  l^{1/2-\delta}}{\sqrt{\pi}} e^{- 2
    l | x - y |^2} \langle l \rangle^{1 / 2 + \delta} u_l (y) \mathd l \mathd
    y \right|\\
    & \leqslant & \int^t_s \int_{\mathbb{R}^2}
    e^{- \frac{1}{2} m^2 / l} \frac{2 |x - y| l^{1/2-\delta}}{\sqrt{\pi}} e^{- 2 l | x - y
    |^2} \mathd y \mathd l \| \langle l \rangle^{1 / 2 + \delta} u_l
    \|_{L_l^{\infty} ([s, t] \times \mathbb{R}^2)}\\
    & \leqslant & \left|
    \int_{\mathbb{R}^2}  \frac{2}{\sqrt{\pi}|x-y|^{2-\delta}}
    e^{- m | x - y |} \mathd y \mathd l \right| \| \langle l \rangle^{1 /
                  2 + \delta} u_l \|_{L_l^{\infty} ([s, t] \times \mathbb{R}^2)}\\
    & \leqslant & C \| \langle l \rangle^{1 / 2 + \delta} u_l
    \|_{L_l^{\infty} ([s, t] \times \mathbb{R}^2)}
  \end{eqnarray*}
  In the case $s \leq m^{2}$ We have $\exp(-m^{2}/l)\exp(-l|x-y|^{2})\lesssim \exp(-m|x-y|)$ so we have to estimate
  \begin{align*}
    &\int_{0}^{1}\int_{\mathbb{R}^{2}} \exp(-m|x-y|)|x-y| l u_{l}(y) \mathrm{d}l \mathrm{d}y \\
    \lesssim & \int_{\mathbb{R}^{2}} |x-y| \exp(-m|x-y|) \mathrm{d}y \|u\|_{L^{\infty}(\mathbb{R}_{+}\times \mathbb{R}^{2})}\\
    & \lesssim  \|u\|_{L^{\infty}(\mathbb{R}_{+} \times \mathbb{R}^{2})}
  \end{align*}
\end{proof}

\begin{lemma}
  Let $w (x) = \exp (- \gamma | x - z |)$ for $x, z \in \mathbb{R}^2$ and $|
  \gamma | < m - \kappa$. Then \ \label{lemma:boundIL2}
  \[ \| w I_{s, t} (u) \|_{L^2 (\mathbb{R}^2)} \leqslant C \langle s
     \rangle^{- 1 / 2} \| w u \|_{L^2 (\mathbb{R}_+ \times \mathbb{R}^2)} \]
  where the constant depends on $\kappa$.
\end{lemma}

\begin{proof}
  It is enough to prove the inequality for $s, t \leqslant 1$ and $s, t
  \geqslant 1$, then the general case will follow from $I_{s, t} (u) = I_{s,
  1} (u) + I_{1, t} (u)$. In the proof we will use several times that
  \[ e^{r | x - z |} e^{- | r | | x - y |} \leqslant e^{r | y - z |} \]
  For $s, t \geqslant 1$
  \begin{eqnarray*}
    &  & \int_{\mathbb{R}^2} \left| \int^t_s \int e^{2 r | x - z |} e^{-
    \frac{1}{2} m^2 / l} \frac{1}{\sqrt{4 \pi}} e^{- 2 l | x - y
    |^2} u_l (y) \mathd l \mathd y \right|^2 \mathd x\\
    & \leqslant & \int_{\mathbb{R}^2} \left( \int_{\mathbb{R}^2} e^{2 r | x -
    z |} \left( \int^t_s e^{- m^2 / l} \frac{1}{4 \pi} e^{- 4 l | x - y |^2}
    \mathd l \right)^{1 / 2} \left( \int^t_s u^2_l (y) \mathd l \right)^{1 /
    2} \mathd y \right)^2 \mathd x\\
    & \leqslant & C \int_{\mathbb{R}^2} \left( \int_{\mathbb{R}^2}
    e^{2 r | x - z |} \left( \frac{1}{| x - y |^2} e^{- 4 s | x - y |^2}
    \right)^{1 / 2} \left( \int^t_s u^2_l (y) \mathd l \right)^{1 / 2} \mathd
    y \right)^2 \mathd x\\
    & \leqslant & C \int_{\mathbb{R}^2} \left( \int_{\mathbb{R}^2}
    e^{2 r | x - z |} \frac{1}{| x - y |} e^{- 2 s | x - y |^2} \left(
    \int^t_s u^2_l (y) \mathd l \right)^{1 / 2} \mathd y \right)^2 \mathd x\\
    & \leqslant & C \int_{\mathbb{R}^2} \left( \int_{\mathbb{R}^2}
    \frac{1}{| x - y |} e^{- s | x - y |^2} \left( \int^t_s e^{2 r | y - z |}
    u^2_l (y) \mathd l \right)^{1 / 2} \mathd y \right)^2 \mathd x\\
    & \leqslant & C s^{- 1} \| w u \|^2_{L^2 (\mathbb{R}_+ \times
    \mathbb{R}^2)},^{}
  \end{eqnarray*}
  where in the last line we have used Young's inequality. We now treat the $s,
  t \leqslant 1$ case.
  \begin{eqnarray*}
    &  & \| w I_{s, t} (u) \|^2_{L^2}\\
    & \leqslant & C \int e^{2 r | x - z |} \int^t_s \left|
    \int_{\mathbb{R}^2} e^{- \frac{1}{2} m^2 / l} \frac{1}{\sqrt{4 \pi}} e^{- 2 l | x - y |^2} u_l (y) \mathd y \right|^2 \mathd x \mathd l
  \end{eqnarray*}
  Note that $e^{- \frac{1}{2} m^2 / l} \frac{1}{\sqrt{4 \pi}} e^{- 2
  l | x - y |^2} \leqslant C_{\kappa} e^{- (m - \kappa) | x - y |}$ so using
  Jensen's inequality
  \[ \begin{array}{ll}
       & \| w I_{s, t} (u) \|^2_{L^2}\\
       \leqslant & \int^t_s \int_{\mathbb{R}^2} \left| \int_{\mathbb{R}^2}
       e^{2 r | x - z |} e^{- \frac{1}{2} m^2 / l} \frac{1}{\sqrt{4 \pi}} e^{- 2 l | x - y |^2} u_l (y) \mathd y \right|^2 \mathd x \mathd
       l\\
       \leqslant & C \int^t_s \int_{\mathbb{R}^2} \left| \int_{\mathbb{R}^2}
       e^{- (m - \kappa / 2) | x - y |} e^{r | x - z |} u_l (y) \mathd y
       \right|^2 \mathd x \mathd l\\
       \leqslant & C_{\kappa} \int^t_s \int_{\mathbb{R}^2} \left|
       \int_{\mathbb{R}^2} e^{- (m - \kappa / 2 - r) | x - y |} e^{r | x - z
       |} u_l (y) \mathd y \right|^2 \mathd x \mathd l\\
       \leqslant & C_{\kappa} \int^t_s \| e^{r | y - z |} u_l (y) \|^2_{L^2}
       \mathd y \mathd l\\
       \leqslant & C_{\kappa} \| w u \|^2_{L^2},
     \end{array} \]
  as long as $m - r - \kappa \geqslant 0$ and we have used Young's inequality.
\end{proof}

In the case where we have no weight we can improve the preceeding estimate to
have constant $1$:

\begin{lemma}
  \label{bound-I-l2}{\tmdummy}
  
  \[ \| (m^2 - \Delta)^{1 / 2} I_{\infty} (u) \|^2_{L^2} \leqslant
     \int^{\infty}_0 \| u_s \|^2_{L^2} \mathd s \]
\end{lemma}

\begin{proof}
  
  \begin{eqnarray*}
    &  & \int_{\mathbb{R}^2} ((m^2 - \Delta)^{1 / 2} I_{0, \infty} (u))^2
    \mathd x\\
    & = & \int_{\mathbb{R}^2} (m^2 + | k |^2) \left( \CF I_{0, \infty} (u)
    (k) \right)^2 \mathd k\\
    & = & \int_{\mathbb{R}^2} (m^2 + | k |^2) \left( \int^{\infty}_0
    \frac{1}{t} e^{- (m^2 + | k |^2) / 2 t} \CF u_t (k) \mathd t \right)^2
    \mathd k\\
    & \leqslant & \int_{\mathbb{R}^2} (m^2 + | k |^2) \left( \int^{\infty}_0
    \frac{1}{t^2} e^{- (m^2 + | k |^2) / t} \mathd t \right) \int^{\infty}_0
    \left( \CF u_s (k) \right)^2 \mathd s \mathd k\\
    & = & \int_{\mathbb{R}^2} \int^{\infty}_0 \left( \CF u_s (k) \right)^2
    \mathd s \mathd k\\
    & = & \int^{\infty}_0 \| u_s \|^2_{L^2} \mathd s
  \end{eqnarray*}
\end{proof}

\begin{lemma}
  \label{IH1estimate}Let $w_y (x) = \exp (- \gamma | x - z |)$ for $x, z \in
  \mathbb{R}^2$ and $| \gamma | < m$
  \[ \| w I_{s, t} (u) \|_{H^1 (\mathbb{R}^2)} \leqslant C \| w u \|_{L^2
     (\mathbb{R}_+ \times \mathbb{R}^2)} \]
\end{lemma}

\begin{proof}
  Without loss of generality we may set $z = 0$. We first discuss the case $s,
  t \geqslant m$
  \begin{eqnarray*}
    &  & \int \left| \int^t_s \int_{\mathbb{R}^2} \exp (r | x |) e^{-
    \frac{1}{2} m^2 / l} \nabla_x \frac{1}{\sqrt{4 \pi}} e^{- 2 l | x - y |^2}
    u_l (y) \mathd l \mathd y \right|^2 \mathd x\\
    & \leqslant & \int_{\mathbb{R}^2} \left| \int^t_s \int_{\mathbb{R}^2}
    e^{- \frac{1}{2} m^2 / l} \nabla_x \frac{1}{\sqrt{4 \pi}} e^{- 2 l | x - y
    |^2} u_l (y) \exp (r | y |) \mathd l \mathd y \right|^2 \mathd x\\
    &  & + \int_{\mathbb{R}^2} \left| \int^t_s \int_{\mathbb{R}^2} e^{-
    \frac{1}{2} m^2 / l} (\exp (r | x |) - \exp (r | y |)) \nabla_x
    \frac{1}{\sqrt{4 \pi}} e^{- 2 l | x - y |^2} u_l (y) \mathd l \mathd y
    \right|^2 \mathd x\\
    & = & \Iota + \Iota \Iota
  \end{eqnarray*}
  We identify Term $\Iota$ as
  \[ \| \nabla_x I_{s, t} (w u) \|_{L^2} \]
  which is bounded by $\| w u \|_{L^2 (\mathbb{R}_+ \times \mathbb{R}^2)}$
  from the unweighted estimate. To estimate Term $\Iota \Iota$ we have
  \begin{eqnarray*}
    &  & \int_{\mathbb{R}^2} \left| \int^t_s \int_{\mathbb{R}^2} e^{-
    \frac{1}{2} m^2 / l} (\exp (r | x |) - \exp (r | y |)) \nabla_x
    \frac{1}{\sqrt{4 \pi}} e^{- 2 l | x - y |^2} u_l (y) \mathd l \mathd y
    \right|^2 \mathd x\\
    & = & \int_{\mathbb{R}^2} \left| \int^t_s \int_{\mathbb{R}^2} e^{-
    \frac{1}{2} m^2 / l} (\exp (r | x |) - \exp (r | y |)) \frac{2 l | x - y
    |}{\sqrt{\pi}} e^{- 2 l | x - y |^2} u_l (y) \mathd l \mathd y \right|^2
    \mathd x\\
    & \leqslant & \int_{\mathbb{R}^2} \left( \int_{\mathbb{R}^2} | (\exp (r |
    x |) - \exp (r | y |)) | \frac{1}{| x - y |^{2}} e^{- m | x - y |^2} \| u_l
    (y) \|_{L^2 (\mathbb{R}_+)} \mathd y \right)^2 \mathd x\\
    & = & \int_{\mathbb{R}^2} \left( \int_{\mathbb{R}^2} \frac{| (\exp (r | x
    |) - \exp (r | y |)) (\exp (- r | y |)) |}{| x - y |^{2}} \exp (- m | x - y
    |^2) \exp (r | y |) \| u_l (y) \|_{L^2 (\mathbb{R}_+)} \mathd y \right)^2
    \mathd x
  \end{eqnarray*}
  Now we claim that
  \[ \exp (- m / 2 | x - y |^2) \frac{1}{| x - y |} | (\exp (r | x |) - \exp
     (r | y |)) | \exp (- r | y |) \leqslant C. \]
  Indeed
  \begin{eqnarray*}
    &  & \exp (- m / 2 | x - y |^2) \frac{1}{| x - y |} | (\exp (r
    (\nobracket | x | - | y |) - 1) |\\
    & \leqslant & C \exp (- m / 2 | x - y |^2) \frac{| | x | - | y | |}{| x -
    y |} | (\exp (r (\nobracket | x | - | y |)) |\\
    & = & C \frac{| | x | - | y | |}{| x - y |} \exp (r (\nobracket
    \nobracket | x | - | y |) - m / 2 | x - y |^2)
  \end{eqnarray*}
  which is uniformly bouned by reverse triangle inequality. So in total our
  term is bounded by
  \begin{eqnarray*}
    &  & \int_{\mathbb{R}^2} \left( \int_{\mathbb{R}^2} \frac{ \exp (- m / 2 | x - y
    |^2)}{|x-y|} \exp (r | y |) \| u_l (y) \|_{L^2 (\mathbb{R}_+)} \mathd y \right)^2
    \mathd x\\
    & \leqslant & C \| \exp (r | y |) \| u_l (y) \|_{L^2 (\mathbb{R}_+)}
    \|_{L^2 (\mathbb{R}^2)}\\
    & = & C \| w u \|_{L^2 (\mathbb{R}_+ \times \mathbb{R}^2)}
  \end{eqnarray*}
  where we  were able to  Young's convolution inquality since $\exp(-m|x|)|x|^{-1}$ is in $L^{1}$. For $s, t \leqslant m$ we
  compute using $e^{- \frac{1}{2} m^2 / l} e^{- 2 l | x - y |^2} \leqslant
  e^{- m | x - y |}$.Then for any $\kappa > 0$ such that $m - r > \kappa$ we
  have \
  \begin{eqnarray*}
    &  & \| \nabla I_{s, t} (u) \|^2_{L^{2, z}}\\
    & = & \int_{\mathbb{R}^2} \exp (2 r | x |) \left| \int^t_s
    \int_{\mathbb{R}^2} e^{- \frac{1}{2} m^2 / l} \frac{2 l (x -
    y)}{\sqrt{\pi}} e^{- 2 l | x - y |^2} u_l (y) \mathd l \mathd y \right|^2
    \mathd x\\
    & \leqslant & C \int_{\mathbb{R}^2} \exp (2 r | x |) \int^t_s
    \int_{\mathbb{R}^2} (l | x - y | \exp (- m | x - y |) | u_l (y) |)^2
    \mathd l \mathd y \mathd x\\
    & \leqslant & C_{\kappa} \int \int^m_0 \int_{\mathbb{R}^2} (\exp (- (m -
    r - \kappa) | x - y |) \exp (r | y |) | u_l (y) |)^2 \mathd l \mathd y
    \mathd x\\
    & \leqslant & C \| w u \|^2_{L^2 (\mathbb{R}_+ \times \mathbb{R}^2)}
  \end{eqnarray*}
  In the case $s \leqslant m, t > m$ we write $I_{s, t} (u) = I_{s, m} (u) +
  I_{m, t} (u)$ and we can reduce the problem to the previous two cases.
\end{proof}

\section{Stochastic optimal control }\label{app:stoch-cont}

We consider the decomposition $\left(with L = (m^2 - \Delta)\right)$
\[ L^{- 1} = \int^{\infty}_0 J^2_t \mathd t \]
where
\[ J_t = \left( \frac{1}{t^2} e^{- L / t} \right)^{1 / 2} . \]
We denote by
\begin{equation}
  \mathcal{C}_t = \int^t_0 J^2_s \mathd s = L^{- 1} e^{- L / t},
  \label{eq:heatkernel}
\end{equation}
and by $K_t (x, y)$ the kernel of $\mathcal{C}_t$. From the definitions one
can see that
\[ K_t (x, y) = \int^t_0 e^{- m^2 / s} \left( \frac{1}{s^2} \frac{s}{4 \pi}
   e^{- 4 s | x - y |^2} \right) \mathd s = \int^t_0 e^{- m^2 / s^2} \left(
   \frac{1}{4 \pi s} e^{- 4 s | x - y |^2} \right) \mathd s \]
so
\[ K_t (x, x) = \int^t_0 e^{- m^2 / s^2} \left( \frac{1}{4 \pi s} \right)
   \mathd s = \mathbbm{1}_{t \geqslant 1} \frac{1}{4 \pi} \log t + C (t) \]
where $\sup_{t \in \mathbb{R}_+} C (t) < \infty$. Let $0 \leqslant s < t$ and
$u \in L^2 ([s, t], L^2 (\mathbb{R}^2)) .$ For later use we introduce the
notation
\[ I_{s, t} (u) = \int^t_s J_l u_l \mathd l. \]
We are interested in studying the quantities
\[ v_{t, T} (\varphi) = - \log \mathbb{E} [\exp (- V_T (\varphi + W_{t, T}))]
\]
where $W_{t, T} = \int^T_t Q_s \mathd X_s$,with $X$ being a cylindrical
Brownian motion on $L^2 (\mathbb{R}^2)$, and $Z_{t, T} = \exp (- v_{t, T})$,
for $\varphi \in L^2 (\mathbb{R}^2) .$

For the rest of this chapter we will denote by $C^n (L^2 (\mathbb{R}^2))$
functions $L^2 (\mathbb{R}^2) \rightarrow \mathbb{R}$ which are $n$ times
continuously Fr{\'e}chet differntiable with bounded derivatives. Next we can
derive a Hamilton-Jacobi-Bellmann equation for $v_{t, T}$, known in the
physics literature as the Polchinski equation.

\begin{proposition}
  \label{HJB-partitionfunction}Assume that $V_T \in C^2 (L^2 (\mathbb{R}^2))$.
  Then $v_{t, T}$ satisfies
  \[ \frac{\partial}{\partial t} v_{t, T} (\varphi) + \frac{1}{2} \tmop{Tr}
     (\dot{\mathcal{C}}_t \tmop{Hess} v_{t, T} (\varphi)) - \frac{1}{2} \| J_t
     \nabla v_{t, T} (\varphi) \|^2_{L^2 (\mathbb{R}^2)} = 0 \]
  \[ v_{T, T} (\varphi) = V_T (\varphi) . \]
  Furthermore if $V_T \in C^2 (L^2 (\mathbb{R}^2))$ then $v_{t, T} \in C ([0,
  T], C^2 (L^2 (\mathbb{R}^2))) \cap C^1 ([0, T], C (L^2 (\mathbb{R}^2)))$.
\end{proposition}

\begin{proof}
  Write $Z_{t, T} = \exp (- v_{t, T}) =\mathbb{E} [\exp (- V_T (\varphi +
  W_{t, T}))]$. Noting that $W_{t, T}$ has covariance $C_T - C_t$ it is not
  hard to see that
  \begin{eqnarray*}
    \frac{\partial}{\partial t} Z_{t, T} & = & \frac{\partial}{\partial t}
    \mathbb{E} [\exp (- V_T (\varphi + W_{t, T}))]\\
    & = & -\mathbb{E} [\langle W_{t, T}, (\mathcal{C}_T - \mathcal{C}_t)^{-
    2} \dot{C}_t W_{t, T} \rangle_{L^2 (\mathbb{R}^2)} \exp (- V_T (\varphi +
    W_{t, T}))] .
  \end{eqnarray*}
  Now using Gaussian integration by parts (see {\cite{bauerschmidt-2019-book}}
  Exercise 2.1.3)
  \begin{eqnarray*}
    &  & -\mathbb{E} [\langle W_{t, T}, (\mathcal{C}_T - \mathcal{C}_t)^{- 2}
    \dot{C}_t W_{t, T} \rangle_{L^2 (\mathbb{R}^2)} \exp (- V_T (\varphi +
    W_{t, T}))]\\
    & = & - \tmop{Tr} (\dot{\mathcal{C}}_t \tmop{Hess} Z_{t, T} (\varphi)) .
  \end{eqnarray*}
  Applying chain rule we get
  \begin{eqnarray*}
    \frac{\partial}{\partial t} v_{t, T} & = & - \frac{\partial}{\partial t}
    \log Z_{t, T}\\
    & = & - \frac{\frac{\partial}{\partial t} Z_{t, T}}{Z_{t, T}}\\
    & = & \frac{\tmop{Tr} (\dot{C}_t \tmop{Hess} Z_{t, T} (\varphi))}{Z_{t,
    T}}\\
    & = & e^{v_{t, T}} \text{Tr} (\dot{\mathcal{C}}_t \text{Hess} e^{- v_{t,
    T}})\\
    & = & - \text{Tr} (\dot{\mathcal{C}}_t \text{Hess} v_{t, T}) + \langle
    \nabla v_{t, T}, \dot{\mathcal{C}}_t \nabla v_{t, T} \rangle_{L^2
    (\mathbb{R}^2)}
  \end{eqnarray*}
  For the second statement differentiating under the expectation we
  obtain\tmcolor{red}{ \tmcolor{black}{$Z_{t, T} (\varphi) \in C^2 (L^2
  (\mathbb{R}^2))$}}, \ so using our first computation we can deduce from this
  that also $Z_{t, T} \in C^1 ([0, T], C (L^2 (\mathbb{R}^2)))$. Now observing
  that if $V_T \in C^2 (L^2 (\mathbb{R}^2))$ then $\inf_{t, \varphi} Z_{t, T}
  (\varphi) > 0$, and using chain rule we can conclude.
\end{proof}

\begin{definition}
  \label{def:controlproblem}Let $T > 0$, $H$ be a Hilbert space and \ $V_T : H
  \rightarrow \mathbb{R}$ measurable,bounded below. Let $X_t$ be a cylindrical
  process on some Hilbert space $\Xi$. Let $\Lambda$ be a Polish space and $u
  : [0, T] \rightarrow \Lambda$ be a process adapted to $X_t$. Let $Y_{s, t}
  (\varphi, u)$ be a solution to the equation
  \begin{equation}
    \mathd Y_{s, t} (u, \varphi) = \beta (t, Y_{s, t} (u, \varphi), u_t)
    \mathd t + \sigma (t, Y_{s, t} (u, \varphi), u_t) \mathd X_t
    \label{eq:control}
  \end{equation}
  \[ Y_s (u, \varphi) = \varphi . \]
  Where $\beta : [0, T] \times H \times \Lambda \rightarrow H$ and $\sigma :
  [0, T] \times H \times \Lambda \rightarrow \mathcal{L} (\Xi, H)$ are
  measurable. Then we say that $V_{t, T}$ is the value function on the
  stochastic control problem if
  \[ V_{t, T} (\varphi) = \inf_{u \in A ([s, T])} \mathbb{E} \left[ V_T
     (Y_{s, T} (u, \varphi)) + \int^T_s l_t (Y_{s, t}, u_t) \mathd t \right],
  \]
  with $l : [0, T] \times H \times \Lambda \rightarrow \mathbb{R}$ measurable,
  bounded below and we denote by $A ([s, t])$ the space of all processes $u :
  [s, t] \rightarrow \Lambda$ which are adapted to $X_t$. 
\end{definition}

\begin{proposition}[Dynamic programming]
  \label{prop:dynamicprogramming}$V_{t, T}$ as defined above satisfies for any
  $S < T$
  \[ V_{t, T} (\varphi) = \inf_{u \in A ([t, S])} \mathbb{E} \left[ V_{S, T}
     (Y_{t, S} (u, \varphi)) + \int^S_t l_s (Y_{t, s}, u_t) \mathd t \right] .
  \]
\end{proposition}

For a proof see {\cite{Fabbri_2017}} Theorem 2.24.

Now assume that $\sigma (t, Y_t, u_t)$ is self adjoint.We can associate a HJB
equation to the control problem from Definition \ref{def:controlproblem} . It
is:
\begin{equation}
  \frac{\partial}{\partial t} v (t, \varphi) + \frac{1}{2} \inf_{a \in
  \Lambda} [\tmop{Tr} (\sigma^2 (t, \varphi, a) \tmop{Hess} v (t, \varphi)) +
  \langle \nabla v, \beta (t, \varphi, a) \rangle_H + l (t, \varphi, a)] = 0.
  \label{eq:hjb-general}
\end{equation}
\[ v (T, \varphi) = V_T (\varphi) \]

We have the following theorem relating {\eqref{eq:hjb-general}} to the
solution of the control problem:

\begin{proposition}[Verification]
 \label{prop:verification-2} Assume that $v \in C ([0, T], C^{2})
  (H)) \cap C^{1} ([0, T], C (H))$ and $v$ solves
  \eqref{eq:hjb-general} with $v(T,\varphi) = V_T(\varphi)$. Furthermore assume that there exists $u \in A ([t, T])$ and $Y$
  such that $u, Y$ satisfy \eqref{eq:control} and
  \begin{equation}
    u_t \in \text{argmin}_{a \in \Lambda} [\text{Tr} (\sigma^2 (t, Y_t, u_t)
    \text{Hess} v (t, Y_t)) + \langle \nabla v (t, Y_t), \beta (t, Y_t, a)
    \rangle_H + l (t, Y_t, a)] . \label{eq:verification}
  \end{equation}
  Then $v (t, \varphi) = V_{t, T} (\varphi)$ and the pair $u, Y$ is optimal. 
\end{proposition}

For a proof see {\cite{Fabbri_2017}} Theorem 2.36. Now consider the case $H =
\Lambda = L^2 (\mathbb{R}^2)$ and
\begin{eqnarray*}
  \beta (t, \varphi, a) & = & J_t a\\
  \sigma (t, \varphi, a) & = & J_t\\
  l (t, Y_t, a) & = & \frac{1}{2} \| a \|^2_{L^2 (\mathbb{R}^2)} .
\end{eqnarray*}
Then {\eqref{eq:verification}} becomes a minimization problem for a quadratic
functional and reduces to
\[ u_t = - J_t \nabla v (t, Y_{s, t}) . \]
This means if we can solve the equation
\begin{equation}
  \mathd Y_{s, t} = - J^{2}_t \nabla v (t, Y_{s, t}) \mathd t + J_t \mathd X_t,
  \label{eq:minimizer}
\end{equation}
we can apply the verification theorem. Furthermore in this case
{\eqref{eq:hjb-general}} takes the form
\begin{equation}
  \frac{\partial}{\partial t} v (t, \varphi) + \frac{1}{2} \text{Tr}
  (\dot{\mathcal{C}}_t \text{Hess} v (t, \varphi)) - \frac{1}{2} \| J_t \nabla
  v (t, \varphi) \|^2_{L^2 (\mathbb{R}^2)} = 0, \label{eq:hjb-special}
\end{equation}
since
\begin{eqnarray*}
  &  & \inf_{a \in \Lambda} [\tmop{Tr} (\sigma (t, \varphi, a) \tmop{Hess} v
  (t, \varphi)) + \langle \nabla v, \beta (t, \varphi, a) \rangle_H + l (t,
  \varphi, a)]\\
  & = & \inf_{a \in \Lambda} \left[ \tmop{Tr} (J^2_t \tmop{Hess} v (t,
  \varphi)) + \langle \nabla v, J_t a \rangle_{L^2 (\mathbb{R}^2)} +
  \frac{1}{2} \| a \|^2_{L^2 (\mathbb{R}^2)} \right]\\
  & = & \frac{1}{2} \tmop{Tr} (\dot{\mathcal{C}}_t \tmop{Hess} v (t,
  \varphi)) - \frac{1}{2} \| J_t \nabla v (t, \varphi) \|^2_{L^2
  (\mathbb{R}^2)} .
\end{eqnarray*}
\begin{corollary}
  \label{cor:BD-formula}{\tmdummy}
  
  \[ - \log \mathbb{E} [e^{- V_T (\varphi + W_{t, T})}] = \inf_{u \in
     \mathbb{H}_a} \mathbb{E} \left[ V_T (Y_{s, T} (u, \varphi)) + \frac{1}{2}
     \int^T_s \| u_t \|^2_{L^2} \mathd t \right] \]
  where $\mathbb{H}_a$ is the space of processes adapted to $X_t$ such that
  $\mathbb{E} \left[ \int^{\infty}_0 \| u_t \|^2_{L^2} \mathd t \right]$ and
  $Y_t (u, \varphi)$ satisfies
  \[ \mathd Y_{s, t} (u, \varphi) = - J_t u_t \mathd t + J_t \mathd W_t \]
  \[ Y_{s, s} (u, \varphi) = \varphi . \]
  Note that $Y_{s, T} (u, \varphi) = \varphi + W_{t, T} + I_{t, T} (u)$.
  Furthermore the infimum on the r.h.s is attained \ 
\end{corollary}

\begin{proof}
  As already noted $v_{t, T} = - \log \mathbb{E} [e^{- V_T (\varphi + W_{t,
  T})}]$ satisfies the HJB equation {\eqref{eq:hjb-special}} and is in $C ([0,
  T], C^2 (L^2 (\mathbb{R}^2)))$, so $\nabla v_{t, T}$ is Lipschitz continuous
  uniformly in $T$ and bounded. By a standard fix-point argument we can then
  solve \eqref{eq:minimizer}, and so applying the
  verification theorem we obtain
  \[ - \log \mathbb{E} [e^{- V_T (\varphi + W_{t, T})}] = \inf_{u \in A ([s,
     T])} \mathbb{E} \left[ V_T (Y_{s, T} (u, \varphi)) + \frac{1}{2} \int^T_s
     \| u_t \|^2_{L^2} \mathd t \right] . \]
  Since $V_T$ is bounded below we can clearly restrict the infimum on the
  right hand side to $u \in \mathbb{H}_a$.
\end{proof}

\bibliographystyle{plain} \bibliography{qft-drift}

\end{document}